\documentclass[a4paper]{amsart}

\usepackage{amssymb}
\usepackage[english]{babel}
\usepackage[T1]{fontenc}
\usepackage{mathtools}
\usepackage{stmaryrd}
\usepackage{tikz-cd}

\binoppenalty\maxdimen
\relpenalty\maxdimen

\address{Inria \& LIX, CNRS, École polytechnique, Institut Polytechnique de Paris, 91120 Palaiseau, France}
\author{Rubén Muñoz-\relax-Bertrand}
\email{ruben.munoz-bertrand@inria.fr}
\title{Local structure of the overconvergent de Rham--Witt complex}

\newcommand{\lef}{\mathopen{}\left}
\newcommand{\rig}{\mathclose{}\right}

\numberwithin{equation}{section}

\newtheorem{coro}[equation]{Corollary}
\newtheorem{deff}[equation]{Definition}
\newtheorem{lemm}[equation]{Lemma}
\newtheorem{prop}[equation]{Proposition}
\newtheorem{rema}[equation]{Remark}
\newtheorem{thrm}[equation]{Theorem}
\newtheorem*{thrm*}{Theorem}
\newtheorem{xmpl}[equation]{Example}

\begin{document}

\begin{abstract}
We give a general description of the structure of the relative de Rham--Witt complex on a polynomial ring, seen as an algebra over its integral part. After giving a control of the overconvergence of Lazard's morphism, we similarly give the structure of the overconvergent complex, locally on a smooth algebraic variety over a perfect field of positive characteristic.
\end{abstract}

\maketitle

\section*{Introduction}

In $p$-adic cohomology, two major theories have been introduced by Berthelot: crystalline and rigid cohomology. The former is an integral cohomology theory, with nice properties on a proper and smooth variety over a perfect field of positive characteristic. The latter is well-behaved on varieties that need not be proper, but at the cost of becoming a rational cohomology theory.

Another approach from Deligne is the de Rham--Witt complex, which was first studied by Illusie in \cite{complexedederhamwittetcohomologiecristalline}. In this article, he demonstrates how this complex computes the crystalline cohomology on a proper and smooth variety over a perfect field of positive characteristic. A part of the proof is devoted to the study of the structure of the complex in simple situations, such as the polynomial ring case. He often reduces to this explicit situation to demonstrate his results.

Thirty years later, Davis, Langer and Zink defined a subcomplex of the de Rham--Witt complex, and proved how it computes the rigid cohomology of a quasi-projective smooth variety over a perfect field of positive characteristic \cite{overconvergentderhamwittcohomology}.

Their approach is inspired by the work of Monsky and Washnitzer, who also had introduced a $p$-adic cohomology theory for affine varieties in positive characteristic \cite{formalcohomologyi}. This theory is now known to coincide with the rigid cohomology of a smooth affine variety over a perfect field of positive characteristic. The construction of this cohomology uses a notion of overconvergence: they remove all series which do not converge quickly enough for the $p$-adic topology because they cannot be integrated, so that one gets finite dimensional cohomology groups.

Similarly, Davis, Langer and Zink define a notion of overconvergence on the de Rham--Witt complex, and only keep the differentials which converge quickly enough for a specific topology on it. Thus, they called this subcomplex the overconvergent de Rham--Witt complex. For the proof of the comparison theorem with rigid cohomology, it is also essential to know the structure of the overconvergent de Rham--Witt complex.

Both crystalline and rigid cohomology have categories of coefficients. For crystalline cohomology, a result of Bloch gives an equivalence of categories between a subcategory of crystals, and a category of de Rham--Witt connections \cite{crystalsandderhamwittconnections}. Again, the proof of the comparison theorem needs a good understanding of the structure of the de Rham--Witt complex.

A preprint of Ertl gave reasons to believe that a similar result should hold for rigid cohomology and the overconvergent de Rham--Witt complex \cite{comparisonbetweenrigidandoverconvergentcohomologywithcoefficients}. More precisely, it appears that the category of overconvergent $F$-isocrystals should be equivalent to a category of overconvergent de Rham--Witt connections.

It turns out that the strategy of Bloch can be adapted to this context. However, many work is needed to ensure that the overconvergence is kept under control during this process, which includes taking the limit of a sequence given by a polynomial recurrence relation. In other terms, it is far from obvious that one stays within the overconvergent subcomplex of the de Rham--Witt complex after taking such a limit.

A first step in that direction was the previous article of the author \cite{pseudovaluationsonthederhamwittcomplex}, which gives a new definition of overconvergence on the de Rham--Witt complex which is well-behaved with addition and multiplication. This definition is given by maps on the complex which have the property of being a pseudovaluation. We will use this property many times, so we will recall all these results later on.

The second step is the current article. In order to prove the equivalence of categories of coefficients, which will appear in a follow-up paper, it is paramount to understand deeply the structure of the overconvergent de Rham--Witt complex.

However, the structure theorems that are currently known, for instance \cite[corollaire I. 2.15]{complexedederhamwittetcohomologiecristalline} or \cite[theorem 2.8]{derhamwittcohomologyforaproperandsmoothmorphism}, are too poor for our needs. Most of them give an explicit structure in the case polynomial rings, or modulo $p$. Moreover, the de Rham--Witt complex is studied as a module over the Witt vectors of the base field. In practice, this does not give enough information to do $p$-adic analysis.

Nonetheless, they still give valuable input on the complex. For instance, one of the main consequences of the currently known results is that the de Rham--Witt complex has a direct sum decomposition as two subcomplexes. The first one is sometimes called the integral part, and in the affine and overconvergent case it corresponds to the Monsky--Washnitzer complex; that is, the de Rham complex associated to a weakly complete finitely generated lift of the ring. The second subcomplex is acyclic, and is sometimes called the fractional part.

In order to have a more insightful knowledge of the algebraic behaviour of the de Rham--Witt complex, it is useful to understand how these two subcomplexes interact with each other. This is the prime motivation for the two main results of this article: theorem \ref{structuretheoremconvergent} which gives the structure of the de Rham--Witt complex, and theorem \ref{maintheorem} for the overconvergent de Rham--Witt complex. Both give a structure as a module over the zeroth degree of the integral part. Moreover, the overconvergent version of the result describes how overconvergence is encoded in that setting. Finally, we state the structure theorems not only for polynomial rings, but also locally on any smooth variety over a perfect field of positive characteristic, amongst other settings.\vfil

Let us now turn to the content of this article. We first recall the main properties of de Rham--Witt complex, and give pointers in the literature for an introduction. We also recall theorem \ref{structuretheorem}, which is one of the main structure theorems for the de Rham--Witt complex which are currently known.

We then study relatively perfect morphisms. They play an important role here because their definition behaves nicely with Witt vectors. We explain this relationship in the second section.

In the third section, we give the structure of the de Rham--Witt complex modulo $p$. We do so by first studying the de Rham complex of a polynomial ring. Then, we introduce the concept of rebar cages (which is a rough translation of its original French name ``\textit{armature}''). They allow us to study a wide variety of rings for which we can explicitly describe the structure of the de Rham--Witt complex in characteristic $p$.

This work allows us to give the structure of the de Rham--Witt complex in the fourth section. One of the main results, theorem \ref{structuretheoremconvergent}, gives under mild assumptions the local structure of the de Rham--Witt complex, as an algebra over the zeroth degree of its integral part, of an algebraic variety $X$ over a reduced ring $k$ of characteristic $p$.

\begin{thrm*}
There exists a covering of the variety $X$ by affines $\operatorname{Spec}\lef(R\rig)$, having a smooth lift $\operatorname{Spec}\lef(L\rig)$ over $W\lef(k\rig)$, such that for each of these affines there are subsets $H,H',G,G'\subset W\Omega_{R/k}$ of the de Rham--Witt complex of $R$ over $k$, such that for each $x\in W\Omega_{R/k}$ there is a unique map:
\begin{equation*}
l\colon\begin{array}{rl}H\sqcup H'\sqcup G\sqcup G'\to&\widehat{L}\\e\mapsto&l_{e}\end{array}\text{,}
\end{equation*}
such that for a morphism of differential graded algebras $t_{F}\colon\Omega^{\mathrm{sep}}_{\widehat{L}/W\lef(k\rig)}\to W\Omega_{R/k}$, depending on a choice of a Frobenius lift $F$ to $\widehat{L}$, and such that every $e\in H$ satisfies $e=t_{F}\lef(\widetilde{e}\rig)$ for some $\widetilde{e}\in\Omega^{\mathrm{sep}}_{\widehat{L}/W\lef(k\rig)}$, we have a unique writing:
\begin{equation*}
x=\sum_{e\in H}t_{F}\lef(l_{e}\widetilde{e}\rig)+\sum_{e\in H'}t_{F}\lef(l_{e}\rig)e+\sum_{e\in G}t_{F}\lef(l_{e}\rig)e+d\lef(\sum_{e\in G'}t_{F}\lef(l_{e}\rig)e\rig)\text{.}
\end{equation*}
\end{thrm*}

Moreover, the subsets $H$, $H'$, $G$ and $G'$ are given explicitly in concrete examples. With this writing, one can clearly see the integral part indexed by $H$ and $H'$, and the fractional acyclic part indexed by both $G$ and $G'$. Also, in some cases such as the perfect base one, we have $H'=\emptyset$, so that the integral part corresponds to the de Rham complex of a $p$-adically complete lift.

But the other results in this fourth section are more general, and are stated for a wider class of varieties than their overconvergent counterparts; some even in zero characteristic. Even though this requires more work, the author hopes that this will have various applications. For instance, the de Rham--Witt complex has been studied in the context of $p$-adic Hodge theory \cite{integral}. Again, in that paper, an explicit structure theorem was needed; we give here a new point of view on that structure. Also, if we want an overconvergent de Rham--Witt complex over Laurent series fields to compute Lazda--Pál's rigid cohomology \cite{rigidcohomologyoverlaurentseriesfields}, we would need a new definition of overconvergence which can only be made once the structure of the complex is well-known. Finally, but more speculatively, there might be a way to link this work with Scholze's tempered cohomology \cite{thetempereddiskandthetemperedcohomology}. We hope that the generality of this section will ease much future work in that direction.

Once that these results are stated, the remainder of the article is devoted to the verification that overconvergence is compatible with these structures. In the fifth section, we recall the definition of the overconvergent de Rham Witt complex and state a few facts.

Then, we return to relatively perfect algebras. We focus this time on the ones which are finite free as modules, and see how this notion not only behaves nicely with Witt vectors, but also with overconvergence.

In the antepenultimate section we study the Lazard morphism. This morphism gives a way to embed the de Rham complex of Monsky--Washnitzer cohomology into the overconvergent de Rham--Witt complex. We will study how overconvergence is preserved when we apply this morphism.

The results of the two previous sections are then used to give the structure of overconvergent Witt vectors; that is, we state the main structure theorem in degree zero.

Finally, in the last section we prove the main theorem of the article: the structure of the overconvergent de Rham--Witt complex as an algebra over the zeroth degree of its integral part. The statement of theorem \ref{maintheorem} is very similar to theorem \ref{structuretheoremconvergent}, except that all the sets we consider are overconvergent.

\begin{thrm*}
Let $X$ be an algebraic variety over a perfect Noetherian ring $k$. There exists a covering of the variety $X$ by affines $\operatorname{Spec}\lef(\overline{R}\rig)$, having a smooth lift $\operatorname{Spec}\lef(R\rig)$ over $W\lef(k\rig)$, such that for each of these affines there are subsets $H,G,G'\subset W^{\dagger}\Omega_{\overline{R}/k}$ of the de Rham--Witt complex of $\overline{R}$ over $k$, such that for each $x\in W^{\dagger}\Omega_{\overline{R}/k}$ there is a unique map:
\begin{equation*}
l\colon\begin{array}{rl}H\sqcup G\sqcup G'\to&R^{\dagger}\\e\mapsto&l_{e}\end{array}\text{,}
\end{equation*}
with image in the weak completion of $R$, such that for a morphism of differential graded algebras $t_{F}\colon\Omega^{\mathrm{sep}}_{R^{\dagger}/W\lef(k\rig)}\to W^{\dagger}\Omega_{\overline{R}/k}$, depending on a choice of a Frobenius lift $F$ to $R^{\dagger}$, and such that every $e\in H$ satisfies $e=t_{F}\lef(\widetilde{e}\rig)$ for some $\widetilde{e}\in\Omega^{\mathrm{sep}}_{R^{\dagger}/W\lef(k\rig)}$, we have a unique writing:
\begin{equation*}
x=\sum_{e\in H}t_{F}\lef(l_{e}\widetilde{e}\rig)+\sum_{e\in G}t_{F}\lef(l_{e}\rig)e+d\lef(\sum_{e\in G'}t_{F}\lef(l_{e}\rig)e\rig)\text{.}
\end{equation*}
\end{thrm*}

Again, the sets $H$, $G$ and $G'$ are explicit. But also, this theorem comes with formulae which describe the overconvergence pace of the value of the function $l$. In practice, this means that we can now use $p$-adic analysis tools locally on the overconvergent de Rham--Witt complex, which we shall use in the follow-up paper.

As a proof of concept of the strength of this result, we give a new and much shorter proof of the fact that the overconvergent de Rham--Witt complex is a sheaf for the étale topology.

\section*{Acknowledgments}

This work generalises results in my PhD thesis, so I take this opportunity to thank my advisor Daniel Caro. I am also grateful to all the members of my jury Andreas Langer, Tobias Schmidt, Christine Huyghe, Andrea Pulita and Jérôme Poineau. I also want to salute my colleagues in Toulouse for their invaluable support.

This work has been financed by ANR-21-CE39-0009-BARRACUDA.
It was also partially written during my position at Laboratoire de Mathématiques de Versailles, UVSQ, CNRS, Université Paris-Saclay, 45 avenue des États-Unis, 78035 Versailles Cedex, France.
The author was also under ``Contrat EDGAR-CNRS no 277952 UMR 6623, financé par la région Bourgogne-Franche-Comté.'' at Université Marie et Louis Pasteur, CNRS, LmB (UMR 6623), 25000 Besançon, France during the revision of this article.
This work was supported by the HYPERFORM consortium, funded by France through Bpifrance.

Further acknowledgments to be added after the referee process.

\section*{Notations}

In this article, $p$ denotes a prime number.

Let $n\in\mathbb{N}$. Write $k\lef[\underline{X}\rig]\coloneqq k\lef[X_{1},\ldots,X_{n}\rig]$ and $W\lef(k\rig)\lef[\underline{X}\rig]\coloneqq W\lef(k\rig)\lef[X_{1},\ldots,X_{n}\rig]$ for the polynomial ring over the ring of Witt vectors over $k$.

For any $J\subset\lef\llbracket1,n\rig\rrbracket$ and any $a\colon J\to\mathbb{N}$, we will write $\underline{X}^{a}\coloneqq\prod_{j\in J}{X_{j}}^{a\lef(j\rig)}$.

Throughout this article, for any group $G$, or any related algebraic structure, we will put $\overline{G}\coloneqq G/pG$; similarly, for any $g\in G$ its projection in $\overline{G}$ shall be denoted by $\overline{g}$, and we will use the same convention for maps and homomorphisms.

The function $\zeta_{\varepsilon}$ is introduced in \eqref{zetaepsilon}.

The Lazard morphism $t_{F}$ is introduced in definition \ref{lazardmorphism}.

\section{The de Rham--Witt complex for a polynomial ring}

This introductory section contains no new results. We shall define the main notations which are going to be used throughout this work. For a reference on the de Rham--Witt complex, please refer to \cite{derhamwittcohomologyforaproperandsmoothmorphism}.

Let $k$ be a commutative $\mathbb{Z}_{\lef\langle p\rig\rangle}$-algebra, and let $R$ be a commutative $k$-algebra.

We will denote by $W\Omega_{R/k}$ the de Rham--Witt complex of $R$, and by $d$ its differential. Recall that, in degree zero, $W\Omega^{0}_{R/k}$ is isomorphic as a $W\lef(k\rig)$-algebra to the ring of Witt vectors $W\lef(R\rig)$. Furthermore, there is a morphism of graded rings $F\colon W\Omega_{R/k}\to W\Omega_{R/k}$, a morphism of graded groups $V\colon W\Omega_{R/k}\to W\Omega_{R/k}$, as well as a morphism of multiplicative monoids $\lef[\bullet\rig]\colon R\to W\lef(R\rig)$ such that:
\begin{align}
\forall m\in\mathbb{N},\ &\forall r\in R,\ F^{m}\lef(\lef[r\rig]\rig)=\lef[r^{p^{m}}\rig]\text{,}\label{frobeniuswittvectors}\\
\forall m\in\mathbb{N},\ &\forall r\in R,\ F^{m}\lef(d\lef(\lef[r\rig]\rig)\rig)=\lef[r^{p^{m}-1}\rig]d\lef(\lef[r\rig]\rig)\text{,}\label{fdttdt}\\
\forall m\in\mathbb{N}^{*},\ &\forall x\in W\Omega_{R/k},\ F^{m}\lef(d\lef(V^{m}\lef(x\rig)\rig)\rig)=d\lef(x\rig)\text{,}\label{fdvisd}\\
\forall m\in\mathbb{N},\ &\forall x,y\in W\Omega_{R/k},\ V^{m}\lef(xF^{m}\lef(y\rig)\rig)=V^{m}\lef(x\rig)y\text{,}\label{vxfyvxy}\\
\forall m\in\mathbb{N},\ &\forall x\in W\Omega_{R/k},\ d\lef(F^{m}\lef(x\rig)\rig)=p^{m}F^{m}\lef(d\lef(x\rig)\rig)\text{,}\label{dfpfd}\\
\forall m\in\mathbb{N},\ \forall\lef(x_{i}\rig)_{i\in\lef\llbracket1,m\rig\rrbracket}\in&\lef(W\lef(k\lef[\underline{X}\rig]\rig)\rig)^{m},\ d\lef(\prod_{i=1}^{m}x_{i}\rig)=\sum_{i=1}^{m}\lef(\prod_{j\in\lef\llbracket1,m\rig\rrbracket\smallsetminus\lef\{i\rig\}}x_{j}\rig)d\lef(x_{i}\rig)\text{.}\label{dproducts}
\end{align}

For all $m\in\mathbb{N}$, we will denote by $W_{m}\lef(k\rig)$ the ring of truncated Witt vectors over $k$. When $m=0$, it is the zero ring. There also exists a truncated de Rham--Witt complex $W_{m}\Omega_{R/k}$. It is an alternating $W_{m}\lef(k\rig)$-dga, which is short for differential graded $W_{m}\lef(k\rig)$-algebra, such that there exists a surjective morphism of $W_{m+1}\lef(k\rig)$-dgas $W_{m+1}\Omega_{R/k}\to W_{m}\Omega_{R/k}$ such that:
\begin{equation*}
W\Omega_{R/k}\cong \varprojlim_{m\in\mathbb{N}^{*}} W_{m}\Omega_{R/k}\text{.}
\end{equation*}

We can therefore define the ideal of the surjective morphism:
\begin{equation*}
\operatorname{Fil}^{m}\lef(W\Omega_{R/k}\rig)\coloneqq\operatorname{Ker}\lef(W\Omega_{R/k}\to W_{m}\Omega_{R/k}\rig)\text{.}
\end{equation*}

\begin{deff}
The filtration $\lef(\operatorname{Fil}^{m}\lef(W\Omega_{R/k}\rig)\rig)_{m\in\mathbb{N}}$ of two-sided ideals of $W\Omega_{R/k}$ is called the \textbf{canonical filtration}. The associated topology, which is separated and complete, is called the \textbf{canonical topology}.
\end{deff}

In this section, we will introduce basic elements of the de Rham--Witt complex, and we will recall how any element in the de Rham--Witt complex over $k\lef[\underline{X}\rig]$ can be expressed as a series using these elements. We chiefly follow \cite{derhamwittcohomologyforaproperandsmoothmorphism}.

\begin{deff}
A \textbf{weight function} is a map $a\colon\lef\llbracket1,n\rig\rrbracket\to\mathbb{N}\lef[\frac{1}{p}\rig]$; for any $i\in\lef\llbracket1,n\rig\rrbracket$, its values will be written $a_{i}$. We define:
\begin{equation*}
\lef\lvert a\rig\rvert\coloneqq\sum_{i=1}^{n}a_{i}\text{.}
\end{equation*}
\end{deff}

Let $a$ be a weight function. Let $J\subset\lef\llbracket1,n\rig\rrbracket$. In this paper, we will denote by $a|_{J}$ the weight function which for any $i\in\lef\llbracket1,n\rig\rrbracket$ satisfies:
\begin{equation*}
a|_{J}\lef(i\rig)=\begin{cases}a_{i}&\text{ if }i\in J\text{,}\\0&\text{ otherwise.}\end{cases}
\end{equation*}

\begin{deff}
Let $a$ be a weight function. Its \textbf{support} is the set:
\begin{equation*}
\operatorname{Supp}\lef(a\rig)\coloneqq\lef\{i\in\lef\llbracket1,n\rig\rrbracket\mid a_{i}\neq0\rig\}\text{.}
\end{equation*}
\end{deff}

\begin{deff}
A \textbf{partition} of the weight function $a$ is a subset of $\operatorname{Supp}\lef(a\rig)$. Its \textbf{size} is its cardinal. We denote by $\mathcal{P}$ the set of all $\lef(a,I\rig)$, where $a$ is a weight function and $I$ is a partition of $a$. Throughout this paper, the $p$-adic valuation will be denoted by $\operatorname{v}_{p}$. We fix the total order $\preceq$ on $\operatorname{Supp}\lef(a\rig)$:
\begin{multline}\label{ordersupp}
\forall i,i'\in\operatorname{Supp}\lef(a\rig),\ i\preceq i'\\\iff\lef(\lef(\operatorname{v}_{p}\lef(a_{i}\rig)\leqslant\operatorname{v}_{p}\lef(a_{i'}\rig)\rig)\wedge\lef(\lef(\operatorname{v}_{p}\lef(a_{i}\rig)=\operatorname{v}_{p}\lef(a_{i'}\rig)\rig)\implies\lef(i\leqslant i'\rig)\rig)\rig)\text{.}
\end{multline}
\end{deff}

This order depends on the weight function $a$, but in practice no confusion will arise. We denote by $\prec$ the associated strict total order, and we will denote by $\min\lef(a\rig)\in\operatorname{Supp}\lef(a\rig)$ the only element such that $\min\lef(a\rig)\preceq i$ for any $i\in\operatorname{Supp}\lef(a\rig)$.

Let $I\coloneqq\lef\{i_{j}\rig\}_{j\in\lef\llbracket1,m\rig\rrbracket}$ be a partition of $a$. By convention, we will always assume that $i_{j}\prec i_{j'}$ for all $j,j'\in\lef\llbracket1,m\rig\rrbracket$ such that $j<j'$. We will also say that $i_{0}\preceq i$ and $i\prec i_{m+1}$ for any $i\in\operatorname{Supp}\lef(a\rig)$. For any $l\in\lef\llbracket0,m\rig\rrbracket$, we define the following subsets of $\operatorname{Supp}\lef(a\rig)$:
\begin{equation*}
I_{l}\coloneqq\lef\{i\in\operatorname{Supp}\lef(a\rig)\mid i_{l}\preceq i\prec i_{l+1}\rig\}\text{.}
\end{equation*}

We set:
\begin{align*}
\operatorname{v}_{p}\lef(a\rig)&\coloneqq\min\lef\{\operatorname{v}_{p}\lef(a_{i}\rig)\mid i\in\lef\llbracket1,n\rig\rrbracket\rig\}\text{,}\\
u\lef(a\rig)&\coloneqq\max\lef\{0,-\operatorname{v}_{p}\lef(a\rig)\rig\}\text{.}
\end{align*}

If $a$ is not the zero function, we put:
\begin{equation*}
g\lef(a\rig)\coloneqq F^{u\lef(a\rig)+\operatorname{v}_{p}\lef(a\rig)}\lef(d\lef(V^{u\lef(a\rig)}\lef(\lef[\underline{X}^{p^{-\operatorname{v}_{p}\lef(a\rig)}a}\rig]\rig)\rig)\rig)\text{.}
\end{equation*}

Furthermore, if $I$ is a partition of $a$, and for any $\eta\in W\lef(k\lef[\underline{X}\rig]\rig)$, we set:
\begin{equation*}
e\lef(\eta,a,I\rig)\coloneqq\begin{cases}V^{u\lef(a\rig)}\lef(\eta\lef[\underline{X}^{p^{u\lef(a\rig)}a|_{I_{0}}}\rig]\rig)\times\prod_{l=1}^{\#I}g\lef(a|_{I_{l}}\rig)&\text{if }I_{0}\neq\emptyset\text{ or }u\lef(a\rig)=0\text{,}\\d\lef(V^{u\lef(a\rig)}\lef(\eta\lef[\underline{X}^{p^{u\lef(a\rig)}a|_{I_{1}}}\rig]\rig)\rig)\times\prod_{l=2}^{\#I}g\lef(a|_{I_{l}}\rig)&\text{otherwise.}\end{cases}
\end{equation*}

\begin{prop}\label{zpcombination}
Let $\lef(a,I\rig),\lef(b,J\rig)\in\mathcal{P}$ such that $u\lef(a\rig)\geqslant u\lef(b\rig)$. Assume that either $u\lef(a\rig)=0$ or $I_{0}\neq\emptyset$. Denote by $P$ the set of partitions of size $\#I+\#J$ of $\operatorname{Supp}\lef(a+b\rig)$, and put:
\begin{equation*}
v\coloneqq\begin{cases}u\lef(b\rig)&\text{if }J_{0}\neq\emptyset\text{,}\\0&\text{otherwise.}\end{cases}
\end{equation*}

Then for any $\eta,\eta'\in W\lef(k\rig)$ there exists a function $s\colon P\to\mathbb{Z}_{\lef\langle p\rig\rangle}$ such that:
\begin{gather*}
\forall L\in P,\ \begin{cases}p^{v+u\lef(a+b\rig)}\mid s\lef(L\rig)&\text{if }L_{0}=\emptyset\text{,}\\p^{v}\mid s\lef(L'\rig)&\text{otherwise,}\end{cases}\\
e\lef(\eta,a,I\rig)e\lef(\eta',b,J\rig)=\sum_{L\in P}e\lef(s\lef(L\rig)V^{u\lef(a\rig)-u\lef(a+b\rig)}\lef(\eta F^{u\lef(a\rig)-u\lef(b\rig)}\lef(\eta'\rig)\rig),a+b,L\rig)\text{.}
\end{gather*}
\end{prop}

\begin{proof}
The author has demonstrated this proposition in \cite[lemma 2.5]{pseudovaluationsonthederhamwittcomplex} for the case where $u\lef(a\rig)=0$, and in \cite[lemma 2.6]{pseudovaluationsonthederhamwittcomplex} in the case where $I_{0}\neq\emptyset$.
\end{proof}

We recall the action of $F$, $V$ and $d$ on these elements.

\begin{prop}\label{vactionone}
For all $\lef(a,I\rig)\in\mathcal{P}$ and all $\eta\in W\lef(k\rig)$ we have:
\begin{equation*}
V\lef(e\lef(\eta,a,I\rig)\rig)=\begin{cases}e\lef(V\lef(\eta\rig),\frac{a}{p},I\rig)&\text{if }\operatorname{v}_{p}\lef(a\rig)>0\text{,}\\e\lef(p\eta,\frac{a}{p},I\rig)&\text{if }\operatorname{v}_{p}\lef(a\rig)\leqslant0\text{ and }I_{0}=\emptyset\text{,}\\e\lef(\eta,\frac{a}{p},I\rig)&\text{if }\operatorname{v}_{p}\lef(a\rig)\leqslant0\text{ and }I_{0}\neq\emptyset\text{.}\end{cases}
\end{equation*}
\end{prop}

\begin{proof}
This is a reformulation of \cite[proposition 2.5]{derhamwittcohomologyforaproperandsmoothmorphism}.
\end{proof}

\begin{prop}\label{dactionone}
For all $\lef(a,I\rig)\in\mathcal{P}$ and all $\eta\in W\lef(k\rig)$ we have:
\begin{equation*}
d\lef(e\lef(\eta,a,I\rig)\rig)=\begin{cases}0&\text{if }I_{0}=\emptyset\text{,}\\e\lef(\eta,a,I\cup\lef\{\min\lef(a\rig)\rig\}\rig)&\text{if }I_{0}\neq\emptyset\text{ and }\operatorname{v}_{p}\lef(a\rig)\leqslant0\text{,}\\p^{\operatorname{v}_{p}\lef(a\rig)}e\lef(\eta,a,I\cup\lef\{\min\lef(a\rig)\rig\}\rig)&\text{if }I_{0}\neq\emptyset\text{ and }\operatorname{v}_{p}\lef(a\rig)>0\text{.}\end{cases}
\end{equation*}
\end{prop}

\begin{proof}
This is a reformulation of \cite[proposition 2.6]{derhamwittcohomologyforaproperandsmoothmorphism}.
\end{proof}

\begin{thrm}\label{structuretheorem}
Let $x\in W\Omega_{k\lef[\underline{X}\rig]/k}$. There is a unique map $\eta\colon\begin{array}{rl}\mathcal{P}\to&W\lef(k\rig)\\\lef(a,I\rig)\mapsto&\eta_{a,I}\end{array}$ such that:
\begin{equation*}
x=\sum_{\lef(a,I\rig)\in\mathcal{P}}e\lef(\eta_{a,I},a,I\rig)\text{.}
\end{equation*}

Moreover, any such series converges in the canonical topology of $W\Omega_{k\lef[\underline{X}\rig]/k}$ if and only if for all $m\in\mathbb{N}$ we have $V^{u\lef(a\rig)}\lef(\eta_{a,I}\rig)\in V^{m}\lef(W\lef(k\rig)\rig)$ except for finitely many $\lef(a,I\rig)\in\mathcal{P}$.
\end{thrm}

\begin{proof}
See \cite[theorem 2.8]{derhamwittcohomologyforaproperandsmoothmorphism}.
\end{proof}

\begin{deff}
A differential form $x=\sum_{\lef(a,I\rig)\in\mathcal{P}}e\lef(\eta_{a,I},a,I\rig)\in W\Omega_{k\lef[\underline{X}\rig]/k}$ is \textbf{integral} if $\eta_{a,I}=0$ whenever $u\lef(a\rig)\neq0$. We will denote by $W\Omega^{\mathrm{int}}_{k\lef[\underline{X}\rig]/k}$ the subcomplex of all integral elements of the de Rham--Witt complex.

The element $x$ is said to be \textbf{fractional} if $u\lef(a\rig)=0$ implies that $\eta_{a,I}=0$ for all $\lef(a,I\rig)\in\mathcal{P}$. We shall denote by $W\Omega^{\mathrm{frac}}_{k\lef[\underline{X}\rig]/k}$ the subset of all fractional elements of the de Rham--Witt complex.

The element $x$ is said to be \textbf{pure fractional} when $\eta_{a,I}=0$ if $u\lef(a\rig)=0$ or $I_{0}=\emptyset$. We denote by $W\Omega^{\mathrm{frp}}_{k\lef[\underline{X}\rig]/k}$ the subset of all pure fractional elements of the de Rham--Witt complex.
\end{deff}

We thus have two decompositions as graded $W\lef(k\rig)$-modules:
\begin{align}
W\Omega_{k\lef[\underline{X}\rig]/k}&=W\Omega^{\mathrm{int}}_{k\lef[\underline{X}\rig]/k}\oplus W\Omega^{\mathrm{frac}}_{k\lef[\underline{X}\rig]/k}\text{,}\label{intfracdecomposition}\\
W\Omega^{\mathrm{frac}}_{k\lef[\underline{X}\rig]/k}&=W\Omega^{\mathrm{frp}}_{k\lef[\underline{X}\rig]/k}\oplus d\lef(W\Omega^{\mathrm{frp}}_{k\lef[\underline{X}\rig]/k}\rig)\label{frpdfrpdecomposition}\text{.}
\end{align}

\section{Relative perfectness}

In this section, $\overline{R}$ shall denote a commutative ring of characteristic $p$. We will study the notion of relative perfectness, and see how it interacts nicely with Witt vectors.

\begin{deff}
A commutative $\overline{R}$-algebra $\overline{S}$ is said to be \textbf{relatively perfect} when the relative Frobenius $\operatorname{Frob}_{\overline{S}/\overline{R}}$ is an isomorphism. We shall also say that $\overline{S}$ is \textbf{relatively semiperfect} when $\operatorname{Frob}_{\overline{S}/\overline{R}}$ is surjective.
\end{deff}

At the end of this section we will see that, under some reasonable hypotheses, relative perfectness is a necessary condition for the freeness of some Witt vector rings.

\begin{lemm}\label{characteriserelativelysemiperfect}
Let $\overline{S}$ be a commutative $\overline{R}$-algebra. Suppose given a $\overline{R}$-generating family $\lef(s_{i}\rig)_{i\in I}$ of $\overline{S}$, where $I$ is a set. Then the following conditions are equivalent:
\begin{enumerate}
\item the $\overline{R}$-algebra $\overline{S}$ is relatively semiperfect;
\item for all $l\in\mathbb{N}$, the family $\lef({s_{i}}^{p^{l}}\rig)_{i\in I}$ is a $\overline{R}$-generating set of $\overline{S}$;
\item the family $\lef({s_{i}}^{p}\rig)_{i\in I}$ is a $\overline{R}$-generating set of $\overline{S}$.
\end{enumerate}
\end{lemm}

\begin{proof}
Denote by $\psi\colon\overline{R}\to\overline{S}$ the structural morphism of $\overline{S}$. Assume that $\overline{S}$ is a relatively semiperfect $\overline{R}$-algebra. Then by definition we have a surjective morphism $\operatorname{Frob}_{\overline{S}/\overline{R}}\colon\begin{array}{rl}\overline{S}\otimes_{\psi,\operatorname{Frob}_{\overline{R}}}\overline{R}\to&\overline{S}\\s\otimes r\mapsto&s^{p}\psi\lef(r\rig)\end{array}$ where $r\in\overline{R}$ and $s\in\overline{S}$.

Therefore, any element in $\overline{S}$ has an antecedent $x\in\overline{S}\otimes_{\psi,\operatorname{Frob}_{\overline{R}}}\overline{R}$. But $x$ can be written as $\sum_{i\in I}s_{i}\otimes r_{i}$, where $r_{i}\in\overline{R}$ for all $i\in I$ are almost all naught. In particular $\operatorname{Frob}_{\overline{S}/\overline{R}}\lef(x\rig)=\sum_{i\in I}{s_{i}}^{p}\psi\lef(r_{i}\rig)$. We can repeat this process by induction on $l\in\mathbb{N}$ to prove that $\lef({s_{i}}^{p^{l}}\rig)_{i\in I}$ is a $\overline{R}$-generating set of $\overline{S}$.

Conversely, if $\lef({s_{i}}^{p}\rig)_{i\in I}$ is a $\overline{R}$-generating set of $\overline{S}$, then it is clear that $\operatorname{Frob}_{\overline{S}/\overline{R}}$ is surjective since $\lef(s_{i}\rig)_{i\in I}$ is also a $\overline{R}$-generating set of $\overline{S}$.
\end{proof}

The following result is a standard result in the theory of relatively perfect algebras. It was first stated in the context of $p$-bases and formally smooth morphisms. For clarity, we restate and reprove it in our context.

\begin{prop}\label{relativelyperfectimpliesformallyetale}
A relatively perfect commutative $\overline{R}$-algebra is formally étale.
\end{prop}

\begin{proof}
The classical reference is \cite[théorème 21.2.7]{elementsdegeometriealgebriqueivetudelocaledesschemasetdesmorphismesdeschemaspremierepartie}. However the use of all this formalism is useless in this setting, so for the impatient mathematician we indicate how this proof simplifies with contemporary terminology.

Let $f\colon\overline{R}\to\overline{S}$ be a morphism commutative rings making $\overline{S}$ a relatively perfect $\overline{R}$-algebra. Let $\varphi\colon\overline{R}\to\overline{A}$ be a morphism of commutative rings. Let $I$ be an ideal of square zero of $\overline{A}$. Let $\psi\colon\overline{S}\to\overline{A}/I$ be a morphism of $\overline{R}$-algebras. We thus have the following commutative diagram:
\begin{equation*}
\begin{tikzcd}[column sep=huge]
\overline{R}\arrow[r,"\operatorname{Frob}_{\overline{R}}"]\arrow[rrr,bend left,"\varphi"]\arrow[d,"f"']&\overline{R}\arrow[r,"\varphi"]\arrow[d,"f"']&\overline{A}\arrow[d,"\pi"]&\overline{A}\arrow[d,"\pi"]\arrow[l,"\operatorname{Frob}_{\overline{A}}"']\\
\overline{S}\arrow[r,"\operatorname{Frob}_{\overline{S}}"']\arrow[rrr,bend right,"\psi"']&\overline{S}\arrow[r,"\psi"']\arrow[ru,dashed,"g"]&\overline{A}/I&\overline{A}/I\text{.}\arrow[l,"\operatorname{Frob}_{\overline{A}/I}"]\arrow[lu,dotted,"u"']
\end{tikzcd}
\end{equation*}

The dotted arrow $u\colon\overline{A}/I\to\overline{A}$ comes from the universal property of the quotient: since $I$ has square zero, the Frobenius endomorphism must factor that way.

Our goal is to prove that there exists a unique morphism of rings $g\colon\overline{S}\to\overline{A}$ such that the central square commutes. If such a map were to exist, by a little bit of diagram chasing we see that we should have:
\begin{equation*}
g\circ\operatorname{Frob}_{\overline{S}}=\operatorname{Frob}_{\overline{A}}\circ g=u\circ\pi\circ g=u\circ\psi\text{.}
\end{equation*}

Now, since the leftmost square is cocartesian by hypothesis on $\overline{S}$, this gives us the existence and the unicity of the dashed arrow.
\end{proof}

Let us now state a first result characterising free relatively perfect algebras. Later on, with proposition \ref{decompositionfreewitt}, we will give even more equivalent conditions.

\begin{prop}\label{characteriserelativelyperfect}
Let $\overline{S}$ be a commutative $\overline{R}$-algebra. Suppose given a $\overline{R}$-basis $\lef(s_{i}\rig)_{i\in I}$ of $\overline{S}$, where $I$ is a set. Then the following conditions are equivalent:
\begin{enumerate}
\item the $\overline{R}$-algebra $\overline{S}$ is relatively perfect;
\item for all $l\in\mathbb{N}$, the family $\lef({s_{i}}^{p^{l}}\rig)_{i\in I}$ is a $\overline{R}$-basis of $\overline{S}$;
\item the family $\lef({s_{i}}^{p}\rig)_{i\in I}$ is a $\overline{R}$-basis of $\overline{S}$.
\end{enumerate}

Moreover, when $I$ is finite, these conditions are also equivalent to the following ones:
\begin{enumerate}\setcounter{enumi}{3}
\item the $\overline{R}$-algebra $\overline{S}$ is étale;
\item the $\overline{R}$-algebra $\overline{S}$ is formally étale;
\item the $\overline{R}$-algebra $\overline{S}$ is relatively semiperfect;
\item for all $l\in\mathbb{N}$, the family $\lef({s_{i}}^{p^{l}}\rig)_{i\in I}$ is a $\overline{R}$-generating set of $\overline{S}$;
\item the family $\lef({s_{i}}^{p}\rig)_{i\in I}$ is a $\overline{R}$-generating set of $\overline{S}$.
\end{enumerate}
\end{prop}

\begin{proof}
Let $\psi\colon\overline{R}\to\overline{S}$ be the injective structural morphism of $\overline{S}$. Let us assume $\lef(1\rig)$; that is, that $\operatorname{Frob}_{\overline{S}/\overline{R}}$ is an isomorphism. By lemma \ref{characteriserelativelysemiperfect}, then $\lef({s_{i}}^{p^{l}}\rig)_{i\in I}$ is a $\overline{R}$-generating set of $\overline{S}$ for all $l\in\mathbb{N}$.

Let $\lef(r_{i}\rig)_{i\in I}\in\overline{R}^{I}$ be almost all naught elements such that $\sum_{i\in I}{s_{i}}^{p}\psi\lef(r_{i}\rig)=0$. Then, by the injectivity of the relative Frobenius we have $\sum_{i\in I}s_{i}\otimes r_{i}=0$, which in turn implies that $r_{i}=0$ for all $i\in I$. We can repeat this process by induction on $l\in\mathbb{N}$ and show that $\lef({s_{i}}^{p^{l}}\rig)_{i\in I}$ is a $\overline{R}$-basis of $\overline{S}$. So we have $\lef(1\rig)\implies\lef(2\rig)\implies\lef(3\rig)$.

For $\lef(3\rig)\implies\lef(1\rig)$, if $\lef({s_{i}}^{p}\rig)_{i\in I}$ is a $\overline{R}$-basis of $\overline{S}$, then we find that $\operatorname{Frob}_{\overline{S}/\overline{R}}$ is an isomorphism using analogous arguments.

Without any hypothesis on $I$, notice that we always have $\lef(4\rig)\implies\lef(1\rig)$ by \cite[0EBS]{stacksproject}, and $\lef(1\rig)\implies\lef(5\rig)$ by proposition \ref{relativelyperfectimpliesformallyetale}. Also, $\lef(1\rig)\implies\lef(6\rig)$ is easy, and we have the equivalences $\lef(6\rig)\iff\lef(7\rig)\iff\lef(8\rig)$ from lemma \ref{characteriserelativelysemiperfect}.

So we only have two implications left to prove in the case where $I$ is finite. Let us start with $\lef(5\rig)\implies\lef(4\rig)$. Note that since $\overline{S}$ is a finite free $\overline{R}$-module, it is of finite presentation. Hence $\overline{S}$ is also a finitely presented $\overline{R}$-algebra, and we can conclude by applying \cite[00UR]{stacksproject}.

Finally, $\lef(8\rig)\implies\lef(3\rig)$ is basic commutative algebra \cite[05G8]{stacksproject}.
\end{proof}

We now turn to the link between relatively perfect algebras and Witt vectors. We shall use the notion of $\delta$-rings as introduced by Joyal, which are commutative rings endowed with a $\delta$-structure mimicking a lift of a Frobenius. We refer to \cite{joyal} for details.

\begin{deff}\label{subwittfunctor}
Let $\mathsf{C}$ be a subcategory of the category of commutative rings $\mathsf{CRing}$. A \textbf{sub Witt functor} is a functor $\mathcal{W}\colon\mathsf{C}\to\delta\mathsf{Ring}$ from that subcategory to the category of $\delta$-rings, for which there exists a natural transformation $i\colon\mathcal{W}\to W$ with the functor of Witt vectors such that for any commutative ring $A$ in $\mathsf{C}$ we have:
\begin{itemize}
\item the morphism $i_{A}$ is injective;
\item the image of $F|_{\mathcal{W}\lef(A\rig)}$ is in $\mathcal{W}\lef(A\rig)$;
\item for any $w\in W\lef(A\rig)$, we have $w\in\mathcal{W}\lef(A\rig)$ if and only if $V\lef(w\rig)\in\mathcal{W}\lef(A\rig)$;
\item for any $a\in A$, the representative $\lef[a\rig]\in\mathcal{W}\lef(A\rig)$;
\item for every sequence $\lef(w\lef(i\rig)\rig)_{i\in\mathbb{N}}\in{\mathcal{W}\lef(A\rig)}^{\mathbb{N}}$, there is a sequence $\lef(n\lef(i\rig)\rig)_{i\in\mathbb{N}}\in\mathbb{N}^{\mathbb{N}}$ such that $\sum_{i\in\mathbb{N}}p^{n\lef(i\rig)}w\lef(i\rig)$ converges in $\mathcal{W}\lef(A\rig)$.
\end{itemize}
\end{deff}

Notice that the Witt vector functor $W$ is a sub Witt functor. Later in this article, we shall be interested in another sub Witt functor given by overconvergent Witt vectors.

\begin{lemm}\label{tensorwittissquarezero}
Let $\mathcal{W}\colon\mathsf{C}\to\delta\mathsf{Ring}$ be a sub Witt functor. Let $\overline{S}$ be a commutative $\overline{R}$-algebra. Assume that both $\overline{R}$ and $\overline{S}$ are objects in $\mathsf{C}$. Then, the kernel of the morphism of $\overline{R}$-algebras $\mathcal{W}\lef(\overline{S}\rig)\otimes_{\mathcal{W}\lef(\overline{R}\rig)}\overline{R}\to\overline{S}$ is square zero.
\end{lemm}

\begin{proof}
The kernel of this morphism is the image of $V\lef(\mathcal{W}\lef(\overline{S}\rig)\rig)$ through the natural map $\mathcal{W}\lef(\overline{S}\rig)\to\mathcal{W}\lef(\overline{S}\rig)\otimes_{\mathcal{W}\lef(\overline{R}\rig)}\overline{R}$. Let $x,y\in\mathcal{W}\lef(\overline{S}\rig)$. Then, we can apply \cite[IX. \textsection1 proposition 5]{algebrecommutativechapitres} to find that $V\lef(x\rig)V\lef(y\rig)=V^2\lef(F\lef(xy\rig)\rig)=pV\lef(xy\rig)$, which is sent to zero in the above kernel.
\end{proof}

\begin{lemm}\label{wittfreeimpliesfinitefree}
Let $\mathcal{W}\colon\mathsf{C}\to\delta\mathsf{Ring}$ be a sub Witt functor. Let $\overline{S}$ be a commutative $\overline{R}$-algebra. Assume that both $\overline{R}$ and $\overline{S}$ are reduced and are objects in $\mathsf{C}$. Let $I$ be a set. Assume that there exists a basis $\lef(s\lef(i\rig)\rig)_{i\in I}\in{\mathcal{W}\lef(\overline{S}\rig)}^{I}$ of $\mathcal{W}\lef(\overline{S}\rig)$ as a $\mathcal{W}\lef(\overline{R}\rig)$-module.

Then, $I$ is a finite set.
\end{lemm}

\begin{proof}
To get the decomposition of an element $w\in\mathcal{W}\lef(\overline{S}\rig)$, one can first reduce it to the free $\overline{\mathcal{W}\lef(\overline{R}\rig)}$-module $\overline{\mathcal{W}\lef(\overline{S}\rig)}$. We get an almost nought family of Witt vectors $\lef(w\lef(i\rig)\rig)_{i\in I}\in{\mathcal{W}\lef(\overline{R}\rig)}^{I}$ such that $w-\sum_{i\in I}s\lef(i\rig)w\lef(i\rig)$ is a multiple of $p$. As $\overline{R}$ is reduced, $\mathcal{W}\lef(\overline{R}\rig)$ is $p$-torsion free, so we can divide the result and restart the process. By $p$-adic convergence, we then get the unique decomposition.

If $I$ is infinite, we can view $\mathbb{N}$ as a subset of $I$. We can find a sequence $\lef(n\lef(i\rig)\rig)_{i\in\mathbb{N}}\in\mathbb{N}^{\mathbb{N}}$ such that $\sum_{i\in\mathbb{N}}p^{n\lef(i\rig)}s\lef(i\rig)$ is an element in $\mathcal{W}\lef(\overline{S}\rig)$. The above process shows that it cannot be written as a finite $\overline{\mathcal{W}\lef(\overline{R}\rig)}$-linear combination of elements in $\lef(s\lef(i\rig)\rig)_{i\in I}$, hence $I$ has to be finite.
\end{proof}

\begin{lemm}\label{wittfreeimpliesredfree}
Let $\mathcal{W}\colon\mathsf{C}\to\delta\mathsf{Ring}$ be a sub Witt functor. Let $\overline{S}$ be a commutative $\overline{R}$-algebra. Assume that both are objects in $\mathsf{C}$. Let $I$ be a set. Assume that there exists a basis $\lef(s\lef(i\rig)\rig)_{i\in I}\in{\mathcal{W}\lef(\overline{S}\rig)}^{I}$ of $\mathcal{W}\lef(\overline{S}\rig)$ as a $\mathcal{W}\lef(\overline{R}\rig)$-module such that $\lef(F^{n}\lef(s\lef(i\rig)\rig)\rig)_{i\in I}\in{\mathcal{W}\lef(\overline{S}\rig)}^{I}$ also is a $\mathcal{W}\lef(\overline{R}\rig)$-basis for all $n\in\mathbb{N}$.

Denote by $\mathcal{W}'\lef(\overline{R}_{\mathsf{red}}\rig)$ and $\mathcal{W}'\lef(\overline{S}_{\mathsf{red}}\rig)$ the respective images of $\mathcal{W}\lef(\overline{R}\rig)$ and $\mathcal{W}\lef(\overline{S}\rig)$ in $W\lef(\overline{R}_{\mathsf{red}}\rig)$ and $W\lef(\overline{S}_{\mathsf{red}}\rig)$. Then we have an isomorphism of $\mathcal{W}\lef(\overline{R}\rig)$-algebras:
\begin{equation*}
\mathcal{W}'\lef(\overline{S}_{\mathsf{red}}\rig)\cong\mathcal{W}\lef(\overline{S}\rig)\otimes_{\mathcal{W}\lef(\overline{R}\rig)}\mathcal{W}'\lef(\overline{R}_{\mathsf{red}}\rig)\text{.}
\end{equation*}
\end{lemm}

\begin{proof}
The surjective morphism of $\mathcal{W}\lef(\overline{R}\rig)$-algebras $\varphi\colon\mathcal{W}\lef(\overline{S}\rig)\to\mathcal{W}'\lef(\overline{S}_{\mathsf{red}}\rig)$ factors through $\mathcal{W}\lef(\overline{S}\rig)\otimes_{\mathcal{W}\lef(\overline{R}\rig)}\mathcal{W}'\lef(\overline{R}_{\mathsf{red}}\rig)$. Let us find the inverse of this factorisation by the universal property of the quotient. The kernel of $\varphi$ is the set of the $w\in\mathcal{W}\lef(\overline{S}\rig)$ such that for every $n\in\mathbb{N}$ there is an integer $m_{n}\in\mathbb{N}$ satisfying $F^{m_{n}}\lef(w\rig)\in V^{n}\lef(\mathcal{W}\lef(\overline{S}\rig)\rig)$.

Let $w\in\operatorname{Ker}\lef(\varphi\rig)$. For each $n\in\mathbb{N}$, denote by $w_{\leqslant n}\in\mathcal{W}\lef(\overline{S}\rig)$ the Witt vector having the same first $n$ components as $w$, and whose other components are all nought. It is indeed in $\mathcal{W}\lef(\overline{S}\rig)$ because $w_{\leqslant n}=\sum_{i=0}^{n-1}V^{i}\lef(\lef[w_i\rig]\rig)$ and $\mathcal{W}$ is a sub Witt functor. In particular, $F^{m_{n}}\lef(w_{\leqslant n}\rig)=0$ and $\lim_{n\to\infty}w_{\leqslant n}=w$ for the $V$-adic topology.

Notice that the image of these $w_{\leqslant n}$ in $\mathcal{W}\lef(\overline{S}\rig)\otimes_{\mathcal{W}\lef(\overline{R}\rig)}\mathcal{W}'\lef(\overline{R}_{\mathsf{red}}\rig)$ is $0$. Indeed, we have $p^{m_{n}}w_{\leqslant n}=0$, but the tensor product is $p$-torsion free because it is a free module over a $p$-torsion free ring. From this, we see that for all $n\in\mathbb{N}$, $w$ and $w-w_{\leqslant n}$ have the same image $x\in\mathcal{W}\lef(\overline{S}\rig)\otimes_{\mathcal{W}\lef(\overline{R}\rig)}\mathcal{W}'\lef(\overline{R}_{\mathsf{red}}\rig)$.

For every such $n\in\mathbb{N}$, let $\lef(r_{n,i}\rig)_{i\in I}\in{\mathcal{W}\lef(\overline{R}\rig)}^{I}$ be the unique family of Witt vectors satisfying $w-w_{\leqslant n}=V^{n}\lef(\sum_{i\in I}r_{n,i}F^{n}\lef(s\lef(i\rig)\rig)\rig)$. Thus, $w-w_{\leqslant n}=\sum_{i\in I}V^{n}\lef(r_{n,i}\rig)s\lef(i\rig)$. In particular, $x\in\bigcap_{n\in\mathbb N}\bigoplus_{i\in I}V^{n}\lef(\mathcal{W}'\lef(\overline{R}_{\mathsf{red}}\rig)\rig)s\lef(i\rig)$, so $x=0$.
\end{proof}

\begin{prop}\label{whencanitbewittfree}
Let $\mathcal{W}\colon\mathsf{C}\to\delta\mathsf{Ring}$ be a sub Witt functor. Let $\overline{S}$ be a commutative $\overline{R}$-algebra. Assume that both $\overline{R}$ and $\overline{S}$ are objects in $\mathsf{C}$. Let $I$ be a finite set. Assume that there exists a basis $\lef(s\lef(i\rig)\rig)_{i\in I}\in{\mathcal{W}\lef(\overline{S}\rig)}^{I}$ of $\mathcal{W}\lef(\overline{S}\rig)$ as a $\mathcal{W}\lef(\overline{R}\rig)$-module.

Then $\overline{S}$ is a relatively semiperfect $\overline{R}$-algebra.
\end{prop}

\begin{proof}
Let $\mathfrak{m}$ be a maximal ideal of $\overline{R}$. Denote by $\kappa\lef(\mathfrak{m}\rig)\coloneqq\overline{R}_{\mathfrak{m}}/\mathfrak{m}\overline{R}_{\mathfrak{m}}\cong\overline{R}/\mathfrak{m}$ the residue field at $\mathfrak{m}$. We thus have a surjective map $\mathcal{W}\lef(\overline{S}\rig)\to W_{\#I+1}\lef(\overline{S}\otimes_{\overline{R}}\kappa\lef(\mathfrak{m}\rig)\rig)$. Hence, we have the following commutative diagram of $\mathcal{W}\lef(\overline{R}\rig)$-algebras, in which all the arrows are surjective:
\begin{equation*}
\begin{tikzcd}
\mathcal{W}\lef(\overline{S}\rig)\otimes_{\mathcal{W}\lef(\overline{R}\rig)}W_{\#I+1}\lef(\kappa\lef(\mathfrak{m}\rig)\rig)\arrow[d]\arrow[r]&W_{\#I+1}\lef(\overline{S}\otimes_{\overline{R}}\kappa\lef(\mathfrak{m}\rig)\rig)\arrow[d]\\
\mathcal{W}\lef(\overline{S}\rig)\otimes_{\mathcal{W}\lef(\overline{R}\rig)}\kappa\lef(\mathfrak{m}\rig)\arrow[r]&\overline{S}\otimes_{\overline{R}}\kappa\lef(\mathfrak{m}\rig)\text{.}
\end{tikzcd}
\end{equation*}

Let $x\in\overline{S}\otimes_{\overline{R}}\kappa\lef(\mathfrak{m}\rig)$. For each integer $u\in\lef\llbracket1,\#I+1\rig\rrbracket$, let us choose an antecedent $\sum_{i\in I}s\lef(i\rig)\otimes w\lef(u,i\rig)\in\mathcal{W}\lef(\overline{S}\rig)\otimes_{\mathcal{W}\lef(\overline{R}\rig)}W_{\#I+1}\lef(\kappa\lef(\mathfrak{m}\rig)\rig)$ by the top horizontal arrow of $V^{u}\lef(x\rig)$, where all $w\lef(u,i\rig)\in W_{\#I+1}\lef(\kappa\lef(\mathfrak{m}\rig)\rig)$.

Projecting these elements in $\mathcal{W}\lef(\overline{S}\rig)\otimes_{\mathcal{W}\lef(\overline{R}\rig)}\kappa\lef(\mathfrak{m}\rig)$, which is a $\kappa\lef(\mathfrak{m}\rig)$-vector space of dimension $\#I$, we get a family $\lef\{\sum_{i\in I}s\lef(i\rig)\otimes{w\lef(u,i\rig)}_{0}\rig\}_{u\in\lef\llbracket1,\#I+1\rig\rrbracket}$ which is $\kappa\lef(\mathfrak{m}\rig)$-linearly dependent. That is, for each $u\in\lef\llbracket1,\#I+1\rig\rrbracket$ there is a scalar $k_{u}\in\kappa\lef(\mathfrak{m}\rig)$ such that we have a non trivial relation:
\begin{equation*}
\forall i\in I,\ \sum_{u=1}^{\#I+1}k_{u}{w\lef(u,i\rig)}_{0}=0\text{.}
\end{equation*}

For all $i\in I$, let $y_{i}\coloneqq\sum_{u=1}^{\#I+1}\lef[k_{u}\rig]w\lef(u,i\rig)\in V\lef(W_{\#I+1}\lef(\kappa\lef(\mathfrak{m}\rig)\rig)\rig)$. We thus get in $W_{\#I+1}\lef(\overline{S}\otimes_{\overline{R}}\kappa\lef(\mathfrak{m}\rig)\rig)$ the following formulae:
\begin{equation*}
\sum_{i\in I}y_{i}s\lef(i\rig)=\sum_{u=1}^{\#I+1}\lef[k_{u}\rig]V^{u}\lef(x\rig)\text{.}
\end{equation*}

Let $u\in\lef\llbracket1,\#I+1\rig\rrbracket$ be such that $k_{u}\neq0$. Then by \eqref{vxfyvxy} and \cite[IX. \textsection1 proposition 5]{algebrecommutativechapitres}, looking at the $u$-th coordinate in the above equality we find that $x$ can be written as a polynomial with coefficients in $\kappa\lef(\mathfrak{m}\rig)$ and indeterminates $\lef\{{s\lef(i\rig)_{v}}^{p}\rig\}_{\substack{i\in I\\v\in\lef\llbracket0,\#I\rig\rrbracket}}$.

By Nakayama's lemma, this implies that there is a finite set of generators of the $\overline{R}_{\mathfrak{m}}$-module $\overline{S}\otimes_{\overline{R}}\overline{R}_{\mathfrak{m}}$ containing $\lef\{{s\lef(i\rig)_{0}}^{p}\rig\}_{i\in I}$. Thus, lemma \ref{characteriserelativelysemiperfect} tells us that this is a relatively semiperfect $\overline{R}_{\mathfrak{m}}$-algebra. So, $\overline{S}$ is a relatively perfect $\overline{R}$-algebra.
\end{proof}

\begin{coro}\label{whencanitbewittfreecoro}
Let $\mathcal{W}\colon\mathsf{C}\to\delta\mathsf{Ring}$ be a sub Witt functor. Let $\overline{S}$ be a commutative $\overline{R}$-algebra. Assume that both $\overline{R}$ and $\overline{S}$ are objects in $\mathsf{C}$. Let $I$ be a set. Assume that there exists a basis $\lef(s\lef(i\rig)\rig)_{i\in I}\in{\mathcal{W}\lef(\overline{S}\rig)}^{I}$ of $\mathcal{W}\lef(\overline{S}\rig)$ as a $\mathcal{W}\lef(\overline{R}\rig)$-module such that $\lef(F^{n}\lef(s\lef(i\rig)\rig)\rig)_{i\in I}$ also is a $\mathcal{W}\lef(\overline{R}\rig)$-basis for all $n\in\mathbb{N}$.

Then, $I$ is finite, $\overline{S}$ is an étale $\overline{R}$-algebra, and $\lef({s\lef(i\rig)}_{0}\rig)_{i\in I}$ is a basis of $\overline{S}$ as a $\overline{R}$-module.
\end{coro}

\begin{proof}
Let us first apply lemma \ref{wittfreeimpliesredfree} to get a sub Witt functor $\mathcal{W}'$ as in its statement, on the subcategory of $\mathsf{CRing}$ whose only morphism which is not an identity is $\overline{R}_{\mathsf{red}}\to\overline{S}_{\mathsf{red}}$. Then, we can apply lemma \ref{wittfreeimpliesfinitefree} and get that $I$ is finite. Proposition \ref{whencanitbewittfree} then tells us that $\overline{S}$ is a relatively semiperfect $\overline{R}$-algebra.

Let $\lef(r\lef(i\rig)\rig)_{i\in I}\in\overline{R}^{I}$ such that $\sum_{i\in I}{s\lef(i\rig)}_{0}r\lef(i\rig)=0$. Let $w\coloneqq\sum_{i\in I}s\lef(i\rig)\lef[r\lef(i\rig)\rig]$. Then, there exists $w'\in\mathcal{W}\lef(\overline{S}\rig)$ such that $w=V\lef(w'\rig)$. By hypothesis, there is also a family $\lef(t\lef(i\rig)\rig)_{i\in I}\in{\mathcal{W}\lef(\overline{R}\rig)}^{I}$ such that $w'=\sum_{i\in I}F\lef(s\lef(i\rig)\rig)t\lef(i\rig)$. Thus, $w=\sum_{i\in I}s\lef(i\rig)V\lef(t\lef(i\rig)\rig)$. By unicity, this implies that $r\lef(i\rig)=0$ for all $i\in I$. We have thus shown that $\lef({s\lef(i\rig)}_{0}\rig)_{i\in I}$ is a basis of $\overline{S}$ as a $\overline{R}$-module, and we can conclude using proposition \ref{characteriserelativelyperfect}.
\end{proof}

This last result gives us the necessary conditions for some bases of the $W\lef(\overline{R}\rig)$-module to come from a basis of the $\overline{R}$-module. We will see that these conditions are sometimes sufficient in proposition \ref{decompositionfreewitt}. The condition on the $\mathcal{W}\lef(\overline{R}\rig)$-basis is very mild. It is in fact always true in the situations we shall be interested in for the functor of Witt vectors, and the author knows no counterexample in the general setting. We shall therefore focus on this setting in the next section.

\section{The de Rham--Witt complex for a polynomial ring modulo $p$}

Unless otherwise stated, here $k$ shall denote a commutative $\mathbb{Z}_{\lef\langle p\rig\rangle}$-algebra. In this section we will study the de Rham--Witt complex modulo $p$.

We fix $n\in\mathbb{N}$ and keep the notation $k\lef[\underline{X}\rig]=\lef[X_{1},\ldots,X_{n}\rig]$. Recall that $\mathcal{P}$ has been defined as the set of all $\lef(a,I\rig)$, where $a\colon\lef\llbracket1,n\rig\rrbracket\to\mathbb{N}\lef[\frac{1}{p}\rig]$ is a weight function and $I\subset\operatorname{Supp}\lef(a\rig)$ is a partition of $a$.

We are first going to study the sub-$\overline{W\lef(k\rig)}$-module of $\overline{W\Omega^{\mathrm{int}}_{k\lef[\underline{X}\rig]/k}}$ whose elements are all $\overline{e\lef(\eta,a,I\rig)}$, with $\eta\in W\lef(k\rig)$ and $\lef(a,I\rig)\in\mathcal{P}$ such that $\operatorname{v}_{p}\lef(a\rig)=0$ and $I_{0}\neq\emptyset$. These elements are studied in a less general setting in \cite{overconvergentderhamwittcohomology}, where they are called ``primitive basic Witt differentials''. According to theorem \ref{structuretheorem} and by definition of these elements, this $\overline{W\lef(k\rig)}$-module is free with basis:
\begin{equation*}
\lef\{\overline{e\lef(1,a,I\rig)}\mid\lef(a,I\rig)\in\mathcal{P},\ \operatorname{v}_{p}\lef(a\rig)=0,\ I_{0}\neq\emptyset\rig\}\text{.}
\end{equation*}

Our first aim is to demonstrate that this set is actually a free $\overline{F}^{u}\lef(\overline{W\lef(k\lef[\underline{X}\rig]\rig)}\rig)$-module for any $u\in\mathbb{N}^{*}$. This will allow us to study the $\overline{W\lef(k\lef[\underline{X}\rig]\rig)}$-module structure of the fractional part of the de Rham--Witt complex at the end of the section.

For any $\lef(a,I\rig)\in\mathcal{P}$ such that $\operatorname{v}_{p}\lef(a\rig)=0$, we shall consider the following conditions:
\begin{align}
\exists i\in\operatorname{Supp}\lef(a\rig),\ &a_{i}>p^{u}\text{,}\label{multbypumodpone}\\
\exists i\in I_{0},\ &a_{i}=p^{u}\text{,}\label{multbypumodptwo}\\
\exists l\in\lef\llbracket1,\#I\rig\rrbracket,\ \exists i\in I_{l},\ &\lef(\#I_{l}\geqslant2\rig)\wedge\lef(a_{i}=p^{u}\rig)\text{.}\label{multbypumodpthree}
\end{align}

The purpose of following technical lemmas is to show that the $\overline{F}^{u}\lef(\overline{W\lef(k\lef[\underline{X}\rig]\rig)}\rig)$-module is free, with a basis indexed on the couples which do not satisfy these conditions.

\begin{lemm}\label{combinationone}
Let $u\in\mathbb{N}^{*}$. Let $\eta\in W\lef(k\rig)$. Let $a$ be a weight function such that $\operatorname{v}_{p}\lef(a\rig)=0$. Let $I$ be a partition of $a$ such that $I_{0}\neq\emptyset$. Assume that either \eqref{multbypumodpone} or \eqref{multbypumodptwo} is true for $\lef(a,I\rig)$.

Then, the element $e\lef(\eta,a,I\rig)$ is a $\mathbb{Z}_{\lef\langle p\rig\rangle}$-linear combination of elements of the form:
\begin{equation*}
\lef[{X_{i}}^{p^{u}}\rig]e\lef(\eta,b,J\rig)
\end{equation*}
meeting all the following conditions:
\begin{gather*}
i\in\lef\llbracket1,n\rig\rrbracket\text{,}\\
\lef(b,J\rig)\in\mathcal{P}\text{,}\\
\operatorname{v}_{p}\lef(b\rig)=0\text{,}\\
\lef\lvert a\rig\rvert=\lef\lvert b\rig\rvert+p^{u}\text{,}\\
J_{0}\neq\emptyset\text{,}\\
\#J=\#I\text{.}
\end{gather*}
\end{lemm}

\begin{proof}
By definition:
\begin{equation*}
e\lef(\eta,a,I\rig)=\eta\lef[\underline{X}^{a|_{I_{0}}}\rig]\times\prod_{l=1}^{\#I}F^{\operatorname{v}_{p}\lef(a|_{I_{l}}\rig)}\lef(d\lef(\lef[\underline{X}^{p^{-\operatorname{v}_{p}\lef(a|_{I_{l}}\rig)}a|_{I_{l}}}\rig]\rig)\rig)\text{.}
\end{equation*}

In particular, the lemma is obvious when there exists $i\in I_{0}$ such that $a_{i}\geqslant p^{u}$. So assume that there is $l'\in\lef\llbracket1,\#I\rig\rrbracket$ such that $a_{i}>p^{u}$ for a given $i\in I_{l'}$.

Using \eqref{fdttdt} we get:
\begin{multline*}
F^{\operatorname{v}_{p}\lef(a|_{I_{l'}}\rig)}\lef(d\lef(\lef[\underline{X}^{p^{-\operatorname{v}_{p}\lef(a|_{I_{l'}}\rig)}a|_{I_{l'}}}\rig]\rig)\rig)\\=\lef[\underline{X}^{\lef(1-p^{-\operatorname{v}_{p}\lef(a|_{I_{l'}}\rig)}\rig)a|_{I_{l'}}}\rig]d\lef(\lef[\underline{X}^{p^{-\operatorname{v}_{p}\lef(a|_{I_{l'}}\rig)}a|_{I_{l'}}}\rig]\rig)\text{.}
\end{multline*}

But by Leibniz rule we find:
\begin{multline*}
F^{\operatorname{v}_{p}\lef(a|_{I_{l'}}\rig)}\lef(d\lef(\lef[\underline{X}^{p^{-\operatorname{v}_{p}\lef(a|_{I_{l'}}\rig)}a|_{I_{l'}}}\rig]\rig)\rig)\\=\lef[{X_{i}}^{a_{i}}\rig]\lef[\underline{X}^{\lef(1-p^{-\operatorname{v}_{p}\lef(a|_{I_{l'}}\rig)}\rig)a|_{I_{l'}\smallsetminus\lef\{i\rig\}}}\rig]d\lef(\lef[\underline{X}^{p^{-\operatorname{v}_{p}\lef(a|_{I_{l'}}\rig)}a|_{I_{l'}\smallsetminus\lef\{i\rig\}}}\rig]\rig)\\+p^{-\operatorname{v}_{p}\lef(a|_{I_{l'}}\rig)}a_{i}\lef[{X_{i}}^{a_{i}-1}\rig]\lef[\underline{X}^{a|_{I_{l'}\smallsetminus\lef\{i\rig\}}}\rig]d\lef(\lef[X_{i}\rig]\rig)\text{.}
\end{multline*}

By hypothesis on $a_{i}$, this sum is divided by $\lef[{X_{i}}^{p^{u}}\rig]$. Therefore, we are done proving the lemma if we show that the product of the remaining multiplicand with:
\begin{multline*}
\eta\lef[\underline{X}^{a|_{I_{0}}}\rig]\times\prod_{l\in\lef\llbracket1,\#I\rig\rrbracket\smallsetminus\lef\{l'\rig\}}F^{\operatorname{v}_{p}\lef(a|_{I_{l}}\rig)}\lef(d\lef(\lef[\underline{X}^{p^{-\operatorname{v}_{p}\lef(a|_{I_{l}}\rig)}a|_{I_{l}}}\rig]\rig)\rig)\\=e\lef(\eta,a|_{\operatorname{Supp}\lef(a\rig)\smallsetminus I_{l'}},I\smallsetminus\lef\{i_{l'}\rig\}\rig)
\end{multline*}
is a $\mathbb{Z}_{\lef\langle p\rig\rangle}$-linear combination of elements of the form $e\lef(\eta,b,J\rig)$, where:
\begin{gather*}
\lef(b,J\rig)\in\mathcal{P}\text{,}\\
\operatorname{v}_{p}\lef(b\rig)=0\text{,}\\
\lef\lvert a\rig\rvert=\lef\lvert b\rig\rvert+p^{u}\text{,}\\
J_{0}\neq\emptyset\text{,}\\
\#J=\#I\text{.}
\end{gather*}

To do so, we are going to use a useful reasoning, which we shall call the ``minimum trick'' in the remainder of this section. We first use proposition \ref{dactionone} where it applies in the sum above. Because $u\lef(a_{I_{l'}}\rig)=0$, we can then use proposition \ref{zpcombination} to the product we want to consider, but from which we first exclude the smallest term for the order \eqref{ordersupp}, that is $\lef[{X_{\min\lef(a\rig)}}^{a_{\min\lef(a\rig)}}\rig]$.

Then, the product of the result with this excluded term will be a linear combination of the desired form. We indeed have $J_{0}\neq\emptyset$ because either $\operatorname{v}_{p}\lef(a_{i}\rig)=0$, or $\operatorname{v}_{p}\lef(a_{i}-p^{u}\rig)>0=\operatorname{v}_{p}\lef(a_{\min\lef(a\rig)}\rig)$, so in all cases $\min\lef(a\rig)\prec i$ on $b$.
\end{proof}

\begin{lemm}\label{combinationtwo}
Let $u\in\mathbb{N}^{*}$. Let $\eta\in W\lef(k\rig)$. Let $a$ be a weight function satisfying $\operatorname{v}_{p}\lef(a\rig)=0$. Let $I$ be a partition of $a$ with $I_{0}\neq\emptyset$. Assume that \eqref{multbypumodpthree} is true for $\lef(a,I\rig)$.

Then, the element $e\lef(\eta,a,I\rig)$ is a $\mathbb{Z}_{\lef\langle p\rig\rangle}$-linear combination of elements of either of the forms:
\begin{gather*}
\lef[{X_{i}}^{p^{u}}\rig]e\lef(\eta,b,J\rig)\text{,}\\
e\lef(\eta,a,J'\rig)
\end{gather*}
where $\lef(a,J'\rig)$ satisfies none of the conditions \eqref{multbypumodpone}, \eqref{multbypumodptwo} nor \eqref{multbypumodpthree}, and all the following ones are met:
\begin{gather*}
i\in\lef\llbracket1,n\rig\rrbracket\text{,}\\
\lef(b,J\rig),\lef(a,J'\rig)\in\mathcal{P}\text{,}\\
\operatorname{v}_{p}\lef(b\rig)=0\text{,}\\
\lef\lvert a\rig\rvert=\lef\lvert b\rig\rvert+p^{u}\text{,}\\
J_{0}\neq\emptyset\text{,}\\
{J'}_{0}\neq\emptyset\text{,}\\
\#J=\#{J'}=\#I\text{.}
\end{gather*}
\end{lemm}

\begin{proof}
If $\lef(a,I\rig)$ satisfies either \eqref{multbypumodpone} or \eqref{multbypumodptwo}, then we can conclude immediately with lemma \ref{combinationone}. Therefore, we only need to prove the lemma when $a_{i}\leqslant p^{u}$ for all $i\in\operatorname{Supp}\lef(a\rig)$.

By hypothesis, there is $j\in\lef\llbracket1,\#I\rig\rrbracket$ such that $\#I_{l}\geqslant2$ and $a_{i}=p^{u}$ for some $i\in I_l$. Assume first that there is $i'\in I_{l}$ such that $i\neq i'$ and $a_{i'}=p^{u}$. Then by \eqref{frobeniuswittvectors}, we get the following equality for this factor of $e\lef(\eta,a,I\rig)$:
\begin{multline*}
F^{\operatorname{v}_{p}\lef(a|_{I_{l}}\rig)}\lef(d\lef(\lef[\underline{X}^{p^{-\operatorname{v}_{p}\lef(a|_{I_{l}}\rig)}a|_{I_{l}}}\rig]\rig)\rig)\\=\lef[{X_{i}}^{p^{u}}\rig]F^{\operatorname{v}_{p}\lef(a|_{I_{l}}\rig)}\lef(d\lef(\lef[\underline{X}^{p^{-\operatorname{v}_{p}\lef(a|_{I_{l}}\rig)}a|_{I_{l}\smallsetminus\lef\{i\rig\}}}\rig]\rig)\rig)\\+\lef[{X_{i'}}^{p^{u}}\rig]\lef[\underline{X}^{a|_{I_{l}\smallsetminus\lef\{i,i'\rig\}}}\rig]F^{\operatorname{v}_{p}\lef(a|_{I_{l}}\rig)}\lef(d\lef(\lef[{X_{i}}^{p^{u-\operatorname{v}_{p}\lef(a|_{I_{l}}\rig)}}\rig]\rig)\rig)\text{.}
\end{multline*}

In this case, one can conclude using a reasoning similar to the minimum trick used in the proof of lemma \ref{combinationone}.

Thus, it only remains to prove the lemma when $\lef(a,I\rig)$ has the property that for any $l\in\lef\llbracket1,\#I\rig\rrbracket$, if $a_{i}=p^{u}$, then $a_{i'}<p^{u}$ for all $i'\in I_{l}$ satisfying $i\neq i'$. In particular, by definition of the order \eqref{ordersupp}, then $l$ is the only element of $\lef\llbracket1,\#I\rig\rrbracket$ satisfying \eqref{multbypumodpthree}.

A similar calculation as above, and then applying \eqref{dfpfd} to the result, leads to:
\begin{multline*}
F^{\operatorname{v}_{p}\lef(a|_{I_{l}}\rig)}\lef(d\lef(\lef[\underline{X}^{p^{-\operatorname{v}_{p}\lef(a|_{I_{l}}\rig)}a|_{I_{l}}}\rig]\rig)\rig)\\=\lef[{X_{i}}^{p^{u}}\rig]F^{\operatorname{v}_{p}\lef(a|_{I_{l}}\rig)}\lef(d\lef(\lef[\underline{X}^{p^{-\operatorname{v}_{p}\lef(a|_{I_{l}}\rig)}a|_{I_{l}\smallsetminus\lef\{i\rig\}}}\rig]\rig)\rig)\\+p^{u-\operatorname{v}_{p}\lef(a|_{I_{l}}\rig)}\lef[\underline{X}^{a|_{I_{l}\smallsetminus\lef\{i\rig\}}}\rig]F^{u}\lef(d\lef(\lef[X_{i}\rig]\rig)\rig)\text{.}
\end{multline*}

Putting this sum back in the product defining $e\lef(\eta,a,I\rig)$ will lead to a new sum. It is easy to see that the first term of the sum is of the form $\lef[{X_{i}}^{p^{u}}\rig]e\lef(\eta,b,J\rig)$ with the conditions of the statement of the lemma. The second term of the sum will be a product with multiplicand $\prod_{\substack{i\in\operatorname{Supp}\lef(a\rig)\\ a_{i}=p^{u}}}F^{u}\lef(d\lef(\lef[X_{i}\rig]\rig)\rig)$, and we can apply the minimum trick to the multiplier, leading to a $\mathbb{Z}_{\lef\langle p\rig\rangle}$-linear combination of elements of the form $e\lef(\eta,a,J'\rig)$ as in the statement of the lemma.
\end{proof}

\begin{prop}\label{linearpuintegral}
Let $u\in\mathbb{N}^{*}$. Let $\eta\in W\lef(k\rig)$ and $\lef(b,L\rig)\in\mathcal{P}$ be such that $\operatorname{v}_{p}\lef(b\rig)=0$ and $L_{0}\neq\emptyset$. Then $e\lef(\eta,b,L\rig)$ can be uniquely written as a $\mathbb{Z}_{\lef\langle p\rig\rangle}$-linear combination of elements of the form 
\begin{equation*}
\lef[\underline{X}^{p^{u}c}\rig]e\lef(\eta,a,I\rig)
\end{equation*}
where $c$ is a weight function with integral values, $\lef(a,I\rig)\in\mathcal{P}$ satisfies none of the conditions \eqref{multbypumodpone}, \eqref{multbypumodptwo} nor \eqref{multbypumodpthree}, and the following conditions are met:
\begin{gather*}
b=a+p^{u}c\text{,}\\
\operatorname{v}_{p}\lef(a\rig)=0\text{,}\\
I_{0}\neq\emptyset\text{,}\\
\#I=\#L\text{.}
\end{gather*}
\end{prop}

\begin{proof}
We are first going to show that such a writing exists by induction on $\lef\lvert b\rig\rvert$. If $\lef\lvert b\rig\rvert<p^{u}$, then $\lef(b,L\rig)$ cannot satisfy any of the conditions \eqref{multbypumodpone}, \eqref{multbypumodptwo} and \eqref{multbypumodpthree}. So we can put $\lef(a,I\rig)\coloneqq\lef(b,L\rig)$ and define $c$ to be the zero function and we are done.

Now assume that there is $m\in\mathbb{N}^{*}$ for which the proposition is proven for all $b$ such that $\lef\lvert b\rig\rvert<mp^{u}$. Let $\lef(b,L\rig)\in\mathcal{P}$ be such that $\operatorname{v}_{p}\lef(b\rig)=0$, $\lef\lvert b\rig\rvert<\lef(m+1\rig)p^{u}$ and $L_{0}\neq\emptyset$. If $\lef(b,L\rig)$ satisfies none of the aforementioned conditions, then we can conclude with the same argument than above. Otherwise, we can apply lemmas \ref{combinationone} and \ref{combinationtwo} and write $e\lef(\eta,b,L\rig)$ as a $\mathbb{Z}_{\lef\langle p\rig\rangle}$-linear combination of elements of either form $e\lef(\eta,b,L'\rig)$ or $\lef[{X_{i}}^{p^{u}}\rig]e\lef(\eta,b',L''\rig)$, where $i\in\lef\llbracket1,n\rig\rrbracket$, $\lef(b,L'\rig),\lef(b',L''\rig)\in\mathcal{P}$, $\operatorname{v}_{p}\lef(b'\rig)=0$, $\lef\lvert b\rig\rvert=\lef\lvert b'\rig\rvert+p^{u}$, ${L'}_{0}\neq\emptyset$, ${L''}_{0}\neq\emptyset$ and $\#{L'}=\#{L''}=\#L$. Also, $\lef\lvert b'\rig\rvert<mp^{u}$ and $\lef(b,L'\rig)$ satisfies none of the conditions \eqref{multbypumodpone}, \eqref{multbypumodptwo} nor \eqref{multbypumodpthree}. Furthermore, proposition \ref{zpcombination} implies that any such $b'$ equals $b$ on the whole set of departure, except for $i$. We are finished proving the existence after applying our induction hypothesis.

To see that this writing is unique, fix $m\in\mathbb{N}$ and $b$ a weight function with $\operatorname{v}_{p}\lef(b\rig)=0$. Now consider the $W\lef(k\rig)$-module generated by all $e\lef(1,b,L\rig)$ with $L$ a partition of $b$ such that $L_{0}\neq\emptyset$ and $\#L=m$. By theorem \ref{structuretheorem}, it is a free $W\lef(k\rig)$-module, and this $W\lef(k\rig)$-generating set is a $W\lef(k\rig)$-basis. Therefore, it has rank $\binom{\#\operatorname{Supp}\lef(b\rig)-1}{m}$.

But we have just shown that the set of all $\lef[\underline{X}^{p^{u}c}\rig]e\lef(1,a,I\rig)$, where $a$ and $c$ are weight functions taking values in $\mathbb{N}$ and $I$ is a partition of $a$, such that $b=a+p^{u}c$, and that $\operatorname{v}_{p}\lef(a\rig)=0$, and that $I_{0}\neq\emptyset$, and that $\#I=m$ and that $\lef(a,I\rig)$ satisfies none of the conditions \eqref{multbypumodpone}, \eqref{multbypumodptwo} and \eqref{multbypumodpthree} is a $W\lef(k\rig)$-generating set of that free $W\lef(k\rig)$-module. This set does not generate a bigger $W\lef(k\rig)$-module by virtue of proposition \ref{zpcombination}, and of \cite[proposition 2.2]{overconvergentderhamwittcohomology} whose proof holds verbatim in our situation even though it was demonstrated in a less general setting. We are now going to compute the cardinal of this set.

Consider such $a$, $c$ and $I$. Observe that $c$ is uniquely determined by $a$, so we only have to count how many $\lef(a,I\rig)\in\mathcal{P}$ satisfy all the needed constraints. Notice that for any $\lef(a,I\rig)\in\mathcal{P}$, to meet none of the three conditions above is equivalent to saying that that for all $i\in\lef\llbracket1,n\rig\rrbracket$ we have $a_{i}\leqslant p^{u\lef(b\rig)}$, and $i\in I$ whenever $a_{i}=p^{u\lef(b\rig)}$. Consider $P\coloneqq\lef\{i\in\operatorname{Supp}\lef(b\rig)\mid b_{i}\in p^{u}\mathbb{N}\rig\}$ and $Q\coloneqq\operatorname{Supp}\lef(b\rig)\smallsetminus P$. Remark that if $i\in Q$, then $a_{i}$ has to be the remainder of the division of $b_{i}$ by $p^{u}$. However, if $i\in P$, then $a_{i}\in\lef\{0,p^{u}\rig\}$. Assume that we have chosen $l\in\lef\llbracket0,m\rig\rrbracket$ elements $i\in P$ to verify $a_{i}=p^{u}$; in other words, such that $i\in I$. Then we have $\binom{\#P}{l}$ choices for such elements. We are left to pick $m-l$ elements in $Q$ to belong in $I$, knowing that we cannot select the smallest $i$ for the order $\prec$ as $I_{0}\neq\emptyset$. That is, we still have $\binom{\#Q-1}{m-l}$ elements to choose. Therefore, the $W\lef(k\rig)$-generating set has a cardinality of $\sum_{l=0}^{m}\binom{\#P}{l}\binom{\#Q-1}{m-l}$.

By the Chu--Vandermonde identity, this $W\lef(k\rig)$-generating set also has cardinality $\binom{\#\operatorname{Supp}\lef(b\rig)-1}{m}$, so it implies that it is a $W\lef(k\rig)$-basis and that the writing is unique \cite[05G8]{stacksproject}.
\end{proof}

Now that this analysis is done, let us notice how the ring of Witt vectors decomposes modulo $p$.

\begin{prop}\label{splittingwittvectorsmodp}
Assume that $\overline{W\lef(k\rig)}\to\overline{k}$ has a section. Then, we have an isomorphism of $\overline{k}$-modules for all $u\in\mathbb{N}$:
\begin{equation*}
\overline{W_{u+1}\lef(k\rig)}\cong\overline{W_{u}\lef(k\rig)}\oplus\overline{V}^{u}\lef(\overline{W_{1}\lef(k\rig)}\rig)\text{.}
\end{equation*}
\end{prop}

\begin{proof}
We proceed by induction on $u\in\mathbb{N}^{*}$, the direct sum in the case $u=1$ being clear. So let us assume that the proposition is shown for $u\in\mathbb{N}^{*}$. By hypothesis, we have $\overline{W_{u+2}\lef(k\rig)}\cong\overline{W_{1}\lef(k\rig)}\oplus\overline{V}\lef(\overline{W_{u+1}\lef(k\rig)}\rig)$. We thus only have to make sure that when $\sum_{i=1}^{u+1}V^{i}\lef(x\lef(i\rig)\rig)=pw$, with $\lef(x\lef(i\rig)\rig)_{i\in\lef\llbracket1,u+1\rig\rrbracket}\in{W_{u+1}\lef(k\rig)}^{u+1}$ and $w\in W_{u+2}\lef(k\rig)$, then all of the $V^{i}\lef(x\lef(i\rig)\rig)$ are divisible by $p$.

As before, we see that we must have $pw_{0}=0$ and ${x\lef(1\rig)}_{0}\equiv{w_{0}}^{p}\pmod{p}$. This implies that $x\lef(1\rig)=\lef[w_{0}\rig]^{p}+pr$ for some $r\in W_{u+1}\lef(k\rig)$. As $pw_{0}=0$, we deduce from the construction of $F$ on Witt vectors that $p\lef[w_{0}\rig]=V\lef(\lef[w_{0}\rig]^{p}\rig)$, so that $p\lef(\lef[w_{0}\rig]+V\lef(r\rig)\rig)=V\lef(x\lef(1\rig)\rig)$.

We are thus get $\sum_{i=2}^{u}V^{i}\lef(c\lef(i\rig)\rig)=pw'$, with $w'=w-\lef[w_{0}\rig]-V\lef(r\rig)\in V\lef(W_{u+1}\lef(k\rig)\rig)$. By injectivity of $V$, we conclude by the induction hypothesis.
\end{proof}

\begin{deff}
We shall say that $k$ has a \textbf{rebar cage} if $\overline{W\lef(k\rig)}\to\overline{k}$ has a section and there exists a $p$-torsion free and $p$-adically complete and separated commutative ring of characteristic zero $C$, as well as a map $C\to W\lef(k\rig)$ inducing a surjective morphism of rings $C\to k$ which is an isomorphism modulo $p$, and such that for each $u\in\mathbb{N}$ and each $m\in\mathbb{N}$ there is a subset $\mathcal{B}\lef(u,m\rig)\subset C$ satisfying:
\begin{itemize}
\item the set $\lef\{\overline{V}^{m}\lef(\overline{b}\rig)\rig\}_{b\in\mathcal{B}\lef(u,m\rig)}$ is a ${\operatorname{Frob}_{\overline{k}}}^{u}\lef(\overline{k}\rig)$-generating set of $\overline{V}^{m}\lef(\overline{W_{1}\lef(k\rig)}\rig)$;
\item the set $\lef\{\overline{b}\rig\}_{b\in\mathcal{B}\lef(u,m\rig)}$ is a ${\operatorname{Frob}_{\overline{k}}}^{u+m}\lef(\overline{k}\rig)$-free family of $\overline{k}$.
\end{itemize}
\end{deff}

These two conditions are equivalent to saying that every $x\in\overline{V}^{m}\lef(\overline{W_{1}\lef(k\rig)}\rig)$ has a unique writing of the form $x=\overline{V}^{m}\lef(\sum_{b\in\mathcal{B}\lef(u,m\rig)}x_{b}\overline{b}\rig)$ where the $x_{b}\in{\operatorname{Frob}_{\overline{k}}}^{u+m}\lef(\overline{k}\rig)$ are almost all nought.

Notice that the second condition implies that the image of $\mathcal{B}\lef(u,m\rig)$ in $W\lef(k\rig)$ is not in $V\lef(W\lef(k\rig)\rig)$. Also, by proposition \ref{splittingwittvectorsmodp}, for every $b\in\mathcal{B}\lef(u,m\rig)$, the element $V^{m}\lef(b\rig)$ is in the image of $C$ if and only if $m=0$.

\begin{xmpl}\label{perfectrebarcage}
When $k$ is a perfect ring of characteristic $p$, and $C=W\lef(k\rig)$. Then, one simply has to take:
\begin{equation*}
\mathcal{B}\lef(u,m\rig)=\begin{cases}\lef\{1\rig\}&\text{if }m=0\text{,}\\\emptyset&\text{otherwise.}\end{cases}
\end{equation*}
\end{xmpl}

\begin{xmpl}\label{laurentrebarcage}
When $k=l\lef(\lef(T\rig)\rig)$ is the field of Laurent series over a perfect field $l$ of characteristic $p$, then one can let $C$ be a Cohen ring with residue field $k$ \cite[0328]{stacksproject}. If we also denote by $T\in C$ an element whose image in $k$ is $T$, we can define the following rebar cage on $k$:
\begin{equation*}
\mathcal{B}\lef(u,m\rig)=\begin{cases}\lef\{T^{i}\rig\}_{i\in\lef\llbracket0,p^{u}-1\rig\rrbracket}&\text{if }m=0\text{,}\\\lef\{T^{i}\rig\}_{\substack{i\in\lef\llbracket0,p^{u+m}-1\rig\rrbracket\\p^{m}\nmid i}}&\text{otherwise.}\end{cases}
\end{equation*}
\end{xmpl}

\begin{xmpl}\label{semiperfectrebarcage}
Assume that $\overline{k}$ is semiperfect, and that $k$ is a $p$-torsion free and $p$-adically complete and separated $\mathbb{Z}_{\lef\langle p\rig\rangle}$-algebra. For instance, $k$ could be either an integral perfectoid ring, or the ring of Witt vectors associated to a semiperfect ring of characteristic $p$. Then, we can choose $C=k$, use the map $\lef[\bullet\rig]\colon C\to W\lef(k\rig)$ and take:
\begin{equation*}
\mathcal{B}\lef(u,m\rig)=\lef\{1\rig\}\text{.}
\end{equation*}
\end{xmpl}

As one can notice, $\mathbb{Z}_{\lef\langle p\rig\rangle}$-algebras which have a rebar cage include many base rings useful in arithmetic geometry. We are going to use these cages in order to construct our structure theorem for the de Rham--Witt complex. Our bricks will be the sub-$\overline{k}\lef[\underline{X}\rig]$-modules of $\overline{W\Omega_{k\lef[\underline{X}\rig]/k}}$ below.

Assume that $k$ has a rebar cage, and assume given a map $L\to W\lef(k\lef[\underline{X}\rig]\rig)$ extending $C\to W\lef(k\rig)$, yielding a surjective ring morphism $L\to k\lef[\underline{X}\rig]$ and an isomorphism of $\overline{k}$-algebras $\overline{L}\cong\overline{k}\lef[\underline{X}\rig]$. Let $b\in\mathcal{B}\lef(u,m\rig)$ and let $\lef(a,I\rig)\in\mathcal{P}$ be a weight function and partition couple. We shall put:
\begin{equation*}
\mathcal{M}\lef(b,a,I\rig)\coloneqq\lef\{\overline{e\lef(V^{m}\lef(b\rig)P^{p^{u}},a,I\rig)}\mid P\in L\rig\}\text{.}
\end{equation*}

This depends on $u$ and $m$, but we omit them from the notation as $b\in\mathcal{B}\lef(u,m\rig)$.

\begin{prop}\label{generalmodpstructure}
Assume that $k$ has a rebar cage, and that we have a map $L\to W\lef(k\lef[\underline{X}\rig]\rig)$ as above.

For any subset $J\subset\lef\llbracket1,n\rig\rrbracket$, denote by $\chi_{J}$ the indicator function. Let $t\in\mathbb{N}$. Then, the epimorphism $\overline{W_{m+1}\Omega^{t}_{k\lef[\underline{X}\rig]/k}}\to\overline{W_{m}\Omega^{t}_{k\lef[\underline{X}\rig]/k}}$ splits, and we have an isomorphism of $\overline{k}\lef[\underline{X}\rig]$-modules for each $m\in\mathbb{N}$:
\begin{multline*}
\operatorname{Ker}\lef(\overline{W_{m+1}\Omega^{t}_{k\lef[\underline{X}\rig]/k}}\to\overline{W_{m}\Omega^{t}_{k\lef[\underline{X}\rig]/k}}\rig)\cong\bigoplus_{\substack{\lef(\chi_{I},I\rig)\in\mathcal{P}\\\#I=t}}\bigoplus_{b\in\mathcal{B}\lef(0,m\rig)}\mathcal{M}\lef(b,a,\operatorname{Supp}\lef(a\rig)\rig)\\\oplus\bigoplus_{u=1}^{m}\bigoplus_{\substack{\lef(a,I\rig)\in\mathcal{P}\\\operatorname{v}_{p}\lef(a\rig)=0\\\forall i\in\lef\llbracket1,n\rig\rrbracket,\ a_{i}<p^{u}\\I_{0}\neq\emptyset\\\#I\leqslant t}}\bigoplus_{\substack{J\subset\lef\llbracket1,n\rig\rrbracket\smallsetminus\operatorname{Supp}\lef(a\rig)\\\#J=t-\#I}}\bigoplus_{b\in\mathcal{B}\lef(u,m-u\rig)}\mathcal{M}\lef(b,p^{-u}a+\chi_{J},I\cup J\rig)\\\oplus d\lef(\bigoplus_{u=1}^{m}\bigoplus_{\substack{\lef(a,I\rig)\in\mathcal{P}\\\operatorname{v}_{p}\lef(a\rig)=0\\\forall i\in\lef\llbracket1,n\rig\rrbracket,\ a_{i}<p^{u}\\I_{0}\neq\emptyset\\\#I\leqslant t-1}}\bigoplus_{\substack{J\subset\lef\llbracket1,n\rig\rrbracket\smallsetminus\operatorname{Supp}\lef(a\rig)\\\#J=t-1-\#I}}\bigoplus_{b\in\mathcal{B}\lef(u,m-u\rig)}\mathcal{M}\lef(b,p^{-u}a+\chi_{J},I\cup J\rig)\rig)\text{.}
\end{multline*}
\end{prop}

\begin{proof}
By theorem \ref{structuretheorem} and proposition \ref{splittingwittvectorsmodp}, the epimorphism splits.

We are going to show that the first line in the above isomorphism corresponds to the integral part of the complex as in \eqref{intfracdecomposition}, that the second one to the pure fractional part, and that the third one to the image through $d$ of the pure fractional part as in \eqref{frpdfrpdecomposition}. First, apply proposition \ref{splittingwittvectorsmodp} and \cite[corollary 2.13]{derhamwittcohomologyforaproperandsmoothmorphism} to have such a decomposition as a $\overline{W\lef(k\rig)}$-module.

We have $\overline{W_{1}\Omega^{t}_{k\lef[\underline{X}\rig]/k}}\cong\Omega^{t}_{\overline{k}\lef[\underline{X}\rig]/\overline{k}}$ as $\overline{k}\lef[\underline{X}\rig]$-modules, so we find that the integral part is also isomorphic to the first line.

Notice that if we know that the second line is isomorphic to the pure fractional part, we can immediately conclude the proof. So let us focus on that second line.

Let $\lef(c,L\rig)\in\mathcal{P}$. Recall the definition:
\begin{equation*}
u\lef(c\rig)=\max\lef\{0,-\operatorname{v}_{p}\lef(c\rig)\rig\}\text{.}
\end{equation*}

Assume that $u\lef(c\rig)\in\lef\llbracket1,m\rig\rrbracket$. Let $\eta\in V^{m-u\lef(c\rig)}\lef(W\lef(k\rig)\rig)$ be a Witt vector such that $e\lef(\eta,c,L\rig)\in W\Omega_{k\lef[\underline{X}\rig]/k}^{\mathrm{frp},t}$. In other terms, $u\lef(c\rig)\neq0$ and $L_{0}\neq\emptyset$. By proposition \ref{vactionone} we have:
\begin{equation*}
e\lef(\eta,c,L\rig)=V^{u\lef(c\rig)}\lef(e\lef(\eta,p^{u\lef(c\rig)}c,L\rig)\rig)\text{.}
\end{equation*}

Use proposition \ref{linearpuintegral} to write uniquely $\overline{e\lef(\eta,p^{u\lef(c\rig)}c,L\rig)}$ as a $\mathbb{F}_{p}$-linear combination of elements of the form $\overline{\lef[\underline{X}^{p^{u\lef(c\rig)}c'}\rig]e\lef(\eta,a',I'\rig)}$, where $a'$ and $c'$ are weight functions taking values in $\mathbb{N}$ with $p^{u\lef(c\rig)}\lef\lvert c\rig\rvert=\lef\lvert a'\rig\rvert+p^{u\lef(c\rig)}\lef\lvert c'\rig\rvert$, $\operatorname{v}_{p}\lef(a'\rig)=0$ and ${I'}_{0}\neq\emptyset$, and where $I'$ is a partition of $a'$ such that $\lef(a',I'\rig)$ does not satisfy \eqref{multbypumodpone}, nor \eqref{multbypumodptwo}, nor \eqref{multbypumodpthree}.

Furthermore, \eqref{frobeniuswittvectors} yields $\lef[\underline{X}^{p^{u\lef(c\rig)}c'}\rig]=F^{u\lef(c\rig)}\lef(\lef[\underline{X}^{c'}\rig]\rig)$. So applying \eqref{vxfyvxy} twice and using proposition \ref{vactionone} again, we get:
\begin{equation*}
V^{u\lef(c\rig)}\lef(\lef[\underline{X}^{p^{u\lef(c\rig)}c'}\rig]e\lef(\eta,a',I'\rig)\rig)=e\lef(\lef[\underline{X}^{p^{u\lef(c\rig)}c'}\rig]\eta,p^{-u\lef(c\rig)}a',I'\rig)\text{.}
\end{equation*}

It remains to determine that $\lef(p^{-u\lef(c\rig)}a',I'\rig)$ is of the form $\lef(p^{-u\lef(c\rig)}a+\chi_{J},I\cup J\rig)$ as in the statement. To see this, recall that $\lef(a',I'\rig)$ does not satisfy the three above-mentioned conditions. It is equivalent to saying that that for all $i\in\lef\llbracket1,n\rig\rrbracket$ we have $a_{i}\leqslant p^{u\lef(c\rig)}$, and if $a_{i}=p^{u\lef(c\rig)}$ then $i\in I'$.
\end{proof}

For the remainder of this section, we assume that $k$ has a rebar cage. Let $t\in\mathbb{N}$ and $m\in\mathbb{N}$. For any subset $J\subset\lef\llbracket1,n\rig\rrbracket$, denote by $\chi_{J}$ the indicator function. Proposition \ref{generalmodpstructure} motivates the following definition:
\begin{align*}
G\lef(t,m\rig)&\coloneqq\lef\{e\lef(V^{m-u}\lef(b\right),\frac{a+p^{u}\chi_{J}}{p^{u}},I\cup J\rig)\mid\begin{array}{c}u\in\lef\llbracket1,m\rig\rrbracket,\ \lef(a,I\rig)\in\mathcal{P},\\J\subset\lef\llbracket1,n\rig\rrbracket\smallsetminus\operatorname{Supp}\lef(a\rig),\\\operatorname{v}_{p}\lef(a\rig)=0,\\\forall i\in\lef\llbracket1,n\rig\rrbracket,\ a_{i}<p^{u},\\I_{0}\neq\emptyset,\ \#I+\#J=t,\\b\in\mathcal{B}\lef(u,m-u\rig)\end{array}\rig\}\text{,}\\
H\lef(t,m\rig)&\coloneqq\lef\{e\lef(V^{m}\lef(b\rig),\chi_{I},I\rig)\mid\begin{array}{c}\lef(\chi_{I},I\rig)\in\mathcal{P},\\\#I=t,\\b\in\mathcal{B}\lef(0,m\rig)\end{array}\rig\}\text{.}
\end{align*}

These are subsets of $W\Omega^{t}_{k\lef[\underline{X}\rig]/k}$. By convention, we also put:
\begin{equation*}
G\lef(-1,m\rig)\coloneqq\emptyset\text{.}
\end{equation*}

For every $e\in G\lef(t,m\rig)\sqcup H\lef(t,m\rig)$, we shall denote by $\eta\lef(e\rig)\in W\lef(k\rig)\smallsetminus V\lef(W\lef(k\rig)\rig)$, by $u\lef(e\rig)\in\mathbb{N}$, and by $\lef(a\lef(e\rig),I\lef(e\rig)\rig)\in\mathcal{P}$ the weight function and partition satisfying:
\begin{equation}\label{enotation}
e=e\lef(V^{u\lef(e\rig)}\lef(\eta\lef(e\rig)\rig),a\lef(e\rig),I\lef(e\rig)\rig)\text{.}
\end{equation}

We now are in position to express a local structure proposition of the de Rham--Witt complex modulo $p$. For simplicity, we only state it in the case where $\overline{k}$ is reduced.

\begin{prop}\label{truncatedderhamwittlocalstructure}
Assume that $k$ has a rebar cage, and that $\overline{k}$ is reduced. Let $R$ be a commutative étale $k\lef[\underline{X}\rig]$-algebra. Then, the epimorphism $\overline{W_{m+1}\Omega^{t}_{R/k}}\to\overline{W_{m}\Omega^{t}_{R/k}}$ splits, and we have an isomorphism of $\overline{R}$-modules for each $t\in\mathbb{N}$ and each $m\in\mathbb{N}$:
\begin{multline*}
\operatorname{Ker}\lef(\overline{W_{m+1}\Omega^{t}_{R/k}}\to\overline{W_{m}\Omega^{t}_{R/k}}\rig)\\\cong\bigoplus_{e\in H\lef(t,m\rig)}\overline{e}\overline{R}\oplus\bigoplus_{e\in G\lef(t,m\rig)}\overline{e}\overline{R}\oplus d\lef(\bigoplus_{e\in G\lef(t-1,m\rig)}\overline{e}\overline{R}\rig)\text{.}
\end{multline*}

We can be more precise when $R=k\lef[\underline{X}\rig]$. Let $\lef(c,L\rig)\in\mathcal{P}$. When $u\lef(c\rig)\neq0$, assume that $L_{0}\neq\emptyset$. Let $b\in\mathcal{B}\lef(u\lef(c\rig),m-u\lef(c\rig)\rig)$. We can associate to every $e\in H\lef(t,m\rig)\sqcup G\lef(t,m\rig)$ with $u\lef(a\lef(e\rig)\rig)=u\lef(c\rig)$, $u\lef(e\rig)=m-u\lef(c\rig)$ and $\eta\lef(e\rig)=b$ a polynomial $P_{e}\in k\lef[\underline{X}\rig]$ such that:
\begin{gather*}
\deg\lef(P_{e}\rig)+\lef\lvert a\lef(e\rig)\rig\rvert=\lef\lvert c\rig\rvert\text{,}\\
\overline{e\lef(b,c,L\rig)}=\sum_{\substack{e\in H\lef(t,m\rig)\sqcup G\lef(t,m\rig)\\u\lef(a\lef(e\rig)\rig)=u\lef(c\rig)\\u\lef(e\rig)=m-u\lef(c\rig)\\\eta\lef(e\rig)=b}}\overline{\lef[P_{e}\rig]e}\text{.}
\end{gather*}
\end{prop}

\begin{proof}
By the étale base change property of the truncated de Rham--Witt complex \cite[lemma 10.8]{integral}, we can assume that $R=k\lef[\underline{X}\rig]$. Then, one can always construct a map $L\to W\lef(k\lef[\underline{X}\rig]\rig)$ in order to apply proposition \ref{generalmodpstructure} to have a basis as a $\overline{k}$-module, which allows us to decompose the module into three parts.

Then, the statement is mainly a reformulation of proposition \ref{linearpuintegral}, we only use the hypothesis that $\overline{k}$ is reduced to prove that for all $e\in G\lef(t,m\rig)$ the $\overline{k}\lef[\underline{X}\rig]$-module $\mathcal{M}\lef(\eta\lef(e\rig),a\lef(e\rig),I\lef(e\rig)\rig)$ is free of rank $1$ with generator $\overline{e}$ using \eqref{vxfyvxy}; and similarly with $e\in H\lef(t,m\rig)$ and $e\in G\lef(t-1,m\rig)$. Thus, only the last part of the statement needs a proof.

But that statement is immediate from either \eqref{dproducts} in the integral case, or from the proof of proposition \ref{generalmodpstructure} by noticing that $\lef\lvert a'\rig\rvert=\lef\lvert a\rig\rvert+p^{u\lef(c\rig)}\#J$.
\end{proof}

\section{Local structure of the de Rham--Witt complex}

The goal of this section is to state various general structure theorems for the de Rham--Witt complex.

In this section, $k$ is a $p$-adically complete and separated commutative $\mathbb{Z}_{\lef\langle p\rig\rangle}$-algebra. In that case, $W\lef(k\rig)$ is also $p$-adically complete and separated \cite[proposition 3]{thedisplayofaformal}. Assume moreover that $W\lef(k\rig)$ has no $p$-torsion. We keep $n\in\mathbb{N}$ and the notation $k\lef[\underline{X}\rig]=k\lef[X_{1},\ldots,X_{n}\rig]$.

In all that follows, we further assume given a complete and separated commutative ring of characteristic zero $C$, as well as a map $C\to W\lef(k\rig)$ inducing a surjective morphism of rings $C\to k$ which is an isomorphism modulo $p$, such that that $k$ and has a rebar cage with this data.

Unless where otherwise stated, the map $C\to W\lef(k\rig)$ may not be a morphism of rings.

We shall also let $R$ be an étale commutative $k\lef[\underline{X}\rig]$-algebra. Let $L$ be an étale commutative $C\lef[\underline{X}\rig]$-algebra lifting $R$, that is such that $L\otimes_{C\lef[\underline{X}\rig]}k\lef[\underline{X}\rig]\cong R$ as $C\lef[\underline{X}\rig]$-algebras. Such a lift always exists \cite[04D1]{stacksproject}.

Denote by $\widehat{R}$ and $\widehat{L}$ the respective $p$-adic completions of $R$ and $L$. We let $\widehat{L}\to W\lef(\widehat{R}\rig)$ be a map such that we have the following commutative diagram, in which all arrows except the exterior diagonal ones are morphisms of rings, and where all vertical arrows are surjective:
\begin{equation*}
\begin{tikzcd}
&C\arrow[ddd]\arrow[rr]\arrow[ddddl]\arrow[dr]&&C\lef[\underline{X}\rig]\arrow[ddd]\arrow[r]&L\arrow[ddd]\arrow[r]&\widehat{L}\arrow[ddd]\arrow[ddddr]&\\
&&\overline{C}\arrow[d,"\rotatebox{270}{$\simeq$}"]&&&&\\
&&\overline{k}&&&&\\
&k\arrow[rr]\arrow[ru]&&k\lef[\underline{X}\rig]\arrow[r]&R\arrow[r]&\widehat{R}&\\
W\lef(k\rig)\arrow[ru]\arrow[rrrrrr]&&&&&&W\lef(\widehat{R}\rig)\arrow[lu]
\end{tikzcd}
\end{equation*}

Let us first focus on the $p$-torsion free setting. This setting includes the one given by example \ref{semiperfectrebarcage}. In that case, in the above diagram the two smaller horizontal lines are the same.

\begin{prop}\label{wittcompletionandcompletionwitt}
Let $m\in\mathbb{N}$. Assume that $k$ is $p$-torsion free. We have a canonical isomorphism of rings:
\begin{equation*}
\widehat{W_{m}\lef(R\rig)}\to W_{m}\lef(\widehat{R}\rig)\text{.}
\end{equation*}

Furthermore, when we choose $C=k$ and $\lef[\bullet\rig]\colon C\to W\lef(k\rig)$ to be the map of the rebar cage of $k$, then $\widehat{R}\cong\widehat{L}$.
\end{prop}

\begin{proof}
We have $V^{m}\lef(W\lef(\widehat{R}\rig)\rig)=\cap_{i\in\mathbb{N}}\lef(V^{m}\lef(W\lef(\widehat{R}\rig)\rig)+p^{i}W\lef(\widehat{R}\rig)\rig)$. Indeed, notice that for each Witt vector $w$ in the intersection, for each $j\in\lef\llbracket0,m-1\rig\rrbracket$ then $w_{j}$ is divisible by all powers of $p$. This gives the equality because $\widehat{R}$ is $p$-adically separated.

The map in the statement is given by applying the completion functor to the inclusion map $W_{m}\lef(R\rig)\to W_{m}\lef(\widehat{R}\rig)$. As the target is $p$-adically complete and separated by \cite[proposition 3]{thedisplayofaformal} and \cite[031A]{stacksproject}, it does not change. We can then conclude by applying \cite[lemma 1.1.2]{onthetopologicalcyclichomologyofthealgebraicclosureofalocalfield}, because $R$ is a flat $k$-algebra, so it also has no $p$-torsion.

The second part of the statement proves itself.
\end{proof}

\begin{deff}
Let $m\in\mathbb{N}$. The \textbf{continuous de Rham--Witt complex} of length $m$ of the morphism $k\to R$ is:
\begin{equation*}
W_{m}\Omega^{\mathrm{cont}}_{R/k}\coloneqq\varprojlim_{i\in\mathbb{N}}W_{m}\Omega_{R/k}/p^{i}\text{.}
\end{equation*}
\end{deff}

The continuous de Rham--Witt complex has been introduced in \cite[definition 10.11]{integral}, in order to prove an integral analogue of Fontaine's conjecture. In that paper, they use an explicit description of the complex for the proof of some of their theorems. The following proposition gives another description of the complex, but we study the $\widehat{L}$-module structure of the complex, instead of the $W\lef(k\rig)$-module one.

\begin{prop}\label{structurecaracteristiczero}
Let $m\in\mathbb{N}^{*}$. Assume that $k$ has no $p$-torsion. Let $C=k$ and $\lef[\bullet\rig]\colon C\to W\lef(k\rig)$ to be the map of the rebar cage of $k$. Assume also that either $\overline{k}$ is reduced, or $R=k\lef[\underline{X}\rig]$. Then, $\widehat{R}\cong\widehat{L}$ as rings, and we can project the map $\lef[\bullet\rig]$ so that it takes its values in $W_{m}\Omega^{\mathrm{cont},0}_{R/k}$. Moreover, for every $t\in\mathbb{N}$ and every $x\in W_{m}\Omega^{\mathrm{cont},t}_{R/k}$, there exists a unique map:
\begin{equation*}
l\colon\begin{array}{rl}\bigoplus_{r=0}^{m-1}H\lef(t,r\rig)\sqcup G\lef(t,r\rig)\sqcup G\lef(t-1,r\rig)\to&\widehat{L}\\e\mapsto&l_{e}\end{array}\text{,}
\end{equation*}
such that $l_{e}=\sum_{i\in\mathbb{N}}p^{i}{\lambda_{i}}^{p^{r}}$, for some not necessarily unique $\lambda_{i}\in\widehat{L}$ for every $i\in\mathbb{N}$, when $e\in H\lef(t,r\rig)\sqcup G\lef(t,r\rig)\sqcup G\lef(t-1,r\rig)$, and:
\begin{multline*}
x=\sum_{r=0}^{m-1}\sum_{e\in H\lef(t,r\rig)}e\lef(V^{r}\lef(\eta\lef(e\rig)\lef[l_{e}\rig]\rig),a\lef(e\rig),I\lef(e\rig)\rig)\\+\sum_{r=0}^{m-1}\sum_{e\in G\lef(t,r\rig)}e\lef(V^{u\lef(e\rig)}\lef(\eta\lef(e\rig)\lef[l_{e}\rig]\rig),a\lef(e\rig),I\lef(e\rig)\rig)\\+d\lef(\sum_{r=0}^{m-1}\sum_{e\in G\lef(t-1,r\rig)}e\lef(V^{u\lef(e\rig)}\lef(\eta\lef(e\rig)\lef[l_{e}\rig]\rig),a\lef(e\rig),I\lef(e\rig)\rig)\rig)\text{.}
\end{multline*}
\end{prop}

\begin{proof}
We have $\widehat{W_{m}\lef(R\rig)}\cong W_{m}\Omega^{0,\mathrm{cont}}_{R/k}$. Since $\lef[\bullet\rig]$ takes values in $W\lef(\widehat{R}\rig)$, we can project it to $W_{m}\lef(\widehat{R}\rig)$ and then apply proposition \ref{wittcompletionandcompletionwitt} to get a map with values in $W_{m}\Omega^{0,\mathrm{cont}}_{R/k}$.

Now, the conclusion is a consequence of either proposition \ref{generalmodpstructure} or \ref{truncatedderhamwittlocalstructure} depending on the hypothesis made, and of \eqref{vxfyvxy}. Indeed, we have such an identification modulo $p$, and as $\widehat{L}$ and $W_{m}\Omega^{\mathrm{cont}}_{R/k}$ are both $p$-adically complete and separated, and the latter has no $p$-torsion, we are finished.
\end{proof}

The proposition enables us to give a new description of the complex studied by Bhatt, Morrow and Scholze.

\begin{xmpl}
Let $m\in\mathbb{N}^{*}$. Let $k=\mathcal{O}$ be the ring of integers of a complete algebraically closed extension of $\mathbb{Q}_{p}$. Let $\mathfrak{X}=\operatorname{Spf}\lef(\widehat{\mathcal{O}\lef[\underline{X}\rig]}\rig)$. Consider Fontaine's period ring $A_{\mathrm{inf}}\coloneqq W\lef(\varprojlim_{\operatorname{Frob}_{\mathcal{O}/p\mathcal{O}}}\mathcal{O}/p\mathcal{O}\rig)$. There is a natural ring morphism $\theta_{m}\colon A_{\mathrm{inf}}\to W_{m}\lef(\mathcal{O}\rig)$. For details, refer to \cite[lemma 3.2]{integral}.

In this context Bhatt, Morrow and Scholze have introduced a sheaf of $A_{\mathrm{inf}}$-modules $A\Omega_{\mathfrak{X}}$ on $\mathfrak{X}$ such that:
\begin{equation*}
A\Omega_{\mathfrak{X}}\otimes^{\mathbf{L}}_{A_{\mathrm{inf}},\theta_{m}}W_{m}\lef(\mathcal{O}\rig)\cong W_{m}\Omega^{\mathrm{cont}}_{\widehat{\mathcal{O}\lef[\underline{X}\rig]}/\mathcal{O}}\text{.}
\end{equation*}

This formula can be found in the proof of \cite[theorem 14.1]{integral}.

We can apply proposition \ref{structurecaracteristiczero} with $L=R=\mathcal{O}\lef[\underline{X}\rig]$ and thus get the $\widehat{\mathcal{O}\lef[\underline{X}\rig]}$-algebra structure of $W_{m}\Omega^{\mathrm{cont}}_{\widehat{\mathcal{O}\lef[\underline{X}\rig]}/\mathcal{O}}$. Moreover, $k$ has the rebar cage given by example \ref{semiperfectrebarcage}, which is a quite simple one. It can be used to simplify the writing as a convergent series. We omit the details.
\end{xmpl}

The de Rham--Witt complex of Hesselholt and Madsen of an integral perfectoid ring has also been studied thoroughly in \cite{onthederhamwittcomplexoverperfectoidrings}. This complex is closely related to the de Rham--Witt complex of Langer and Zink that we study here \cite[lemma 1.1]{theabsoluteandrelativederhamwittcomplexes}. But as the de Rham--Witt complex in not in general $p$-adically complete, these methods can only bring results modulo $p$.

However, if $k\cong\overline{k}$ is a reduced ring of characteristic $p$, and when the de Rham--Witt complex is $p$-torsion free, for instance in the setting of examples \ref{perfectrebarcage} and \ref{laurentrebarcage}, then we can be more precise.

We shall now assume for the remainder of this section that the map $C\to W\lef(k\rig)$ is in fact a morphism of $\delta$-rings, as it is the case in the aforementioned examples.

By lifting through the formally smooth morphisms $C/p^{m}C\to L/p^{m}L$ for varying $m\in\mathbb{N}^{*}$, we can fix a lift of the Frobenius endomorphism $\operatorname{Frob}_{\overline{R}}\colon\overline{R}\to\overline{R}$ to the $p$-adic completion of $L$:
\begin{equation*}
F\colon\widehat{L}\to\widehat{L}\text{.}
\end{equation*}

Now, notice that $\widehat{L}$ has no $p$-torsion because $L$ is flat over the $p$-torsion ring $C$. In particular, $F$ promotes $\widehat{L}$ to a $\delta$-ring, and by adjunction \cite[théorème 4]{joyal} we get a morphism of $\delta$-rings:
\begin{equation*}
t_{F}\colon\widehat{L}\to W\lef(\widehat{R}\rig)\cong W\lef(R\rig)\text{.}
\end{equation*}

In this article, for any morphism of commutative rings $R\to S$ we shall denote by $\Omega_{R/S}$ the de Rham complex of $S$ over $R$, and we shall write:
\begin{equation*}
\Omega^{\mathrm{sep}}_{R/S}\coloneqq\Omega_{R/S}/\bigcap_{n\in\mathbb{N}}p^{n}\Omega_{R/S}\text{.}
\end{equation*}

The functor, together with the obvious morphisms, associating to any commutative $R$-algebra $S$ the $p$-adically separated alternating $R$-dga $\Omega^{\mathrm{sep}}_{R/S}$ is left adjoint to the restriction functor associating to any $p$-adically separated alternating $R$-dga $T$ the commutative $R$-algebra $T^{0}$.

For instance, by adjunction $t_{F}$ extends to a morphism of $W\lef(k\rig)$-dgas:
\begin{equation*}
\Omega^{\mathrm{sep}}_{\widehat{L}/W\lef(k\rig)}\to W\Omega_{R/k}\text{.}
\end{equation*}

With our constructions, we can consider every element in $H\lef(t,0\rig)$ as an element of $\Omega^{\mathrm{sep}}_{\widehat{L}/W\lef(k\rig)}$ instead, so that we can state the following theorem.

\begin{thrm}\label{structuretheoremconvergent}
Assume that $k\cong\overline{k}$ is a reduced ring of characteristic $p$ and that $W\Omega_{\overline{R}/\overline{k}}$ is $p$-torsion free. Then, for every $t\in\mathbb{N}$ and every $x\in W\Omega^{t}_{\overline{R}/\overline{k}}$, there exists a unique map:
\begin{equation*}
l\colon\begin{array}{rl}\bigsqcup_{m\in\mathbb{N}}\lef(H\lef(t,m\rig)\sqcup G\lef(t,m\rig)\sqcup G\lef(t-1,m\rig)\rig)\to&\widehat{L}\\e\mapsto&l_{e}\end{array}\text{,}
\end{equation*}
such that:
\begin{multline*}
x=\sum_{H\lef(t,0\rig)}t_{F}\lef(l_{e}e\rig)+\sum_{e\in\bigsqcup_{m\in\mathbb{N}^{*}}H\lef(t,m\rig)}t_{F}\lef(l_{e}\rig)e\\+\sum_{e\in\bigsqcup_{m\in\mathbb{N}^{*}}G\lef(t,m\rig)}t_{F}\lef(l_{e}\rig)e+d\lef(\sum_{e\in\bigsqcup_{m\in\mathbb{N}^{*}}G\lef(t-1,m\rig)}t_{F}\lef(l_{e}\rig)e\rig)\text{.}
\end{multline*}

Moreover, for all $u\in\mathbb{N}^{*}$ we have that $x\in\operatorname{Fil}^{u}W\Omega^{t}_{\overline{R}/\overline{k}}$ if and only if for all $m\in\lef\llbracket0,u\rig\rrbracket$ and all $e\in H\lef(t,m\rig)\sqcup G\lef(t,m\rig)\sqcup G\lef(t-1,m\rig)$ we have $p^{u-m-u\lef(e\rig)}\mid l_{e}$.
\end{thrm}

\begin{proof}
As $\overline{k}$ is reduced, we deduce from theorem \ref{structuretheorem} and \cite[corollary 2.13]{derhamwittcohomologyforaproperandsmoothmorphism} there is a $W\lef(k\rig)$-module $F_{m}$ sitting in the following diagram with exact arrow and column of $W\lef(k\rig)$-modules for all $m\in\mathbb{N}^{*}$:
\begin{equation*}
\begin{tikzcd}
&0\arrow[d]&&\\
0\arrow[r]&F_{m}\arrow[r]\arrow[d]&W_{m}\Omega_{\overline{k}\lef[\underline{X}\rig]/\overline{k}}\arrow[r,"\times p"]&W_{m}\Omega_{\overline{k}\lef[\underline{X}\rig]/\overline{k}}\\
&W_{m}\Omega_{\overline{k}\lef[\underline{X}\rig]/\overline{k}}\arrow[d]&&\\
&W_{m-1}\Omega_{\overline{k}\lef[\underline{X}\rig]/\overline{k}}\text{.}
\end{tikzcd}
\end{equation*}

Applying the étale base change property of the de Rham-Witt complex \cite[proposition 1.7]{derhamwittcohomologyforaproperandsmoothmorphism}, we thus find that $W\Omega_{\overline{R}/\overline{k}}$ has no $p$-torsion.

Then, the writing as a convergent series is an immediate consequence of proposition \ref{truncatedderhamwittlocalstructure}. Indeed, as $\widehat{L}$ and $W\Omega_{\overline{R}/\overline{k}}$ are both $p$-adically complete and separated, and the latter has no $p$-torsion, we are finished.

As for the property on the filtration, it is also an immediate consequence of the diagram above, and proposition \ref{truncatedderhamwittlocalstructure} again.
\end{proof}

Notice that we retrieve a decomposition similar to \eqref{intfracdecomposition} and \eqref{frpdfrpdecomposition}. We shall later see how, in some cases, it is actually the same decomposition. For the remainder of this article, our goal will be to show that this decomposition preserves the notion of overconvergence.

\section{The overconvergent de Rham--Witt complex}

Unless otherwise stated, in this section $k$ will always be a commutative ring of characteristic $p$. We are going to recall the definition of the overconvergent de Rham--Witt complex, and extend some known basic results in this general setting. In practice, for the complex to be useful, $k$ should only be Noetherian and perfect, but here these assumptions are not necessary.

Remember that we are using the convention $k\lef[\underline{X}\rig]=k\lef[X_{1},\ldots,X_{n}\rig]$ for a given integer $n\in\mathbb{N}$. In this section, we shall consider $\varphi\colon k\lef[\underline{X}\rig]\to\overline{R}$ a surjective morphism of commutative $k$-algebras.

In this article, many of the proofs are analytic in nature. The following definition will be our main tool in that regard.

\begin{deff}
A \textbf{pseudovaluation} on a ring $A$ is a map $v\colon A\to\mathbb{R}\cup\lef\{+\infty,-\infty\rig\}$ such that:
\begin{gather*}
v\lef(0\rig)=+\infty\text{,}\\
v\lef(1\rig)=0\text{,}\\
\forall a\in A,\ v\lef(a\rig)=v\lef(-a\rig)\text{,}\\
\forall a,b\in A,\ v\lef(a+b\rig)\geqslant\min\lef\{v\lef(a\rig),v\lef(b\rig)\rig\}\text{,}\\
\forall a,b\in A,\ \lef(v\lef(a\rig)\neq-\infty\rig)\wedge\lef(v\lef(b\rig)\neq-\infty\rig)\implies\lef(v\lef(ab\rig)\geqslant v\lef(a\rig)+v\lef(b\rig)\rig)\text{.}
\end{gather*}
\end{deff}

\begin{xmpl}
Let $R$ be a commutative ring of characteristic $p$. Then by \cite[IX. \textsection1 proposition 5]{algebrecommutativechapitres} the following map is a pseudovaluation on the ring of Witt vectors of $R$:
\begin{equation*}
\operatorname{v}_{V}\colon\begin{array}{rl}W\lef(R\rig)\to&\mathbb{R}\cup\lef\{+\infty,-\infty\rig\}\\w\mapsto&\max\lef\{u\in\mathbb{N}\cup\lef\{+\infty\rig\}\mid w\in V^{u}\lef(W\lef(R\rig)\rig)\rig\}\end{array}\text{.}
\end{equation*}
\end{xmpl}

We have seen in theorem \ref{structuretheorem} that any $w\in W\Omega_{k\lef[\underline{X}\rig]/k}$ can be written as a convergent series $w=\sum_{\lef(a,I\rig)\in\mathcal{P}}e\lef(\eta_{a,I},a,I\rig)$. Using this, for any $\varepsilon>0$ we define:
\begin{multline}\label{zetaepsilon}
\zeta_{\varepsilon}\colon\\\begin{array}{rl}W\Omega_{k\lef[\underline{X}\rig]/k}\to&\mathbb{R}\cup\lef\{+\infty,-\infty\rig\}\\w\mapsto&\begin{cases}\inf_{\lef(a,I\rig)\in\mathcal{P}}\lef\{2n\operatorname{v}_{V}\lef(\eta_{a,I}\rig)+\#Iu\lef(a\rig)-\varepsilon\lef\lvert a\rig\rvert\rig\}&\text{if }I_{0}=\emptyset\text{,}\\\inf_{\lef(a,I\rig)\in\mathcal{P}}\lef\{2n\operatorname{v}_{V}\lef(\eta_{a,I}\rig)+\lef(\#I+1\rig)u\lef(a\rig)-\varepsilon\lef\lvert a\rig\rvert\rig\}&\text{if }I_{0}\neq\emptyset\text{.}\end{cases}\end{array}
\end{multline}

When the ring $k$ is reduced, this definition implies that:
\begin{equation}\label{multiplicationbypzeta}
\forall\varepsilon>0,\ \forall w\in W\Omega_{k\lef[\underline{X}\rig]/k},\ \zeta_{\varepsilon}\lef(pw\rig)=\zeta_{\varepsilon}\lef(w\rig)+2n\text{.}
\end{equation}

The main theorem from the author's previous article \cite[theorem 3.17]{pseudovaluationsonthederhamwittcomplex} states that for every $\varepsilon>0$, the map $\zeta_{\varepsilon}$ is a pseudovaluation when $k$ is any commutative ring with characteristic $p$. We invite the reader who wants to understand why these maps are defined like this to refer to that paper.

Besides, we immediately get from proposition \ref{dactionone} the following inequality:
\begin{equation}\label{dzetaepsilon}
\forall x\in W\Omega_{k\lef[\underline{X}\rig]/k},\ \zeta_{\varepsilon}\lef(d\lef(x\rig)\rig)\geqslant\zeta_{\varepsilon}\lef(x\rig)\text{.}
\end{equation}

It also follows from the definition that if $\lef(w_{u}\rig)_{u\in\mathbb{N}}\in{W\Omega_{k\lef[\underline{X}\rig]/k}}^{\mathbb{N}}$ is a sequence which converges to $0$ for the $p$-adic topology, and such that there exists $\varepsilon>0$ and $M\in\mathbb{R}$ such that $\zeta_{\varepsilon}\lef(w_{u}\rig)\geqslant M$ for all $u\in\mathbb{N}$, then:
\begin{equation}\label{padicoverconvergence}
\zeta_{\varepsilon}\lef(\sum_{u\in\mathbb{N}}w_{u}\rig)\geqslant M\text{.}
\end{equation}

In degree zero -- that is, for Witt vectors -- there was already a pseudovaluation defined by Davis, Langer and Zink. It has a very convenient definition which employs Witt coordinates. We recall it here, as well as some basic properties that can be found in \cite{overconvergentwittvectors}.

In what follows, we will consider real tuples $b=\lef(b_{i}\rig)_{i\in\lef\llbracket1,n\rig\rrbracket}\in{\lef]0,+\infty\rig[}^{n}$ and $c=\lef(c_{i}\rig)_{i\in\lef\llbracket1,n\rig\rrbracket}\in{\lef]0,+\infty\rig[}^{n}$. For all $P=\sum_{j\in\mathbb{N}^{n}}a_{j}\underline{X}^{j}\in k\lef[\underline{X}\rig]$, where the $a_{j}\in k$ for any $j\in\mathbb{N}^{n}$ are almost all nought, we define the following pseudovaluation:
\begin{equation*}
v_{b}\lef(P\rig)\coloneqq\min\lef\{-\sum_{i=1}^{n}b_{i}j_{i}\mid j\in\mathbb{N}^{n},\ b_{j}\neq0\rig\}\text{,}
\end{equation*}
with the convention $v_{b}\lef(0\rig)=+\infty$. Then, of course for every $P\in k\lef[\underline{X}\rig]$ we have $v_{\lef(1,\ldots,1\rig)}\lef(P\rig)=-\deg\lef(P\rig)$.

Consider the following pseudovaluation on $\overline{R}$:
\begin{equation*}
v_{\varphi,b}\colon\begin{array}{rl}\overline{R}\to&\mathbb{R}\cup\lef\{+\infty,-\infty\rig\}\\r\mapsto&\sup\lef\{v_{b}\lef(P\rig)\mid P\in\varphi^{-1}\lef(\lef\{r\rig\}\rig)\rig\}\text{.}\end{array}
\end{equation*}

Let $m\lef(b,c\rig)\coloneqq\min\lef\{\frac{c_{i}}{b_{i}}\mid i\in\lef\llbracket1,n\rig\rrbracket\rig\}$ and $M\lef(b,c\rig)\coloneqq\max\lef\{\frac{c_{i}}{b_{i}}\mid i\in\lef\llbracket1,n\rig\rrbracket\rig\}$ be two positive real numbers. We immediately see that:
\begin{equation}\label{equivalencesofv}
\forall P\in k\lef[\underline{X}\rig],\ m\lef(b,c\rig)v_{\varphi,b}\lef(P\rig)\geqslant v_{\varphi,c}\lef(P\rig)\geqslant M\lef(b,c\rig)v_{\varphi,b}\lef(P\rig)\text{.}
\end{equation}

Now, for any $\varepsilon>0$ we shall define:
\begin{equation*}
\gamma_{\varepsilon,\varphi,b}\colon\begin{array}{rl}W\lef(\overline{R}\rig)\to&\mathbb{R}\cup\lef\{+\infty,-\infty\rig\}\\\lef(w_{u}\rig)_{u\in\mathbb{N}}\mapsto&\inf\lef\{u+\varepsilon p^{-u}v_{\varphi,b}\lef(w_{u}\rig)\mid u\in\mathbb{N}\rig\}\text{.}\end{array}
\end{equation*}

According to \cite[proposition 2.3]{overconvergentwittvectors}, $\gamma_{\varepsilon,\varphi,b}$ is a pseudovaluation on $W\lef(\overline{R}\rig)$. We deduce from \eqref{equivalencesofv} the following inequalities:
\begin{equation}\label{equivalencesofgamma}
\forall w\in W\lef(\overline{R}\rig),\ \gamma_{m\lef(b,c\rig)\varepsilon,\varphi,b}\lef(w\rig)\geqslant\gamma_{\varepsilon,\varphi,c}\lef(w\rig)\geqslant\gamma_{M\lef(b,c\rig)\varepsilon,\varphi,b}\lef(w\rig)\text{.}
\end{equation}

This pseudovaluation is related with the one that we have previously introduced.

\begin{prop}\label{gammaandzetawitt}
Assume that $n\in\mathbb{N}^{*}$. Then for any $w\in W\lef(k\lef[\underline{X}\rig]\rig)$ and any $\varepsilon>0$, we have:
\begin{equation*}
2n\gamma_{\frac{\varepsilon}{2n},\operatorname{Id}_{k\lef[\underline{X}\rig]},\lef(1,\ldots,1\rig)}\lef(w\rig)\geqslant\zeta_{\varepsilon}\lef(w\rig)\geqslant\gamma_{\varepsilon,\operatorname{Id}_{k\lef[\underline{X}\rig]},\lef(1,\ldots,1\rig)}\lef(w\rig)\text{.}
\end{equation*}
\end{prop}

\begin{proof}
Let $w\in W\lef(k\lef[\underline{X}\rig]\rig)$. Using theorem \ref{structuretheorem}, we can write this element as a series:
\begin{equation*}
w=\sum_{a\in{\mathbb{N}\lef[\frac{1}{p}\rig]}^{n}}V^{u\lef(a\rig)}\lef(\eta_{a}\lef[\underline{X}^{p^{u\lef(a\rig)}a}\rig]\rig)\text{,}
\end{equation*}
where $\eta_{a}\in W\lef(k\rig)$ for any $a\in{\mathbb{N}\lef[\frac{1}{p}\rig]}^{n}$. As this series converges for the $V$-adic topology, and since we are dealing with pseudovaluations, it is enough to show the inequality with $w=V^{u\lef(a\rig)}\lef(\eta_{a}\lef[\underline{X}^{p^{u\lef(a\rig)}a}\rig]\rig)$ for some $a\in{\mathbb{N}\lef[\frac{1}{p}\rig]}^{n}$, in which case:
\begin{align*}
\gamma_{\varepsilon,\operatorname{Id}_{k\lef[\underline{X}\rig]},\lef(1,\ldots,1\rig)}\lef(w\rig)&=\operatorname{v}_{V}\lef(\eta_{a}\rig)+u\lef(a\rig)-\varepsilon\lef\lvert a\rig\rvert\text{,}\\
\zeta_{\varepsilon}\lef(w\rig)&=2n\operatorname{v}_{V}\lef(\eta_{a}\rig)+u\lef(a\rig)-\varepsilon\lef\lvert a\rig\rvert\text{,}\\
2n\gamma_{\frac{\varepsilon}{2n},\operatorname{Id}_{k\lef[\underline{X}\rig]},\lef(1,\ldots,1\rig)}\lef(w\rig)&=2n\operatorname{v}_{V}\lef(\eta_{a}\rig)+2nu\lef(a\rig)-\varepsilon\lef\lvert a\rig\rvert\text{.}
\end{align*}
\end{proof}

The definition of this other pseudovaluation makes it easier to understand the behaviour of the \emph{Verschiebung} endomorphism regarding overconvergence.

\begin{lemm}\label{verschiebunggamma}
Let $\overline{R}$, $\varphi$ and $b$ be as above. Let $w\in W\lef(\overline{R}\rig)$ and $u\in\mathbb{N}$. Then:
\begin{equation*}
\gamma_{\varepsilon,\varphi,b}\lef(V^{u}\lef(w\rig)\rig)=u+\gamma_{p^{-u}\varepsilon,\varphi,b}\lef(w\rig)\text{.}
\end{equation*}
\end{lemm}

\begin{proof}
It is enough to prove this in the case where $\overline{R}=k\lef[\underline{X}\rig]$ and $\varphi=\operatorname{Id}_{k\lef[\underline{X}\rig]}$. Then:
\begin{equation*}
\gamma_{\varepsilon,\varphi,b}\lef(V^{u}\lef(w\rig)\rig)=\inf\lef\{u+i+\varepsilon p^{-u-i}v_{\varphi,b}\lef(w_{i}\rig)\mid i\in\mathbb{N}\rig\}=u+\gamma_{p^{-u}\varepsilon,\varphi,b}\lef(w\rig)\text{.}
\end{equation*}
\end{proof}

By using the fact that multiplication by $p$ in Witt vectors over $k\lef[\underline{X}\rig]$ equals $V\circ W\lef(\operatorname{Frob}_{k\lef[\underline{X}\rig]}\rig)$, when $k$ is reduced we can derive the following equality from the previous lemma:
\begin{equation}\label{multiplicationbypgamma}
\forall w\in W\lef(k\lef[\underline{X}\rig]\rig),\ \gamma_{\varepsilon,\operatorname{Id}_{k\lef[\underline{X}\rig]},b}\lef(pw\rig)=\gamma_{\varepsilon,\operatorname{Id}_{k\lef[\underline{X}\rig]},b}\lef(w\rig)+1\text{.}
\end{equation}

In all the remainder of this article, we will be working with the following complex, which was first introduced in \cite{overconvergentderhamwittcohomology}.

\begin{deff}
The \textbf{overconvergent de Rham--Witt complex} is:
\begin{equation*}
W^{\dagger}\Omega_{k\lef[\underline{X}\rig]/k}\coloneqq\lef\{w\in W\Omega_{k\lef[\underline{X}\rig]/k}\mid\exists\varepsilon>0,\ \zeta_{\varepsilon}\lef(w\rig)\neq-\infty\rig\}\text{.}
\end{equation*}
\end{deff}

In the author's previous article, it was shown that this definition coincides with the one given by Davis, Langer and Zink \cite[proposition 3.5]{pseudovaluationsonthederhamwittcomplex}.

We define the following sub-$W\lef(k\rig)$-dga:
\begin{equation*}
W^{\dagger}\Omega^{\mathrm{int}}_{k\lef[\underline{X}\rig]/k}\coloneqq W\Omega^{\mathrm{int}}_{k\lef[\underline{X}\rig]/k}\cap W^{\dagger}\Omega_{k\lef[\underline{X}\rig]/k}\text{,}
\end{equation*}
and the following sub-graded $W\lef(k\rig)$-modules:
\begin{align*}
W^{\dagger}\Omega^{\mathrm{frac}}_{k\lef[\underline{X}\rig]/k}&\coloneqq W\Omega^{\mathrm{frac}}_{k\lef[\underline{X}\rig]/k}\cap W^{\dagger}\Omega_{k\lef[\underline{X}\rig]/k}\text{,}\\
W^{\dagger}\Omega^{\mathrm{frp}}_{k\lef[\underline{X}\rig]/k}&\coloneqq W\Omega^{\mathrm{frp}}_{k\lef[\underline{X}\rig]/k}\cap W^{\dagger}\Omega_{k\lef[\underline{X}\rig]/k}\text{.}
\end{align*}

We immediately deduce two decompositions as graded $W\lef(k\rig)$-modules:
\begin{align}
W^{\dagger}\Omega_{k\lef[\underline{X}\rig]/k}&=W^{\dagger}\Omega^{\mathrm{int}}_{k\lef[\underline{X}\rig]/k}\oplus W^{\dagger}\Omega^{\mathrm{frac}}_{k\lef[\underline{X}\rig]/k}\text{,}\label{decomposeoverconvergent}\\
W^{\dagger}\Omega^{\mathrm{frac}}_{k\lef[\underline{X}\rig]/k}&=W^{\dagger}\Omega^{\mathrm{frp}}_{k\lef[\underline{X}\rig]/k}\oplus d\lef(W^{\dagger}\Omega^{\mathrm{frp}}_{k\lef[\underline{X}\rig]/k}\rig)\text{.}\label{decomposefracoverconvergent}
\end{align}

To get a functor on $k$-algebras of finite type, we first need the following generalisation of \cite[proposition 0.9]{overconvergentderhamwittcohomology}. In this new proof, we will rely on the fact that $\zeta_{\varepsilon}$ is a pseudovaluation.

\begin{prop}\label{overconvergentfunctor}
Let $m\in\mathbb{N}$, and let $\varphi\colon k\lef[\underline{X}\rig]\to k\lef[Y_{1},\ldots,Y_{m}\rig]$ be a morphism of $k$-algebras. Then the image of $W^{\dagger}\Omega_{k\lef[\underline{X}\rig]/k}$ through the morphism of $W\lef(k\rig)$-dgas $W\Omega_{k\lef[\underline{X}\rig]/k}\to W\Omega_{k\lef[Y_{1},\ldots,Y_{m}\rig]/k}$ functorially induced by $\varphi$ is included in the overconvergent subcomplex $W^{\dagger}\Omega_{k\lef[Y_{1},\ldots,Y_{m}\rig]/k}$.

Moreover, if $\varphi$ is surjective, so is $W^{\dagger}\Omega_{k\lef[\underline{X}\rig]/k}\to W^{\dagger}\Omega_{k\lef[Y_{1},\ldots,Y_{m}\rig]/k}$.
\end{prop}

\begin{proof}
Let us start by applying theorem \ref{structuretheorem} to write $w\in W\Omega_{k\lef[\underline{X}\rig]/k}$ as a convergent series:
\begin{equation*}
w=\sum_{\lef(a,I\rig)\in\mathcal{P}}e\lef(\eta_{a,I},a,I\rig)\text{.}
\end{equation*}

It is sufficient to control uniformly the overconvergence of the image through the functorially induced morphism, which we shall also denote $\varphi$, of all terms of this series.

Denote by $M$ the maximum of the degrees of the polynomials $\varphi\lef(X_{l}\rig)$, where $l$ runs through all $\lef\llbracket1,n\rig\rrbracket$. Then, using proposition \ref{gammaandzetawitt}, we get for all $l\in\lef\llbracket1,n\rig\rrbracket$ and all $\varepsilon>0$, with the obvious notations:
\begin{equation}\label{phixoverconvergence}
\begin{split}
\zeta_{\varepsilon}\lef(\lef[\varphi\lef(X_{l}\rig)\rig]\rig)&\geqslant\gamma_{\varepsilon,\operatorname{Id}_{k\lef[Y_{1},\ldots,Y_{m}\rig]},\lef(1,\ldots,1\rig)}\lef(\lef[\varphi\lef(X_{l}\rig)\rig]\rig)\\
&\geqslant\varepsilon v_{\operatorname{Id}_{k\lef[Y_{1},\ldots,Y_{m}\rig]},\lef(1,\ldots,1\rig)}\lef(\varphi\lef(X_{l}\rig)\rig)\\
&\geqslant-\varepsilon M\text{.}
\end{split}
\end{equation}

Consider $\lef(a,I\rig)\in\mathcal{P}$. We will be using the notations introduced with the definition of a partition, and write $I=\lef\{i_{j}\rig\}_{j\in\lef\llbracket1,\#I\rig\rrbracket}$. When either $I_{0}\neq\emptyset$, or $u\lef(a\rig)=0$, then by the very definition of $e$, we have:
\begin{multline*}
e\lef(\eta_{a,I},a,I\rig)=V^{u\lef(a\rig)}\lef(\eta_{a,I}\lef[\underline{X}^{p^{u\lef(a\rig)}a|_{I_{0}}}\rig]\rig)\\\times\prod_{j=1}^{\#I}F^{u\lef(a|_{I_{j}}\rig)+\operatorname{v}_{p}\lef(a|_{I_{j}}\rig)}\lef(d\lef(V^{u\lef(a|_{I_{j}}\rig)}\lef(\lef[\underline{X}^{p^{-\operatorname{v}_{p}\lef(a|_{I_{j}}\rig)}a|_{I_{j}}}\rig]\rig)\rig)\rig)\text{.}
\end{multline*}

We are now going to study the image through the morphism functorially induced by $\varphi$ of this element. We shall proceed factor by factor. First, we get the following inequalities:
\begin{multline}\label{zetaepsilonverschiebung}
\zeta_{\varepsilon}\lef(V^{u\lef(a\rig)}\lef(\eta_{a,I}\lef[\varphi\lef(\underline{X}^{p^{u\lef(a\rig)}a|_{I_{0}}}\rig)\rig]\rig)\rig)\\\begin{aligned}
&\overset{\hphantom{\eqref{phixoverconvergence}}}{\overset{\ref{gammaandzetawitt}}{\geqslant}}\gamma_{\varepsilon,\operatorname{Id}_{k\lef[Y_{1},\ldots,Y_{m}\rig]},\lef(1,\ldots,1\rig)}\lef(V^{u\lef(a\rig)}\lef(\eta_{a,I}\lef[\varphi\lef(\underline{X}^{p^{u\lef(a\rig)}a|_{I_{0}}}\rig)\rig]\rig)\rig)\\
&\overset{\hphantom{\eqref{phixoverconvergence}}}{\overset{\ref{verschiebunggamma}}{\geqslant}}u\lef(a\rig)+\gamma_{p^{-u\lef(a\rig)}\varepsilon,\operatorname{Id}_{k\lef[Y_{1},\ldots,Y_{m}\rig]},\lef(1,\ldots,1\rig)}\lef(\eta_{a,I}\lef[\varphi\lef(\underline{X}^{p^{u\lef(a\rig)}a|_{I_{0}}}\rig)\rig]\rig)\\
&\overset{\eqref{phixoverconvergence}}{\geqslant}u\lef(a\rig)+\operatorname{v}_{V}\lef(\eta_{a,I}\rig)-\varepsilon M\lef\lvert a|_{I_{0}}\rig\rvert\text{.}\end{aligned}
\end{multline}

In the last line, we have also used the fact that $\gamma_{\varepsilon,\operatorname{Id}_{k\lef[Y_{1},\ldots,Y_{m}\rig]},\lef(1,\ldots,1\rig)}$ is a pseudovaluation.

Now, let $j\in\lef\llbracket1,\#I\rig\rrbracket$. If $u\lef(a|_{I_{j}}\rig)=0$, then:
\begin{multline*}
\zeta_{\varepsilon}\lef(F^{u\lef(a|_{I_{j}}\rig)+\operatorname{v}_{p}\lef(a|_{I_{j}}\rig)}\lef(d\lef(V^{u\lef(a|_{I_{j}}\rig)}\lef(\lef[\varphi\lef(\underline{X}^{p^{-\operatorname{v}_{p}\lef(a|_{I_{j}}\rig)}a|_{I_{j}}}\rig)\rig]\rig)\rig)\rig)\rig)\\\begin{aligned}
&\overset{\hphantom{\eqref{phixoverconvergence}}}{=}\zeta_{\varepsilon}\lef(F^{\operatorname{v}_{p}\lef(a|_{I_{j}}\rig)}\lef(d\lef(\lef[\varphi\lef(\underline{X}^{p^{-\operatorname{v}_{p}\lef(a|_{I_{j}}\rig)}a|_{I_{j}}}\rig)\rig]\rig)\rig)\rig)\\
&\overset{\hphantom{\eqref{phixoverconvergence}}}{\overset{\eqref{fdttdt}}{=}}\zeta_{\varepsilon}\lef(\lef[\varphi\lef(\underline{X}^{p^{-\operatorname{v}_{p}\lef(a|_{I_{j}}\rig)}a|_{I_{j}}}\rig)^{p^{\operatorname{v}_{p}\lef(a|_{I_{j}}\rig)}-1}\rig]d\lef(\lef[\varphi\lef(\underline{X}^{p^{-\operatorname{v}_{p}\lef(a|_{I_{j}}\rig)}a|_{I_{j}}}\rig)\rig]\rig)\rig)\\
&\overset{\hphantom{\eqref{phixoverconvergence}}}{\overset{\eqref{dzetaepsilon}}{\geqslant}}\zeta_{\varepsilon}\lef(\lef[\varphi\lef(\underline{X}^{\lef(1-p^{-\operatorname{v}_{p}\lef(a|_{I_{j}}\rig)}\rig)a|_{I_{j}}}\rig)\rig]\rig)+\zeta_{\varepsilon}\lef(\lef[\varphi\lef(\underline{X}^{p^{-\operatorname{v}_{p}\lef(a|_{I_{j}}\rig)}a|_{I_{j}}}\rig)\rig]\rig)\\
&\overset{\eqref{phixoverconvergence}}{\geqslant}-\varepsilon M\lef\lvert a|_{I_{j}}\rig\rvert\text{.}\end{aligned}
\end{multline*}

In the penultimate line, we have also used the fact that $\zeta_{\varepsilon}$ is a pseudovaluation \cite[theorem 3.17]{pseudovaluationsonthederhamwittcomplex}. Moreover, when $u\lef(a|_{I_{j}}\rig)\neq0$, then $u\lef(a|_{I_{j}}\rig)+\operatorname{v}_{p}\lef(a|_{I_{j}}\rig)=0$ so that:
\begin{multline*}
\zeta_{\varepsilon}\lef(F^{u\lef(a|_{I_{j}}\rig)+\operatorname{v}_{p}\lef(a|_{I_{j}}\rig)}\lef(d\lef(V^{u\lef(a|_{I_{j}}\rig)}\lef(\lef[\varphi\lef(\underline{X}^{p^{-\operatorname{v}_{p}\lef(a|_{I_{j}}\rig)}a|_{I_{j}}}\rig)\rig]\rig)\rig)\rig)\rig)\\\begin{aligned}
&\overset{\hphantom{\eqref{zetaepsilonverschiebung}}}{=}\zeta_{\varepsilon}\lef(d\lef(V^{u\lef(a|_{I_{j}}\rig)}\lef(\lef[\varphi\lef(\underline{X}^{p^{u\lef(a|_{I_{j}}\rig)}a|_{I_{j}}}\rig)\rig]\rig)\rig)\rig)\\
&\overset{\hphantom{\eqref{zetaepsilonverschiebung}}}{\overset{\eqref{dzetaepsilon}}{\geqslant}}\zeta_{\varepsilon}\lef(V^{u\lef(a|_{I_{j}}\rig)}\lef(\lef[\varphi\lef(\underline{X}^{p^{u\lef(a|_{I_{j}}\rig)}a|_{I_{j}}}\rig)\rig]\rig)\rig)\\
&\overset{\eqref{zetaepsilonverschiebung}}{\geqslant} u\lef(a|_{I_{j}}\rig)-\varepsilon M\lef\lvert a|_{I_{j}}\rig\rvert\text{.}\end{aligned}
\end{multline*}

Combining these three inequalities, and as $\zeta_{\varepsilon}$ is a pseudovaluation, we find under our assumption that $I_{0}\neq\emptyset$ or $u\lef(a\rig)=0$, and under the new hypothesis $n\neq0$:
\begin{equation*}
\zeta_{\varepsilon}\lef(\varphi\lef(e\lef(\eta_{a,I},a,I\rig)\rig)\rig)\geqslant \operatorname{v}_{V}\lef(\eta_{a,I}\rig)+u\lef(a\rig)-\varepsilon M\lef\lvert a\rig\rvert\geqslant\frac{\zeta_{2n\varepsilon M}\lef(e\lef(\eta_{a,I},a,I\rig)\rig)}{2n}\text{.}
\end{equation*}

In the case where $n=0$, there was nothing to prove in the first place. And finally, when we drop our assumption to get $I_{0}=\emptyset$ and $u\lef(a\rig)\neq0$, then by proposition \ref{dactionone} we have $e\lef(\eta_{a,I},a,I\rig)=d\lef(e\lef(\eta_{a,I},a,I\smallsetminus\lef\{\min\lef(a\rig)\rig\}\rig)\rig)$; so applying \eqref{dzetaepsilon} to our above result ends the proof of the overconvergence of the image because:
\begin{gather*}
\lef(\lef(I_{0}\neq\emptyset\rig)\vee\lef(u\lef(a\rig)=0\rig)\rig)\implies\frac{\zeta_{2n\varepsilon M}\lef(e\lef(\eta_{a,I},a,I\rig)\rig)}{2n}\geqslant\frac{\zeta_{2n\varepsilon M}\lef(w\rig)}{2n}\text{,}\\
\begin{multlined}\lef(\lef(I_{0}=\emptyset\rig)\wedge\lef(u\lef(a\rig)\neq0\rig)\rig)\\\implies\frac{\zeta_{2n\varepsilon M}\lef(e\lef(\eta_{a,I},a,I\smallsetminus\lef\{\min\lef(a\rig)\rig\}\rig)\rig)}{2n}=\frac{\zeta_{2n\varepsilon M}\lef(e\lef(\eta_{a,I},a,I\rig)\rig)}{2n}\geqslant\frac{\zeta_{2n\varepsilon M}\lef(w\rig)}{2n}\text{.}\end{multlined}
\end{gather*}

The second part of the statement, concerning the surjectivity of the morphism induced by $\varphi$, can be shown using exactly the same easy argument as in the proof of \cite[proposition 0.9]{overconvergentderhamwittcohomology}.
\end{proof}

Proposition \ref{overconvergentfunctor} allows us to use the same arguments as in \cite[definition 1.1]{overconvergentderhamwittcohomology}, in which the complex is extended to any smooth commutative $k$-algebra $\overline{R}$ by considering a surjective morphism $\varphi\colon k\lef[\underline{X}\rig]\to\overline{R}$ of $k$-algebras. The complex $W^{\dagger}\Omega_{\overline{R}/k}$ is then defined as the image of $W^{\dagger}\Omega_{k\lef[\underline{X}\rig]/k}$ through the morphism $W\Omega_{k\lef[\underline{X}\rig]/k}\to W\Omega_{\overline{R}/k}$ functorially induced by $\varphi$. The arguments given by Davis, Langer and Zink still hold here word for word, stating that this definition does not depend on $\varphi$, and is functorial.

For all $m\in\mathbb{N}^{*}$, we also define the two-sided ideal:
\begin{equation*}
\operatorname{Fil}^{m}\lef(W^{\dagger}\Omega_{\overline{R}/k}\rig)\coloneqq\operatorname{Fil}^{m}\lef(W\Omega_{\overline{R}/k}\rig)\cap W^{\dagger}\Omega_{\overline{R}/k}\text{.}
\end{equation*}

\section{A characterisation of finite free relatively perfect algebras}

In this section $k$ shall denote a commutative ring of characteristic $p$. As usual, write $k\lef[\underline{X}\rig]=k\lef[X_{1},\ldots,X_{n}\rig]$ for a fixed $n\in\mathbb{N}$. We shall consider $\varphi\colon k\lef[\underline{X}\rig]\to\overline{R}$ a surjective morphism of commutative $k$-algebras.

\begin{deff}
The ring of \textbf{overconvergent Witt vectors over $\overline{R}$} is:
\begin{equation*}
W^{\dagger}\lef(\overline{R}\rig)\coloneqq W^{\dagger}\Omega^{0}_{\overline{R}/k}\text{.}
\end{equation*}
\end{deff}

Notice that this gives rise to a sub Witt functor, as introduced in definition \ref{subwittfunctor}, from the category of finite type commutative $k$-algebras to the category of $\delta$-rings.

The goal of this section will be to give a characterisation of relatively perfect commutative $\overline{R}$-algebras which are finite free as $\overline{R}$-modules using overconvergent Witt vectors.

The pseudovaluations we have studied for the definition of the overconvergent de Rham--Witt complex behave well with regard to finite free algebras.

\begin{lemm}\label{pseudovaluationbasis}
Let $\overline{S}$ be a commutative $\overline{R}$-algebra. Suppose given a $\overline{R}$-basis $\lef(s_{i}\rig)_{i\in\lef\llbracket1,m\rig\rrbracket}$ of $\overline{S}$, where $m\in\mathbb{N}$. Extend $\varphi$ defined above to a surjective morphism of $k\lef[\underline{X}\rig]$-algebras $\varphi\colon k\lef[X_{1},\ldots,X_{n},Y_{1},\ldots,Y_{m}\rig]\to\overline{S}$ such that $\varphi\lef(Y_{i}\rig)=s_{i}$ for every $i\in\lef\llbracket1,m\rig\rrbracket$.

Then, there exists $\lef(b_{i}\rig)_{i\in\lef\llbracket1,n+m\rig\rrbracket}\in{\lef]0,+\infty\rig[}^{n+m}$ and a real number $C\geqslant0$ such that for all $\lef(r_{j}\rig)_{j\in\lef\llbracket1,m\rig\rrbracket}\in\overline{R}^{m}$ we have:
\begin{equation*}
\min\lef\{v_{\varphi|_{k\lef[\underline{X}\rig]},\lef(b_{i}\rig)_{i\in\lef\llbracket1,n\rig\rrbracket}}\lef(r_{j}\rig)\mid j\in\lef\llbracket1,m\rig\rrbracket\rig\}\geqslant v_{\varphi,\lef(b_{i}\rig)_{i\in\lef\llbracket1,n+m\rig\rrbracket}}\lef(\sum_{j=1}^{m}r_{j}s_{j}\rig)-C\text{.}
\end{equation*}
\end{lemm}

\begin{proof}
For any $j',t'\in\lef\llbracket1,m\rig\rrbracket$, let us choose a preimage $e_{j',t'}\in\bigoplus_{t=1}^{m}k\lef[\underline{X}\rig]Y_{t}$ of $s_{j'}s_{t'}$ by $\varphi$, and let $d\in\mathbb{N}^{*}$ be the maximum of these polynomials' degrees. Put $b_{i}\coloneqq1$ for $i\in\lef\llbracket1,n\rig\rrbracket$ and $b_{i}\coloneqq d$ for $i\in\lef\llbracket n+1,n+m\rig\rrbracket$. We thus have:
\begin{equation*}
\forall j',t'\in\lef\llbracket1,m\rig\rrbracket,\ v_{\lef(b_{i}\rig)_{i\in\lef\llbracket1,n+m\rig\rrbracket}}\lef(e_{j',t'}\rig)\geqslant-2d+1\text{.}
\end{equation*}

For any $j\in\lef\llbracket1,m\rig\rrbracket$, put $r_{j}\in\overline{R}$. Recall that $\varphi\lef(k\lef[\underline{X}\rig]\rig)=\overline{R}\subset\overline{S}$, so:
\begin{multline*}
\min\lef\{v_{\varphi|_{k\lef[\underline{X}\rig]},\lef(b_{i}\rig)_{i\in\lef\llbracket1,n\rig\rrbracket}}\lef(r_{j}\rig)\mid j\in\lef\llbracket1,m\rig\rrbracket\rig\}-d\\\begin{aligned}
&=\min\lef\{\sup\lef\{v_{\lef(b_{i}\rig)_{i\in\lef\llbracket1,n+m\rig\rrbracket}}\lef(U\rig)\mid U\in\varphi^{-1}\lef(\lef\{r_{j}\rig\}\rig)\cap k\lef[\underline{X}\rig]\rig\}-b_{n+j}\mid j\in\lef\llbracket1,m\rig\rrbracket\rig\}\\
&=\min\lef\{\sup\lef\{v_{\lef(b_{i}\rig)_{i\in\lef\llbracket1,n+m\rig\rrbracket}}\lef(V\rig)\mid V\in\varphi^{-1}\lef(\lef\{r_{j}s_{j}\rig\}\rig)\cap k\lef[\underline{X}\rig]Y_{j}\rig\}\mid j\in\lef\llbracket1,m\rig\rrbracket\rig\}\\
&=\sup\lef\{v_{\lef(b_{i}\rig)_{i\in\lef\llbracket1,n+m\rig\rrbracket}}\lef(W\rig)\mid W\in\varphi^{-1}\lef(\lef\{\sum_{j=1}^{m}r_{j}s_{j}\rig\}\rig)\cap\bigoplus_{t=1}^{m}k\lef[\underline{X}\rig]Y_{t}\rig\}\text{.}\end{aligned}
\end{multline*}

Let $P\in k\lef[X_{1},\ldots,X_{n},Y_{1},\ldots,Y_{m}\rig]$ be a preimage of $\sum_{j=1}^{m}r_{j}s_{j}$ by $\varphi$. Let us first give another preimage $Q\in k\lef[\underline{X}\rig]\oplus\bigoplus_{t=1}^{m}k\lef[\underline{X}\rig]Y_{t}$ of $\sum_{j=1}^{m}r_{j}s_{j}$ by $\varphi$ such that $v_{\lef(b_{i}\rig)_{i\in\lef\llbracket1,n+m\rig\rrbracket}}\lef(Q\rig)\geqslant v_{\lef(b_{i}\rig)_{i\in\lef\llbracket1,n+m\rig\rrbracket}}\lef(P\rig)$. Such an element can be constructed by induction on $\deg_{\underline{Y}}\lef(P\rig)$.

When $P\in k\lef[\underline{X}\rig]\oplus\bigoplus_{t=1}^{m}k\lef[\underline{X}\rig]Y_{t}$, we can just take $Q=P$. Otherwise, for each monomial $M$ in $P$ such that there exists $j',t'\in\lef\llbracket1,m\rig\rrbracket$ for which $Y_{j'}Y_{t'}\mid M$, we replace an occurrence of $Y_{j'}Y_{t'}$ in $M$ by $e_{j',t'}$, and we do at most one substitution by monomial. The new polynomial $R$ which we get then satisfies $\varphi\lef(R\rig)=\sum_{j=1}^{m}r_{j}s_{j}$ and $v_{\lef(b_{i}\rig)_{i\in\lef\llbracket1,n+m\rig\rrbracket}}\lef(R\rig)\geqslant v_{\lef(b_{i}\rig)_{i\in\lef\llbracket1,n+m\rig\rrbracket}}\lef(P\rig)$ because $v_{\lef(b_{i}\rig)_{i\in\lef\llbracket1,n+m\rig\rrbracket}}$ is a pseudovaluation and:
\begin{equation*}
v_{\lef(b_{i}\rig)_{i\in\lef\llbracket1,n+m\rig\rrbracket}}\lef(Y_{j'}Y_{t'}\rig)=-2d<v_{\lef(b_{i}\rig)_{i\in\lef\llbracket1,n+m\rig\rrbracket}}\lef(e_{j',t'}\rig)\text{.}
\end{equation*}

Moreover we have $\deg_{\underline{Y}}\lef(R\rig)<\deg_{\underline{Y}}\lef(P\rig)$, which ends the induction.

Let us now show that there is $C\geqslant0$ such that for all $R\in k\lef[\underline{X}\rig]$, there is a polynomial $W\in\bigoplus_{t=1}^{m}k\lef[\underline{X}\rig]Y_{t}$ with $v_{\lef(b_{i}\rig)_{i\in\lef\llbracket1,n+m\rig\rrbracket}}\lef(W\rig)\geqslant v_{\lef(b_{i}\rig)_{i\in\lef\llbracket1,n+m\rig\rrbracket}}\lef(R\rig)-C$ and $\varphi\lef(R\rig)=\varphi\lef(W\rig)$. Take $U$ a preimage of $1$ by $\varphi$ contained in $\bigoplus_{t=1}^{m}k\lef[\underline{X}\rig]Y_{t}$, and put $C\coloneqq-v_{\lef(b_{i}\rig)_{i\in\lef\llbracket1,n+m\rig\rrbracket}}\lef(U\rig)\geqslant0$. Then, one can take $W\coloneqq R U$, and this implies the lemma.
\end{proof}

The nice properties we have seen of the above pseudovaluation can be refined in the case of relatively perfect algebras.

\begin{lemm}\label{basispseudovaluationcontrol}
Let $\overline{S}$ be a commutative $\overline{R}$-algebra. Suppose given $m\in\mathbb{N}$ and $\lef(s_{i}\rig)_{i\in\lef\llbracket1,m\rig\rrbracket}\in\overline{S}^{m}$ such that $\lef({s_{i}}^{p^{l}}\rig)_{i\in\lef\llbracket1,m\rig\rrbracket}$ is a $\overline{R}$-basis of $\overline{S}$ for all $l\in\mathbb{N}$. Extend $\varphi$ to a surjective morphism of $k\lef[\underline{X}\rig]$-algebras $\varphi\colon k\lef[X_{1},\ldots,X_{n},Y_{1},\ldots,Y_{m}\rig]\to\overline{S}$ such that $\varphi\lef(Y_{i}\rig)=s_{i}$ for every $i\in\lef\llbracket1,m\rig\rrbracket$.

Then, there exists $\lef(b_{i}\rig)_{i\in\lef\llbracket1,n+m\rig\rrbracket}\in{\lef]0,+\infty\rig[}^{n+m}$ and two real numbers $C,D\geqslant0$ such that for all $s\in\overline{S}$ and all $l\in\mathbb{N}$, if we decompose $s=\sum_{j=1}^{m}r_{j,l}{s_{j}}^{p^{l}}$ with all $r_{j,l}\in\overline{R}$ for each $j\in\lef\llbracket1,m\rig\rrbracket$, we have:
\begin{equation*}
v_{\varphi|_{k\lef[\underline{X}\rig]},\lef(b_{i}\rig)_{i\in\lef\llbracket1,n\rig\rrbracket}}\lef(r_{j,l}\rig)\geqslant v_{\varphi,\lef(b_{i}\rig)_{i\in\lef\llbracket1,n+m\rig\rrbracket}}\lef(s\rig)-C-p^{l}D\text{.}
\end{equation*}
\end{lemm}

\begin{proof}
We will follow closely the reasoning in the proof of \cite[lemma 2.39]{overconvergentwittvectors}. Let $U_{0}=\lef(u_{j,j'}\rig)_{j,j'\in\lef\llbracket1,m\rig\rrbracket}\in\operatorname{GL}_{m}\lef(\overline{R}\rig)$ be the matrix satisfying:
\begin{equation*}
\begin{pmatrix}{s_{1}}\\\vdots\\{s_{m}}\end{pmatrix}=U_{0}\begin{pmatrix}{s_{1}}^{p}\\\vdots\\{s_{m}}^{p}
\end{pmatrix}\text{.}
\end{equation*}

For all $t\in\mathbb{N}$, let $U_{t}\in\operatorname{GL}_{m}\lef(\overline{R}\rig)$ be the matrix one gets by raising each coefficient in $U_{0}$ to the power of $p^{t}$, so that:
\begin{equation*}
\forall l\in\mathbb{N},\ \begin{pmatrix}{s_{1}}\\\vdots\\{s_{m}}\end{pmatrix}=\lef(\prod_{t=0}^{l-1}U_{t}\rig)\begin{pmatrix}{s_{1}}^{p^{l}}\\\vdots\\{s_{m}}^{p^{l}}\end{pmatrix}\text{.}
\end{equation*}

And by transposition:
\begin{equation*}
\forall l\in\mathbb{N},\ \lef(\prod_{t=0}^{l-1}U_{t}\rig)^{\mathbf{T}}\begin{pmatrix}{r_{1,0}}\\\vdots\\{r_{m,0}}\end{pmatrix}=\begin{pmatrix}{r_{1,l}}\\\vdots\\{r_{m,l}}\end{pmatrix}\text{.}
\end{equation*}

Let $\lef(b_{i}\rig)_{i\in\lef\llbracket1,n+m\rig\rrbracket}\in{\lef]0,+\infty\rig[}^{n+m}$ and $C\geqslant0$ the constants given by lemma \ref{pseudovaluationbasis}. Let $d\coloneqq\min\lef\{v_{\varphi|_{k\lef[\underline{X}\rig]},\lef(b_{i}\rig)_{i\in\lef\llbracket1,n\rig\rrbracket}}\lef(u_{j,j'}\rig)\mid j,j'\in\lef\llbracket1,m\rig\rrbracket\rig\}$. For all $j\in\lef\llbracket1,m\rig\rrbracket$ and all $l\in\mathbb{N}$, as $v_{\varphi|_{k\lef[\underline{X}\rig]},\lef(b_{i}\rig)_{i\in\lef\llbracket1,n\rig\rrbracket}}$ is a pseudovaluation we find:
\begin{equation*}
\begin{aligned}
v_{\varphi|_{k\lef[\underline{X}\rig]},\lef(b_{i}\rig)_{i\in\lef\llbracket1,n\rig\rrbracket}}\lef(r_{j,l}\rig)&\geqslant\min\lef\{v_{\varphi|_{k\lef[\underline{X}\rig]},\lef(b_{i}\rig)_{i\in\lef\llbracket1,n\rig\rrbracket}}\lef(r_{j,0}\rig)+\sum_{t=0}^{l-1}p^{t}d\mid j\in\lef\llbracket1,m\rig\rrbracket\rig\}\\
&=\min\lef\{v_{\varphi|_{k\lef[\underline{X}\rig]},\lef(b_{i}\rig)_{i\in\lef\llbracket1,n\rig\rrbracket}}\lef(r_{j,0}\rig)\mid j\in\lef\llbracket1,m\rig\rrbracket\rig\}+d\frac{p^{l}-1}{p-1}\text{.}
\end{aligned}
\end{equation*}

We thus derive from lemma \ref{pseudovaluationbasis} that:
\begin{equation*}
v_{\varphi|_{k\lef[\underline{X}\rig]},\lef(b_{i}\rig)_{i\in\lef\llbracket1,n\rig\rrbracket}}\lef(r_{j,l}\rig)\geqslant v_{\varphi,\lef(b_{i}\rig)_{i\in\lef\llbracket1,n+m\rig\rrbracket}}\lef(s\rig)-C+p^{l}\frac{d}{p-1}\text{.}
\end{equation*}

The nonpositivity of $d$ winds up the proof.
\end{proof}

\begin{lemm}\label{lowerboundoverconvergentverschiebung}
Let $w\in W^{\dagger}\lef(\overline{R}\rig)$ such that $w_{0}=0$. Let $E<1$ be a real number. Then:
\begin{equation*}
\forall b\in{\lef]0,+\infty\rig[}^{n},\ \exists\delta>0,\ \forall\varepsilon\in\lef]0,\delta\rig],\ \gamma_{\varepsilon,\varphi,b}\lef(w\rig)\geqslant E\text{.}
\end{equation*}
\end{lemm}

\begin{proof}
Fix $b\in{\lef]0,+\infty\rig[}^{n}$. By definition, there is $\epsilon>0$ such that $\gamma_{\epsilon,\varphi,b}\lef(w\rig)\neq-\infty$. If $\gamma_{\epsilon,\varphi,b}\lef(w\rig)\geqslant E$, put $\delta\coloneqq\epsilon$ and the proof is over.

Otherwise, let $\delta\coloneqq\frac{\epsilon\lef(1-E\rig)}{1-\gamma_{\epsilon,\varphi,b}\lef(w\rig)}$. Then for any $\varepsilon\in\lef]0,\delta\rig]$ and any $i\in\mathbb{N}^{*}$ we have:
\begin{equation*}
\begin{split}
i+\varepsilon p^{-i}v_{\varphi,b}\lef(w_{i}\rig)&\geqslant\frac{E-\gamma_{\epsilon,\varphi,b}\lef(w\rig)}{1-\gamma_{\epsilon,\varphi,b}\lef(w\rig)}i+\frac{1-E}{1-\gamma_{\epsilon,\varphi,b}\lef(w\rig)}\lef(i+\epsilon p^{-i}v_{\varphi,b}\lef(w_{i}\rig)\rig)\\
&\geqslant\frac{E-\gamma_{\epsilon,\varphi,b}\lef(w\rig)+\lef(1-E\rig)\gamma_{\epsilon,\varphi,b}\lef(w\rig)}{1-\gamma_{\epsilon,\varphi,b}\lef(w\rig)}\\
&\geqslant E\text{.}
\end{split}
\end{equation*}

This ends the proof because $w_{0}=0$.
\end{proof}

\begin{lemm}\label{pseudovaluationrelativelyperfectbasis}
Let $\overline{S}$ be a relatively perfect commutative $\overline{R}$-algebra. Suppose given a tuple $\lef(s\lef(i\rig)\rig)_{i\in\lef\llbracket1,m\rig\rrbracket}\in{W^{\dagger}\lef(\overline{S}\rig)}^{m}$, where $m\in\mathbb{N}$, such that $\lef({s\lef(i\rig)}_{0}\rig)_{i\in\lef\llbracket1,m\rig\rrbracket}$ is a $\overline{R}$-basis of $\overline{S}$. Extend $\varphi$ to a morphism of $k\lef[\underline{X}\rig]$-algebras $\varphi\colon k\lef[X_{1},\ldots,X_{n},Y_{1},\ldots,Y_{m}\rig]\to\overline{S}$ such that $\varphi\lef(Y_{i}\rig)={s\lef(i\rig)}_{0}$ for every $i\in\lef\llbracket1,m\rig\rrbracket$.

Then, there exists $b\in{\lef]0,+\infty\rig[}^{n+m}$ and $\delta>0$ such that for any $\varepsilon\in\lef]0,\delta\rig]$, any $l\in\mathbb{N}$ and any tuple $\lef(r\lef(i\rig)\rig)_{i\in\lef\llbracket1,m\rig\rrbracket}\in{W\lef(\overline{R}\rig)}^{m}$ we have:
\begin{equation*}
\gamma_{\varepsilon,\varphi,b}\lef(\sum_{i=1}^{m}V^{l}\lef(\lef[{r\lef(i\rig)}_{0}\rig]\rig)s\lef(i\rig)\rig)\geqslant\gamma_{\varepsilon,\varphi,b}\lef(\sum_{i=1}^{m}V^{l}\lef(r\lef(i\rig)\rig)s\lef(i\rig)\rig)\text{.}
\end{equation*}
\end{lemm}

\begin{proof}
Let $b=\lef(b_{i}\rig)_{i\in\lef\llbracket1,n+m\rig\rrbracket}\in{\lef]0,+\infty\rig[}^{n+m}$ and $C,D\geqslant0$ be the reals we get by applying lemma \ref{basispseudovaluationcontrol} in this situation. By proposition \ref{characteriserelativelyperfect}, this is possible. Let $l\in\mathbb{N}$, fix $\lef(r\lef(i\rig)\rig)_{i\in\lef\llbracket1,m\rig\rrbracket}\in{W\lef(\overline{R}\rig)}^{m}$, and let $w\coloneqq\sum_{i=1}^{m}V^{l}\lef(r\lef(i\rig)\rig)s\lef(i\rig)$. For the time being, we let $\varepsilon>0$ be any positive real number.

By definition of $\gamma_{\varepsilon,\varphi,b}$, we have:
\begin{equation*}
l+\varepsilon p^{-l}v_{\varphi,\lef(b_{i}\rig)_{i\in\lef\llbracket1,n+m\rig\rrbracket}}\lef(w_{l}\rig)\geqslant\gamma_{\varepsilon,\varphi,a}\lef(w\rig)\text{.}
\end{equation*}

In other terms:
\begin{equation}\label{easyinequality}
v_{\varphi,\lef(b_{i}\rig)_{i\in\lef\llbracket1,n+m\rig\rrbracket}}\lef(w_{l}\rig)\geqslant\frac{p^{l}}{\varepsilon}\lef(\gamma_{\varepsilon,\varphi,b}\lef(w\rig)-l\rig)\text{.}
\end{equation}

But by \eqref{vxfyvxy} we have $w_{l}=\sum_{i=1}^{m}{r\lef(i\rig)}_{0}{{s\lef(i\rig)}_{0}}^{p^{l}}$, hence:
\begin{equation*}
\begin{split}
v_{\varphi|_{k\lef[\underline{X}\rig]},\lef(b_{i}\rig)_{i\in\lef\llbracket1,n\rig\rrbracket}}\lef({r\lef(i\rig)}_{0}\rig)&\overset{\hphantom{\eqref{easyinequality}}}{\overset{\ref{basispseudovaluationcontrol}}{\geqslant}}v_{\varphi,\lef(b_{i}\rig)_{i\in\lef\llbracket1,n+m\rig\rrbracket}}\lef(w_{l}\rig)-C-p^{l}D\\
&\overset{\eqref{easyinequality}}{\geqslant}\frac{p^{l}}{\varepsilon}\lef(\gamma_{\varepsilon,\varphi,b}\lef(w\rig)-l-\varepsilon\lef(p^{-l}C+D\rig)\rig)\text{.}
\end{split}
\end{equation*}

Thus:
\begin{equation*}
\forall i\in\lef\llbracket1,m\rig\rrbracket,\ \gamma_{\varepsilon,\varphi,b}\lef(V^{l}\lef(\lef[{r\lef(i\rig)}_{0}\rig]\rig)\rig)\geqslant\gamma_{\varepsilon,\varphi,b}\lef(w\rig)-\varepsilon\lef(p^{-l}C+D\rig)\text{.}
\end{equation*}

For all $i\in\lef\llbracket1,m\rig\rrbracket$, we can write $s\lef(i\rig)=\lef[{s\lef(i\rig)}_{0}\rig]+v\lef(i\rig)$ for some $v\lef(i\rig)\in W^{\dagger}\lef(\overline{S}\rig)$ which has to satisfy ${v\lef(i\rig)}_{0}=0$. Applying lemma \ref{lowerboundoverconvergentverschiebung} to these $v\lef(i\rig)$, we see that there exists $\delta>0$ such that when $\varepsilon\leqslant\delta$ we have $\gamma_{\varepsilon,\varphi,b}\lef(v\lef(i\rig)\rig)\geqslant\frac{1}{2}$. By shrinking $\delta$, we can also assume that $\varepsilon\lef(p^{-l}C+D\rig)\leqslant\frac{1}{2}$. Since $\gamma_{\varepsilon,\varphi,b}$ is a pseudovaluation \cite[proposition 2.3]{overconvergentwittvectors}, we get:
\begin{equation*}
\gamma_{\varepsilon,\varphi,b}\lef(\sum_{i=1}^{m}V^{l}\lef(\lef[{r\lef(i\rig)}_{0}\rig]\rig)v\lef(i\rig)\rig)\geqslant\gamma_{\varepsilon,\varphi,b}\lef(w\rig)\text{.}
\end{equation*}

By \cite[IX. \textsection1 N$^\circ$3]{algebrecommutativechapitres}, if we put $x\coloneqq\sum_{i=1}^{m}V^{l}\lef(\lef[{r\lef(i\rig)}_{0}\rig]\rig)\lef[{s\lef(i\rig)}_{0}\rig]$, then for all $k\in\mathbb{N}$ we can write $x_{l+k}$ as a homogenous integral polynomial of degree $p^{k}$ with indeterminates ${r\lef(i\rig)}_{0}{{s\lef(i\rig)}_{0}}^{p^{l}}$ for $i\in\lef\llbracket1,m\rig\rrbracket$.

Let $E\coloneqq-\min\lef\{v_{\varphi|_{k\lef[\underline{X}\rig]},\lef(b_{i}\rig)_{i\in\lef\llbracket1,n\rig\rrbracket}}\lef({s\lef(i\rig)}_{0}\rig)\mid i\in\lef\llbracket1,m\rig\rrbracket\rig\}$. As $v_{\varphi|_{k\lef[\underline{X}\rig]},\lef(b_{i}\rig)_{i\in\lef\llbracket1,n\rig\rrbracket}}$ is a pseudovaluation, we get that:
\begin{equation*}
v_{\varphi|_{k\lef[\underline{X}\rig]},\lef(b_{i}\rig)_{i\in\lef\llbracket1,n\rig\rrbracket}}\lef(x_{l+k}\rig)\geqslant\frac{p^{k+l}}{\varepsilon}\lef(\gamma_{\varepsilon,\varphi,b}\lef(w\rig)-l-\varepsilon\lef(p^{-l}C+D+E\rig)\rig)\text{.}
\end{equation*}

This can be rewritten as follows:
\begin{equation*}
k+l+\varepsilon p^{-k-l}v_{\varphi|_{k\lef[\underline{X}\rig]},\lef(b_{i}\rig)_{i\in\lef\llbracket1,n\rig\rrbracket}}\lef(x_{l+k}\rig)\geqslant k+\gamma_{\varepsilon,\varphi,b}\lef(w\rig)-\varepsilon\lef(p^{-l}C+D+E\rig)\text{.}
\end{equation*}

If $p^{-l}C+D+E=0$, we can keep the same $\delta$ as before, otherwise we need to shrink it again so that $\delta\leqslant\frac{1}{p^{-l}C+D+E}$. Then:
\begin{equation*}
\forall k\in\mathbb{N}^{*},\ k+l+\varepsilon p^{-k-l}v_{\varphi|_{k\lef[\underline{X}\rig]},\lef(b_{i}\rig)_{i\in\lef\llbracket1,n\rig\rrbracket}}\lef(x_{l+k}\rig)\geqslant\gamma_{\varepsilon,\varphi,b}\lef(w\rig)\text{.}
\end{equation*}

Furthermore, for all $i\in\lef\llbracket0,l\rig\rrbracket$ we have $x_{i}=w_{i}$, so $\gamma_{\varepsilon,\varphi,b}\lef(x\rig)\geqslant\gamma_{\varepsilon,\varphi,b}\lef(w\rig)$.
\end{proof}

We are now able to state the characterisation of finite free relatively perfect $\overline{R}$-algebras using overconvergent Witt vectors, which can be understood as the converse of corollary \ref{whencanitbewittfreecoro}.

\begin{prop}\label{decompositionfreewitt}
Let $\overline{S}$ be a commutative $\overline{R}$-algebra. Let $m\in\mathbb{N}$ and fix a tuple $\lef(s\lef(i\rig)\rig)_{i\in\lef\llbracket1,m\rig\rrbracket}\in{W^{\dagger}\lef(\overline{S}\rig)}^{m}$. Then, the following properties are equivalent:
\begin{enumerate}
\item the family $\lef({s\lef(i\rig)}_{0}\rig)_{i\in\lef\llbracket1,m\rig\rrbracket}$ is a basis of $\overline{S}$ as a $\overline{R}$-module, and $\overline{S}$ is a relatively perfect $\overline{R}$-algebra; that is, an algebra satisfying the equivalent conditions given by proposition \ref{characteriserelativelyperfect};
\item the family $\lef(s\lef(i\rig)\rig)_{i\in\lef\llbracket1,m\rig\rrbracket}$ is a basis of $W\lef(\overline{S}\rig)$ as a $W\lef(\overline{R}\rig)$-module;
\item for every $n\in\mathbb{N}$, the family $\lef(F^{n}\lef(s\lef(i\rig)\rig)\rig)_{i\in\lef\llbracket1,m\rig\rrbracket}$ is a basis of $W\lef(\overline{S}\rig)$ as a $W\lef(\overline{R}\rig)$-module;
\item for every $n\in\mathbb{N}$, the family $\lef(F^{n}\lef(s\lef(i\rig)\rig)\rig)_{i\in\lef\llbracket1,m\rig\rrbracket}$ is a basis of $W^{\dagger}\lef(\overline{S}\rig)$ as a $W^{\dagger}\lef(\overline{R}\rig)$-module.
\end{enumerate}

Moreover, when these conditions are satisfied, then if we extend $\varphi$ as a morphism of $k\lef[\underline{X}\rig]$-algebras $\varphi\colon k\lef[X_{1},\ldots,X_{n},Y_{1},\ldots,Y_{m}\rig]\to\overline{S}$ with $\varphi\lef(Y_{i}\rig)={s\lef(i\rig)}_{0}$ for every $i\in\lef\llbracket1,m\rig\rrbracket$, there exists $\lef(b_{i}\rig)_{i\in\lef\llbracket1,n+m\rig\rrbracket}\in{\lef]0,+\infty\rig[}^{n+m}$, $\delta>0$ and $E\geqslant0$ such that for all $\varepsilon\in\lef]0,\delta\rig]$ and for every tuple $\lef(r\lef(i\rig)\rig)_{i\in\lef\llbracket1,m\rig\rrbracket}\in{W^{\dagger}\lef(\overline{R}\rig)}^{m}$ we have:
\begin{gather*}
\forall i\in\lef\llbracket1,m\rig\rrbracket,\ \gamma_{\varepsilon,\varphi|_{k\lef[\underline{X}\rig]},\lef(b_{i}\rig)_{i\in\lef\llbracket1,n\rig\rrbracket}}\lef(r\lef(i\rig)\rig)\geqslant\gamma_{\varepsilon,\varphi,\lef(b_{i}\rig)_{i\in\lef\llbracket1,n+m\rig\rrbracket}}\lef(\sum_{i=1}^{m}r\lef(i\rig)s\lef(i\rig)\rig)-\varepsilon E\text{,}\\
\forall u\in\mathbb{N},\ \sum_{i=1}^{m}r\lef(i\rig)s\lef(i\rig)\in V^{u}\lef(W\lef(\overline{S}\rig)\rig)\iff\forall i\in\lef\llbracket1,m\rig\rrbracket,\ r\lef(i\rig)\in V^{u}\lef(W\lef(\overline{R}\rig)\rig)\text{.}
\end{gather*}
\end{prop}

\begin{proof}
Assume first that $\lef({s\lef(i\rig)}_{0}\rig)_{i\in\lef\llbracket1,m\rig\rrbracket}$ is a basis of $\overline{S}$ as a $\overline{R}$-module, and that $\overline{S}$ is a relatively perfect $\overline{R}$-algebra. This is the setting of proposition \ref{characteriserelativelyperfect}.

Let $w=\lef(w_{u}\rig)_{u\in\mathbb{N}}\in W^{\dagger}\lef(\overline{S}\rig)$. For starters, assume that $w\in V^{u}\lef(W\lef(\overline{S}\rig)\rig)$ for a given $u\in\mathbb{N}$. By proposition \ref{characteriserelativelyperfect}, we can find a unique writing of the form $w_{u}=\sum_{i=1}^{m}{r\lef(i,u\rig)}{{s\lef(i\rig)}_{0}}^{p^{u}}$, where every ${r\lef(i,u\rig)}\in\overline{R}$ for each $i\in\lef\llbracket1,m\rig\rrbracket$. We then have $w-\sum_{i=1}^{m}V^{u}\lef(\lef[{r\lef(i,u\rig)}\rig]F^{u}\lef(s\lef(i\rig)\rig)\rig)\in V^{u+1}\lef(W\lef(\overline{S}\rig)\rig)$.

Recall that $V^{u}\lef(\lef[{r\lef(i,u\rig)}\rig]F^{u}\lef(s\lef(i\rig)\rig)\rig)=V^{u}\lef(\lef[{r\lef(i,u\rig)}\rig]\rig)s\lef(i\rig)$ for all $i\in\lef\llbracket1,m\rig\rrbracket$ by \eqref{vxfyvxy}. Therefore, by repeating this process inductively on $u\in\mathbb{N}$, by $V$-adic convergence we get unique Witt vectors $\lef(r\lef(i\rig)\rig)_{i\in\lef\llbracket1,m\rig\rrbracket}\in{W\lef(\overline{R}\rig)}^{m}$ satisfying:
\begin{equation*}
w=\sum_{i=1}^{m}r\lef(i\rig)s\lef(i\rig)\text{.}
\end{equation*}

So we have shown $\lef(1\rig)\implies\lef(2\rig)$.

Let us now assume $\lef(2\rig)$. Apply proposition \ref{whencanitbewittfree} to see that $\overline{S}$ is a relatively semiperfect $\overline{R}$-algebra. We can then use lemma \ref{characteriserelativelysemiperfect} and notice that we can apply the above construction to find that for each $n\in\mathbb{N}$, the family $\lef(F^{n}\lef(s\lef(i\rig)\rig)\rig)_{i\in\lef\llbracket1,m\rig\rrbracket}$ is a $W\lef(\overline{R}\rig)$-generating family of $W\lef(\overline{S}\rig)$. In particular it is a $W\lef(\overline{R}\rig)$-basis \cite[05G8]{stacksproject}. We have $\lef(2\rig)\implies\lef(3\rig)$.

Corollary \ref{whencanitbewittfreecoro} gives us $\lef(3\rig)\implies\lef(1\rig)$ and $\lef(4\rig)\implies\lef(1\rig)$. To finish the proof, let us prove that $\lef(1\rig)\wedge\lef(3\rig)\implies\lef(4\rig)$.

Let $\lef(b_{i}\rig)_{i\in\lef\llbracket1,n+m\rig\rrbracket}\in{\lef]0,+\infty\rig[}^{n+m}$ and $\delta>0$ satisfying the conditions of lemma \ref{pseudovaluationrelativelyperfectbasis}, so that we find by induction on $l\in\mathbb{N}$ the following inequality:
\begin{equation*}
\forall\varepsilon\in\lef]0,\delta\rig],\ \forall l\in\mathbb{N},\ \gamma_{\varepsilon,\varphi,\lef(b_{i}\rig)_{i\in\lef\llbracket1,n+m\rig\rrbracket}}\lef(\sum_{i=1}^{m}V^{l}\lef(\lef[{r\lef(i\rig)}_{l}\rig]\rig)s\lef(i\rig)\rig)\geqslant\gamma_{\varepsilon,\varphi,\lef(b_{i}\rig)_{i\in\lef\llbracket1,n+m\rig\rrbracket}}\lef(w\rig)\text{.}
\end{equation*}

In particular:
\begin{equation*}
\forall\varepsilon\in\lef]0,\delta\rig],\ \forall l\in\mathbb{N},\ l+\varepsilon p^{-l}v_{\varphi,\lef(b_{i}\rig)_{i\in\lef\llbracket1,n+m\rig\rrbracket}}\lef(\sum_{i=1}^{m}{r\lef(i\rig)}_{l}{s_{i}}^{p^{l}}\rig)\geqslant\gamma_{\varepsilon,\varphi,\lef(b_{i}\rig)_{i\in\lef\llbracket1,n+m\rig\rrbracket}}\lef(w\rig)\text{.}
\end{equation*}

By virtue of lemma \ref{basispseudovaluationcontrol}, we find $C,D\geqslant0$ independently on the $\lef(r\lef(i\rig)\rig)_{i\in\lef\llbracket1,m\rig\rrbracket}$ such that:
\begin{multline*}
\forall\varepsilon\in\lef]0,\delta\rig],\ \forall i\in\lef\llbracket1,m\rig\rrbracket,\ \forall l\in\mathbb{N},\\l+\varepsilon p^{-l}v_{\varphi|_{k\lef[\underline{X}\rig]},\lef(b_{i}\rig)_{i\in\lef\llbracket1,n\rig\rrbracket}}\lef({r\lef(i\rig)}_{l}\rig)\geqslant\gamma_{\varepsilon,\varphi,\lef(b_{i}\rig)_{i\in\lef\llbracket1,n+m\rig\rrbracket}}\lef(w\rig)-\varepsilon\lef(p^{-l}C+D\rig)\text{.}
\end{multline*}

So that, by definition, $\gamma_{\varepsilon,\varphi|_{k\lef[\underline{X}\rig]},\lef(b_{i}\rig)_{i\in\lef\llbracket1,n\rig\rrbracket}}\lef(r\lef(i\rig)\rig)\geqslant\gamma_{\varepsilon,\varphi,\lef(b_{i}\rig)_{i\in\lef\llbracket1,n+m\rig\rrbracket}}\lef(w\rig)-\varepsilon\lef(C+D\rig)$ and by \eqref{equivalencesofgamma} and proposition \ref{gammaandzetawitt} we end the implication in the case $n=0$, the case for general $n\in\mathbb{N}$ being exactly similar.
\end{proof}

Of course, without the overconvergence condition on the tuple, one can get a similar result on the usual ring of Witt vectors.

\section{The Lazard morphism}

In this section, $k$ is a perfect commutative $\mathbb{F}_{p}$-algebra such that $W\lef(k\rig)$ is a Noetherian ring. In that setting, for any commutative $W\lef(k\rig)$-algebra $R$ there is a notion of weak completion due to Monsky--Washnitzer \cite{formalcohomologyi}. It is a $W\lef(k\rig)$-algebra that we shall denote $R^{\dagger}$, consisting of all the elements $z\in\hat{R}$ in the $p$-adic completion of $R$ such that:
\begin{equation*}
\exists m,c\in\mathbb{N},\ \exists\lef(P_{j}\rig)_{j\in\mathbb{N}}\in\lef(R\lef[X_{1},\ldots,X_{m}\rig]\rig)^{\mathbb{N}},\ \begin{cases}\displaystyle\exists\underline{x}\in R^{m},\ z=\sum_{j\in\mathbb{N}}p^{j}P_{j}\lef(\underline{x}\rig)\text{,}\\\forall j\in\mathbb{N},\ \deg\lef(P_{j}\rig)\leqslant c\lef(j+1\rig)\text{.}\end{cases}
\end{equation*}

In what follows, $\overline{R}$ shall denote a smooth commutative $k$-algebra, and $R$ will denote a smooth commutative $W\lef(k\rig)$-algebra lifting $\overline{R}$.

An important morphism in the literature is the canonical morphism from the weak completion of $R$ to the Witt vectors with coefficients in $\overline{R}$. The goal of this section is to study extensively its properties.

We shall consider:
\begin{equation*}
F\colon R^{\dagger}\to R^{\dagger}
\end{equation*}
a lift of the Frobenius endomorphism of $\overline{R}$. By virtue of \cite[3.3.4]{relevementdesalgebreslissesetdeleursmorphismes}, such lifts always exist. Furthermore, the work of Arabia also imply that such lifts satisfy, if needed, some compatibility conditions that we will recall here.

In this section, we shall sometimes state results in the case of a polynomial ring $R=W\lef(k\rig)\lef[\underline{X}\rig]=W\lef(k\rig)\lef[X_{1},\ldots,X_{n}\rig]$ for some $n\in\mathbb{N}$. This is warranted by the lemma below.

\begin{lemm}\label{commutativefrobeniuslift}
Under the above assumptions, we can always find $n\in\mathbb{N}$ and two morphisms of $W\lef(k\rig)$-algebras $g\colon{W\lef(k\rig)\lef[\underline{X}\rig]}^{\dagger}\to{W\lef(k\rig)\lef[\underline{X}\rig]}^{\dagger}$ and $\varphi\colon{W\lef(k\rig)\lef[\underline{X}\rig]}^{\dagger}\to R^{\dagger}$ such that $\varphi$ is surjective and that the following diagram commutes:
\begin{equation*}
\begin{tikzcd}[column sep=large]
{W\lef(k\rig)\lef[\underline{X}\rig]}^{\dagger}\arrow[rr,"g"]\arrow[rd,"\varphi"]\arrow[dd]&&{W\lef(k\rig)\lef[\underline{X}\rig]}^{\dagger}\arrow[rd,"\varphi"]\arrow[dd]&\\
&R^{\dagger}\arrow[rr,crossing over,pos=0.33,"F"]&&R^{\dagger}\arrow[dd]\\
k\lef[\underline{X}\rig]\arrow[rr,pos=0.24,"\operatorname{Frob}_{k\lef[\underline{X}\rig]}"]\arrow[rd]&&k\lef[\underline{X}\rig]\arrow[rd]&\\
&\overline{R}\arrow[rr,"\operatorname{Frob}_{\overline{R}}"]\arrow[from=uu,crossing over]&&\overline{R}\text{.}
\end{tikzcd}
\end{equation*}
\end{lemm}

\begin{proof}
As $R$ is a $W\lef(k\rig)$-algebra of finite type, then $\varphi$ exists by functoriality and is onto \cite[theorem 3.2]{formalcohomologyi}. As we assumed that $W\lef(k\rig)$ is Noetherian, according to \cite[3.3.2]{relevementdesalgebreslissesetdeleursmorphismes} there exists a morphism $g$ making the above diagram commutative.
\end{proof}

Since $k$ is reduced, $W\lef(k\rig)$ has no $p$-torsion. It follows that $R^{\dagger}$ is flat over $W\lef(k\rig)$ \cite[theorem 2.4]{formalcohomologyi}, in particular $R^{\dagger}$ also has no $p$-torsion. In this case, it is known that there exists a morphism of $\delta$-rings $s_{F}\colon R^{\dagger}\to W\lef(R^{\dagger}\rig)$ depending on $F$, satisfying $s_{F}\lef(r\rig)_{0}=r$ for all $r\in R^{\dagger}$. For details, see \cite[pp. 508--10]{complexedederhamwittetcohomologiecristalline}.

\begin{deff}\label{lazardmorphism}
If we denote by $\Pi\colon R^{\dagger}\to\overline{R}$ the canonical projection, the \textbf{Lazard morphism} is the morphism of $W\lef(k\rig)$-algebras:
\begin{equation*}
t_{F}\colon\begin{array}{rl}R^{\dagger}\to&W\lef(\overline{R}\rig)\\r\mapsto&W\lef(\Pi\rig)\circ s_{F}\lef(r\rig)\text{.}\end{array}
\end{equation*}
\end{deff}

Our goal in this section is to study the image of this morphism, and to see how it extends to the de Rham--Witt complex.

Throughout this paper, we shall denote by:
\begin{equation*}
\widetilde{\operatorname{Frob}}\colon{W\lef(k\rig)\lef[\underline{X}\rig]}^{\dagger}\to{W\lef(k\rig)\lef[\underline{X}\rig]}^{\dagger}
\end{equation*}
the canonical lift of the Frobenius endomorphism on $k\lef[\underline{X}\rig]=k\lef[X_{1},\ldots,X_{n}\rig]$, where $n\in\mathbb{N}$, such that $\widetilde{\operatorname{Frob}}|_{W\lef(k\rig)}$ is the Witt--Frobenius morphism, and that any indeterminate $X_{i}$ for $i\in\lef\llbracket1,n\rig\rrbracket$ is mapped to ${X_{i}}^{p}$. As a consequence of the construction of $t_{\widetilde{\operatorname{Frob}}}$ \cite[0. 1.3.18]{complexedederhamwittetcohomologiecristalline} we have:
\begin{equation}\label{xtoteichmullerx}
\forall i\in\lef\llbracket1,n\rig\rrbracket,\ t_{\widetilde{\operatorname{Frob}}}\lef(X_{i}\rig)=\lef[X_{i}\rig]\text{.}
\end{equation}

In the following lemmas, we generalise \cite[proposition 3.1]{overconvergentderhamwittcohomology} to this slightly more general setting. We shall use $\zeta_{\varepsilon}$ as our main tool in that regard.

\begin{lemm}\label{isodegreezero}
The morphism $t_{\widetilde{\operatorname{Frob}}}$ induces an isomorphism of $W\lef(k\rig)$-algebras:
\begin{equation*}
{W\lef(k\rig)\lef[\underline{X}\rig]}^{\dagger}\cong W^{\dagger}\Omega^{\mathrm{int},0}_{k\lef[\underline{X}\rig]/k}\subset W\lef(k\lef[\underline{X}\rig]\rig)\text{.}
\end{equation*}
\end{lemm}

\begin{proof}
By construction, ${W\lef(k\rig)\lef[\underline{X}\rig]}^{\dagger}$ is a sub-$W\lef(k\rig)$-algebra of the $p$-adic completion $\widehat{W\lef(k\rig)\lef[\underline{X}\rig]}$. For any $w\in\widehat{W\lef(k\rig)\lef[\underline{X}\rig]}$ and any $\alpha\in\mathbb{N}^{n}$, we shall write $w\lef(\alpha\rig)\in W\lef(k\rig)$ the Witt vector such that $w=\sum_{\alpha\in\mathbb{N}^{n}}w\lef(\alpha\rig)\underline{X}^{\alpha}$.

This proposition is obvious when $n=0$, so we are going to assume that $n$ is positive. For all $N\in\mathbb{N}$, denote by $\mathbb{N}^{n}_{\geqslant N}=\lef\{\alpha\in\mathbb{N}^{n}\mid\lef\lvert\alpha\rig\rvert\geqslant N\rig\}$. Applying \cite[theorem 2.3]{formalcohomologyi}, we get an isomorphism of $W\lef(k\rig)$-algebras:
\begin{equation*}
{W\lef(k\rig)\lef[\underline{X}\rig]}^{\dagger}\cong\lef\{w\in\widehat{W\lef(k\rig)\lef[\underline{X}\rig]}\mid\exists\varepsilon>0,\ \exists N\in\mathbb{N}^{*},\ \forall\alpha\in\mathbb{N}^{n}_{\geqslant N},\ \frac{\operatorname{v}_{V}\lef(w\lef(\alpha\rig)\rig)}{\lef\lvert\alpha\rig\rvert}\geqslant\varepsilon\rig\}\text{.}
\end{equation*}

The condition $\frac{\operatorname{v}_{V}\lef(w\lef(\alpha\rig)\rig)}{\lef\lvert\alpha\rig\rvert}\geqslant\varepsilon$ is equivalent to $2n\operatorname{v}_{V}\lef(w\lef(\alpha\rig)\rig)-2n\varepsilon\lef\lvert\alpha\rig\rvert\geqslant0$. Only a finite number of $\alpha\in\mathbb{N}^{n}$ do not satisfy $\lef\lvert\alpha\rig\rvert\geqslant N$. So applying \eqref{xtoteichmullerx}, an element of ${W\lef(k\rig)\lef[\underline{X}\rig]}^{\dagger}$ meeting the above condition for some $\varepsilon>0$ is sent via $t_{\widetilde{\operatorname{Frob}}}$ to an element $w\in W\Omega^{\mathrm{int},0}_{k\lef[\underline{X}\rig]/k}$ with $\zeta_{2n\varepsilon}\lef(w\rig)\neq-\infty$. Therefore, we have an injective morphism of $W\lef(k\rig)$-algebras ${W\lef(k\rig)\lef[\underline{X}\rig]}^{\dagger}\to W^{\dagger}\Omega^{\mathrm{int},0}_{k\lef[\underline{X}\rig]/k}$. It only remains to show that it is surjective.

Take $w\in W^{\dagger}\Omega^{\mathrm{int},0}_{k\lef[\underline{X}\rig]/k}$ satisfying $\zeta_{\varepsilon}\lef(w\rig)\neq-\infty$ for some $\varepsilon>0$. For all $\alpha\in\mathbb{N}^{n}$ considered as a weight function, we have $2n\operatorname{v}_{V}\lef(\eta_{\alpha,\emptyset}\rig)-\varepsilon\lef\lvert\alpha\rig\rvert\geqslant\zeta_{\varepsilon}\lef(w\rig)$. In particular, if $\lef\lvert\alpha\rig\rvert\geqslant-\frac{2}{\varepsilon}\zeta_{\varepsilon}\lef(w\rig)$, then $\zeta_{\varepsilon}\lef(w\rig)\geqslant-\frac{\varepsilon}{2}\lef\lvert\alpha\rig\rvert$, so that $2n\operatorname{v}_{V}\lef(\eta_{\alpha,\emptyset}\rig)-\frac{\varepsilon}{2}\lef\lvert\alpha\rig\rvert\geqslant0$ and
\begin{equation*}
\frac{\operatorname{v}_{V}\lef(\eta_{\alpha,\emptyset}\rig)}{\lef\lvert\alpha\rig\rvert}\geqslant\frac{\varepsilon}{4n}\text{,}
\end{equation*}
which ends the proof.
\end{proof}

\begin{lemm}\label{tfaffine}
The image of the morphism $t_{F}$ is overconvergent when $R=W\lef(k\rig)\lef[\underline{X}\rig]$, that is, we have a morphism of $W\lef(k\rig)$-algebras:
\begin{equation*}
t_{F}\colon{W\lef(k\rig)\lef[\underline{X}\rig]}^{\dagger}\to W^{\dagger}\lef(k\lef[\underline{X}\rig]\rig)\text{.}
\end{equation*}
\end{lemm}

\begin{proof}
We briefly revisit arguments given in the proof of \cite[proposition 3.1]{overconvergentwittvectors} to show that they still stand in our more general setting. In what follows, we will consider ${W\lef(k\rig)\lef[\underline{X}\rig]}^{\dagger}$ as a sub-$W\lef(k\rig)$-algebra of $W\lef(k\lef[\underline{X}\rig]\rig)$ using lemma \ref{isodegreezero}.

Let $F\colon{W\lef(k\rig)\lef[\underline{X}\rig]}^{\dagger}\to{W\lef(k\rig)\lef[\underline{X}\rig]}^{\dagger}$ be a Frobenius lift. As ${W\lef(k\rig)\lef[\underline{X}\rig]}^{\dagger}$ is $p$-torsion free, for any $i\in\lef\llbracket1,n\rig\rrbracket$ we can uniquely write $F\lef(X_{i}\rig)={X_{i}}^{p}+p\delta\lef(X_{i}\rig)$, where $\delta$ denotes the $\delta$-structure associated to $F$. Notice that since $\delta\lef(X_{i}\rig)$ is overconvergent, we can find $\varepsilon>0$ small enough so that $\zeta_{\varepsilon}\lef(\delta\lef(X_{i}\rig)\rig)\geqslant-2n$.

Recall that $\zeta_{\varepsilon}$ is a pseudovaluation \cite[theorem 3.17]{pseudovaluationsonthederhamwittcomplex}. So it follows from \eqref{multiplicationbypzeta} that $\zeta_{\varepsilon}\lef(F\lef(X_{i}\rig)\rig)\geqslant-p\varepsilon$. Also, for this sufficiently small $\varepsilon>0$ we derive the following inequality:
\begin{equation*}
\forall\eta\in W\lef(k\rig),\ \forall l\in\mathbb{N}^{n},\ \zeta_{\varepsilon}\lef(F\lef(\eta\underline{X}^{l}\rig)\rig)\geqslant\zeta_{\varepsilon}\lef(\eta\rig)-pl\varepsilon=p\zeta_{\varepsilon}\lef(\eta\underline{X}^{l}\rig)\text{.}
\end{equation*}

And applying \eqref{padicoverconvergence} and the pseudovaluation properties of $\zeta_{\varepsilon}$ we find:
\begin{equation*}
\forall w\in{W\lef(k\rig)\lef[\underline{X}\rig]}^{\dagger},\ \zeta_{\varepsilon}\lef(F\lef(w\rig)\rig)\geqslant p\zeta_{\varepsilon}\lef(w\rig)\text{.}
\end{equation*}

And another induction yields:
\begin{equation}\label{zetaepsilonfrobenius}
\forall m\in\mathbb{N},\ \forall w\in{W\lef(k\rig)\lef[\underline{X}\rig]}^{\dagger},\ \zeta_{\varepsilon}\lef(F^{m}\lef(w\rig)\rig)\geqslant p^{m}\zeta_{\varepsilon}\lef(w\rig)\text{.}
\end{equation}

Remember that we are given a morphism $s_{F}\colon{W\lef(k\rig)\lef[\underline{X}\rig]}^{\dagger}\to W\lef({W\lef(k\rig)\lef[\underline{X}\rig]}^{\dagger}\rig)$ of $W\lef(k\rig)$-algebras through which the image of $w\in{W\lef(k\rig)\lef[\underline{X}\rig]}^{\dagger}$ is the unique Witt vector whose ghost components are $\lef(w,F\lef(w\rig),F^{2}\lef(w\rig),\ldots\rig)$. We have ${s_{F}\lef(w\rig)}_{0}=w$. Assume that we have shown for some $m\in\mathbb{N}$ that:
\begin{equation*}
\forall i\in\lef\llbracket0,m\rig\rrbracket,\ \zeta_{\varepsilon}\lef({s_{F}\lef(w\rig)}_{i}\rig)\geqslant p^{i}\zeta_{\varepsilon}\lef(w\rig)-2n\sum_{j=0}^{i-1}p^{j}\text{.}
\end{equation*}

Then using formula \eqref{multiplicationbypzeta} and the pseudovaluation ones we find:
\begin{equation*}
\forall i\in\lef\llbracket0,m\rig\rrbracket,\ \zeta_{\varepsilon}\lef(p^{i}{{s_{F}\lef(w\rig)}_{i}}^{p^{m+1-i}}\rig)\geqslant p^{m+1}\zeta_{\varepsilon}\lef(w\rig)-2n\sum_{j=m+1-i}^{m}p^{j}+2ni\text{.}
\end{equation*}

But:
\begin{equation*}
\forall i\in\lef\llbracket0,m\rig\rrbracket,\ -2n\sum_{j=m+1-i}^{m}p^{j}+2ni\geqslant-2n\sum_{j=1}^{m}p^{j}+2nm\text{.}
\end{equation*}

Hence from the ghost equation:
\begin{equation*}
\sum_{i=0}^{m+1}p^{i}{{s_{F}\lef(w\rig)}_{i}}^{p^{m+1-i}}=F^{m+1}\lef(w\rig)\text{,}
\end{equation*}
and from the formula \eqref{zetaepsilonfrobenius} and the pseudovaluation properties, we deduce:
\begin{equation*}
\zeta_{\varepsilon}\lef(p^{m+1}{{s_{F}\lef(w\rig)}_{m+1}}\rig)\geqslant p^{m+1}\zeta_{\varepsilon}\lef(w\rig)-2n\sum_{j=1}^{m}p^{j}+2nm\text{.}
\end{equation*}

Using \eqref{multiplicationbypzeta} again, we have then shown by induction on $m\in\mathbb{N}$ that:
\begin{equation*}
\forall m\in\mathbb{N},\ \zeta_{\varepsilon}\lef({s_{F}\lef(w\rig)}_{m}\rig)\geqslant p^{m}\zeta_{\varepsilon}\lef(w\rig)-2n\sum_{j=0}^{m-1}p^{j}\geqslant p^{m}\lef(\zeta_{\varepsilon}\lef(w\rig)-2n\rig)\text{.}
\end{equation*}

As $k$ is semiperfect, this implies after projecting modulo $p$ that:
\begin{equation*}
\forall m\in\mathbb{N},\ -\varepsilon\deg\lef(\overline{{s_{F}\lef(w\rig)}_{m}}\rig)\geqslant p^{m}\lef(\zeta_{\varepsilon}\lef(w\rig)-2n\rig)\text{.}
\end{equation*}

Recall that ${t_{F}\lef(w\rig)}_{m}=\overline{{s_{F}\lef(w\rig)}_{m}}$ for all $m\in\mathbb{N}$ by definition, so in particular:
\begin{equation*}
\gamma_{\varepsilon,\operatorname{Id}_{k\lef[\underline{X}\rig]},\lef(1,\ldots,1\rig)}\lef(t_{F}\lef(w\rig)\rig)\geqslant\zeta_{\varepsilon}\lef(w\rig)-2n\text{,}
\end{equation*}
and the right-hand side is finite when $\varepsilon$ is small enough, so we conclude using proposition \ref{gammaandzetawitt}.
\end{proof}

A consequence of lemma \ref{tfaffine} is that the image of the morphism $t_{F}$ is overconvergent. Indeed, use lemma \ref{commutativefrobeniuslift} to find a commutative diagram as in its statement, so that by both functoriality and the lemma we get in degree zero that $t_{F}$ factors through $W^{\dagger}\lef(\overline{R}\rig)$.

In particular, we obtain by adjunction a morphism of $W\lef(k\rig)$-dgas which we shall denote by the same symbol:
\begin{equation*}
t_{F}\colon\Omega^{\mathrm{sep}}_{R^{\dagger}/W\lef(k\rig)}\to W^{\dagger}\Omega_{\overline{R}/k}\text{.}
\end{equation*}

\begin{lemm}\label{injectionmodp}
Let $M$ and $N$ be two additive commutative groups, and assume that $\bigcap_{n\in\mathbb{N}}p^{n}M=\lef\{0_{M}\rig\}$ and that $N$ has no $p$-torsion. Let $\varphi\colon M\to N$ be a morphism of groups such that the morphism $\overline{\varphi}\colon\overline{M}\to\overline{N}$ induced modulo $p$ is injective. Then $\varphi$ is injective.
\end{lemm}

\begin{proof}
Let $m\in M$ such that $\varphi\lef(m\rig)=0$. Assume that $m\neq0$. By hypothesis on $M$, there is $n\in\mathbb{N}$ and $m'\in M\smallsetminus pM$ such that $m=p^{n}m'$. We get $\varphi\lef(m\rig)=p^{n}\varphi\lef(m'\rig)=0$. As $N$ has no $p$-torsion, it implies that $\varphi\lef(m'\rig)=0$. We obtain in turn that $\overline{\varphi}\lef(\overline{m'}\rig)=0$, so $\overline{m'}=0$. This contradicts $m'\notin pM$, so $m=0$.
\end{proof}

When $R=W\lef(k\rig)\lef[\underline{X}\rig]$, we shall use in the remainder of this article the following morphism of $W\lef(k\rig)$-modules:
\begin{equation*}
v_{F}\colon\begin{array}{rl}\Omega^{\mathrm{sep}}_{R^{\dagger}/W\lef(k\rig)}\to&W^{\dagger}\Omega_{\overline{R}/k}\\w\mapsto&t_{F}\lef(w\rig)-t_{\widetilde{\operatorname{Frob}}}\lef(w\rig)\text{.}\end{array}
\end{equation*}

We find directly:
\begin{align}
&\forall x\in{W\lef(k\rig)\lef[\underline{X}\rig]}^{\dagger},\ v_{F}\lef(x\rig)\in V\lef(W^{\dagger}\lef(k\lef[\underline{X}\rig]\rig)\rig)\text{,}\label{vfisinvw}\\
&\begin{multlined}\forall x,y \in\Omega^{\mathrm{sep}}_{{W\lef(k\rig)\lef[\underline{X}\rig]}^{\dagger}/W\lef(k\rig)},\\v_{F}\lef(xy\rig)=v_{F}\lef(x\rig)v_{F}\lef(y\rig)+t_{\widetilde{\operatorname{Frob}}}\lef(x\rig)v_{F}\lef(y\rig)+v_{F}\lef(x\rig)t_{\widetilde{\operatorname{Frob}}}\lef(y\rig)\text{.}\end{multlined}\label{productformulavf}
\end{align}

\begin{prop}\label{lazardinjectivity}
The morphism $t_{F}\colon\Omega^{\mathrm{sep}}_{R^{\dagger}/W\lef(k\rig)}\to W^{\dagger}\Omega_{\overline{R}/k}$ is injective, as well as its reduction modulo $p$.

Moreover, if $\overline{R}=k\lef[\underline{X}\rig]$, then the morphism $v_{F}\colon\Omega^{\mathrm{sep}}_{R^{\dagger}/W\lef(k\rig)}\to W^{\dagger}\Omega_{\overline{R}/k}$ takes values in $\operatorname{Fil}^{1}\lef(W^{\dagger}\Omega_{\overline{R}/k}\rig)$.
\end{prop}

\begin{proof}
Modulo $p$, the morphism becomes $\overline{t_{F}}\colon\Omega_{\overline{R}/k}\to\overline{W^{\dagger}\Omega_{\overline{R}/k}}$. By construction in \cite[1.3]{derhamwittcohomologyforaproperandsmoothmorphism}, we have $W_{1}\Omega_{\overline{R}/k}=\Omega_{\overline{R}/k}$, so we get a canonical arrow $\overline{W^{\dagger}\Omega_{\overline{R}/k}}\to\Omega_{\overline{R}/k}$ of $k$-dgas. After composing it with $\overline{t_{F}}$, we obtain a morphism $i_{F}\colon\Omega_{\overline{R}/k}\to\Omega_{\overline{R}/k}$ of $k$-dgas. By adjunction, this morphism is uniquely determined by the restriction morphism in degree $0$, which, by construction of $t_{F}$, is $\operatorname{Id}_{\overline{R}}$. Hence, $i_{F}$ is the identity morphism, and $\overline{t_{F}}$ is injective. We can conclude by using lemma \ref{injectionmodp}.

The second part of the proposition comes from the fact that after projecting $v_{F}$ to $W_{1}\Omega_{\overline{R}/k}$, the morphism becomes zero from the above discussion.
\end{proof}

The next result was shown in Christopher Davis' PhD thesis in the case of a perfect field of positive characteristic. We generalise it with a different proof which uses $\zeta_{\varepsilon}$.

\begin{prop}\label{lazardbijectivity}
The morphism $t_{\widetilde{\operatorname{Frob}}}$ induces an isomorphism of $W\lef(k\rig)$-dgas:
\begin{equation*}
\Omega^{\mathrm{sep}}_{{W\lef(k\rig)\lef[\underline{X}\rig]}^{\dagger}/W\lef(k\rig)}\cong W^{\dagger}\Omega^{\mathrm{int}}_{k\lef[\underline{X}\rig]/k}\text{.}
\end{equation*}
\end{prop}

\begin{proof}
By theorem \ref{structuretheorem}, any $w\in W^{\dagger}\Omega^{\mathrm{int}}_{k\lef[\underline{X}\rig]/k}$ can be written as a convergent series $w=\sum_{\lef(a,I\rig)\in\mathcal{P}}e\lef(\eta_{a,I},a,I\rig)$, where for all $\lef(a,I\rig)\in\mathcal{P}$ the weight function $a$ takes integral values, and $\eta_{a,I}\in W\lef(k\rig)$. Moreover, there exists $\varepsilon>0$ and $C\in\mathbb{R}$ such that $2n\operatorname{v}_{V}\lef(\eta_{a,I}\rig)-\varepsilon\lef\lvert a\rig\rvert>C$.

Let $a\colon\lef\llbracket1,n\rig\rrbracket\to\mathbb{N}$ be a weight function with integral values. In this case:
\begin{multline*}
F^{\operatorname{v}_{p}\lef(a\rig)}\lef(d\lef(\lef[\underline{X}^{p^{-\operatorname{v}_{p}\lef(a\rig)}a}\rig]\rig)\rig)\\\begin{aligned}
&\overset{\eqref{dproducts}}{=}\sum_{i=1}^{n}F^{\operatorname{v}_{p}\lef(a\rig)}\lef(d\lef(\lef[{X_{i}}^{p^{-\operatorname{v}_{p}\lef(a\rig)}a_{i}}\rig]\rig)\prod_{j\in\operatorname{Supp}\lef(a\rig)\smallsetminus\lef\{i\rig\}}\lef[{X_{j}}^{p^{-\operatorname{v}_{p}\lef(a\rig)}a_{j}}\rig]\rig)\\
&\overset{\eqref{frobeniuswittvectors}}{\overset{\eqref{fdttdt}}{=}}\sum_{i=1}^{n}p^{-\operatorname{v}_{p}\lef(a\rig)}a_{i}\lef[{X_{i}}^{a_{i}-1}\rig]d\lef(\lef[{X_{i}}\rig]\rig)\prod_{j\in\operatorname{Supp}\lef(a\rig)\smallsetminus\lef\{i\rig\}}\lef[{X_{j}}^{a_{j}}\rig]\text{.}
\end{aligned}
\end{multline*}

In particular, $F^{\operatorname{v}_{p}\lef(a\rig)}\lef(d\lef(\lef[\underline{X}^{p^{-\operatorname{v}_{p}\lef(a\rig)}a}\rig]\rig)\rig)$ can be written as a linear combination of elements of the form $w_{i}d\lef(\lef[{X_{i}}\rig]\rig)$ with $i\in\lef\llbracket1,n\rig\rrbracket$ and $w_{i}\in W^{\dagger}\Omega^{\mathrm{int},0}_{k\lef[\underline{X}\rig]/k}$ such that:
\begin{equation*}
\zeta_{\varepsilon}\lef(w_{i}\rig)\geqslant\varepsilon\lef(1-\lef\lvert a\rig\rvert\rig)\text{.}
\end{equation*}

Let $I\subset\operatorname{Supp}\lef(a\rig)$ be a partition of $a$. We deduce from the above discussion, and from the fact that $\zeta_{\varepsilon}$ is a pseudovaluation \cite[theorem 3.17]{pseudovaluationsonthederhamwittcomplex}, that we can write $e\lef(\eta_{a,I},a,I\rig)$ as a linear combination of elements of the form $w_{J}\prod_{j\in J}d\lef(\lef[{X_{j}}\rig]\rig)$ with $J\subset\operatorname{Supp}\lef(a\rig)$ and $w_{J}\in W^{\dagger}\Omega^{\mathrm{int},0}_{k\lef[\underline{X}\rig]/k}$ such that $\#J=\#I$, $\operatorname{v}_{V}\lef(w_{J}\rig)\geqslant\operatorname{v}_{V}\lef(\eta_{a,I}\rig)$ and:
\begin{equation*}
\zeta_{\varepsilon}\lef(w_{J}\rig)\geqslant2n\operatorname{v}_{V}\lef(\eta_{a,I}\rig)+\varepsilon\lef(\#J-\lef\lvert a\rig\rvert\rig)\geqslant\varepsilon\#J+C\text{.}
\end{equation*}

By overconvergence in the canonical topology, this means that the set of all $\prod_{j\in J}d\lef(\lef[{X_{j}}\rig]\rig)$ with varying $J\subset\lef\llbracket1,n\rig\rrbracket$ satisfying $\#J=\#I$ generate $W^{\dagger}\Omega^{\mathrm{int},\#I}_{k\lef[\underline{X}\rig]/k}$ as a $W^{\dagger}\Omega^{\mathrm{int},0}_{k\lef[\underline{X}\rig]/k}$-module.

This last fact allows us to conclude that the image of the morphism $t_{\widetilde{\operatorname{Frob}}}$ is all of $W^{\dagger}\Omega^{\mathrm{int}}_{k\lef[\underline{X}\rig]/k}$ by virtue of \cite[theorem 4.5]{formalcohomologyi} and lemma \ref{isodegreezero}; finally, proposition \ref{lazardinjectivity} states that it is injective.
\end{proof}

In the next two propositions, we see how the Lazard morphism is well-behaved regarding overconvergence. These results are primordial in the proofs of the main results.

\begin{prop}\label{mingammaepsilon}
Assume that $R=W\lef(k\rig)\lef[\underline{X}\rig]$. Let $b\coloneqq\lef(b_{i}\rig)_{i\in\lef\llbracket1,n\rig\rrbracket}\in\lef]0,+\infty\rig[^{n}$ and $\mu>0$. We can find $\delta>0$ such that for any $\varepsilon\in\lef]0,\delta\rig]$ and for any weight function $a$ taking values in $\mathbb{N}$ we have:
\begin{equation*}
\gamma_{\varepsilon,\operatorname{Id}_{k\lef[\underline{X}\rig]},b}\lef(v_{F}\lef(\underline{X}^{a}\rig)\rig)\geqslant1-\mu-\varepsilon\sum_{i=1}^{n}b_{i}a_{i}\text{.}
\end{equation*}
\end{prop}

\begin{proof}
We can assume without loss of generality that $\mu\leqslant1$. As a consequence of \eqref{vfisinvw} and lemma \ref{lowerboundoverconvergentverschiebung}, we know that we can find $\delta>0$ such that for any $\varepsilon\in\lef]0,\delta\rig]$ we have:
\begin{equation}\label{smallenougheta}
\forall i\in\lef\llbracket1,n\rig\rrbracket,\ \gamma_{\varepsilon,\operatorname{Id}_{k\lef[\underline{X}\rig]},b}\lef(v_{F}\lef(X_{i}\rig)\rig)\geqslant1-\mu-\varepsilon b_{i}\text{.}
\end{equation}

Take such $\delta$ and $\varepsilon$. We shall prove this proposition by induction on $\lef\lvert a\rig\rvert\in\mathbb{N}$. When $\lef\lvert a\rig\rvert\leqslant1$, there is nothing to do. Assume that the proposition is shown for any weight function $a'$ with integral values such that $\lef\lvert a'\rig\rvert\leqslant j$, for some $j\in\mathbb{N}^{*}$. Let $a$ and $a'$ be two weight functions with integral values such that $\lef\lvert a'\rig\rvert=\lef\lvert a\rig\rvert-1=j$ and $a-a'$ also has values in $\mathbb{N}$; in particular, there exists $i\in\lef\llbracket1,n\rig\rrbracket$ such that $a_{i}={a'}_{i}+1$. According to \eqref{xtoteichmullerx} and \eqref{productformulavf}, we have:
\begin{equation*}
v_{F}\lef(\underline{X}^{a}\rig)=v_{F}\lef(\underline{X}^{a'}X_{i}\rig)=v_{F}\lef(\underline{X}^{a'}\rig)\lef(v_{F}\lef(X_{i}\rig)+\lef[X_{i}\rig]\rig)+\lef[\underline{X}^{a'}\rig]v_{F}\lef(X_{i}\rig)\text{.}
\end{equation*}

Since $\gamma_{\varepsilon,\operatorname{Id}_{k\lef[\underline{X}\rig]},b}$ is a pseudovaluation \cite[proposition 2.3]{overconvergentwittvectors}, by \eqref{smallenougheta} we find the following inequality which concludes the proof:
\begin{multline*}
\gamma_{\varepsilon,\operatorname{Id}_{k\lef[\underline{X}\rig]},b}\lef(v_{F}\lef(\underline{X}^{a}\rig)\rig)\\\geqslant\min\lef\{1-\mu-\varepsilon\sum_{i'=1}^{n}b_{i'}{a'}_{i'}-\varepsilon b_{i},-\varepsilon\sum_{i'=1}^{n}b_{i'}{a'}_{i'}+1-\mu-\varepsilon b_{i}\rig\}\text{.}
\end{multline*}
\end{proof}

\begin{prop}\label{vfgamma}
Identify ${W\lef(k\rig)\lef[\underline{X}\rig]}^{\dagger}$ with $W^{\dagger}\Omega^{\mathrm{int},0}_{k\lef[\underline{X}\rig]/k}$ as $W\lef(k\rig)$-modules by virtue of proposition \ref{lazardbijectivity}. Let $b\in\lef]0,+\infty\rig[^{n}$ and $\mu>0$. Then there exists $\delta>0$ such that for any $\varepsilon\in\lef]0,\delta\rig]$ we have:
\begin{align*}
\forall x\in{W\lef(k\rig)\lef[\underline{X}\rig]}^{\dagger},\ \gamma_{\varepsilon,\operatorname{Id}_{k\lef[\underline{X}\rig]},b}\lef(v_{F}\lef(x\rig)\rig)&\geqslant1-\mu+\gamma_{\varepsilon,\operatorname{Id}_{k\lef[\underline{X}\rig]},b}\lef(x\rig)\text{,}\\
\forall x\in{W\lef(k\rig)\lef[\underline{X}\rig]}^{\dagger},\ \zeta_{\varepsilon}\lef(v_{F}\lef(x\rig)\rig)&\geqslant1-\mu+\zeta_{\varepsilon}\lef(x\rig)\text{.}
\end{align*}
\end{prop}

\begin{proof}
One can write $x\in{W\lef(k\rig)\lef[\underline{X}\rig]}^{\dagger}$ as a $p$-adically convergent series:
\begin{equation*}
x=\sum_{a\colon\lef\llbracket1,n\rig\rrbracket\to\mathbb{N}}\eta_{a}\underline{X}^{a}
\end{equation*}
where $\eta_{a}\in W\lef(k\rig)$ for any weight function $a\colon\lef\llbracket1,n\rig\rrbracket\to\mathbb{N}$. Consider now a finite subset $S\subset\lef\{a\colon\lef\llbracket1,n\rig\rrbracket\to\mathbb{N}\rig\}$. According to proposition \ref{mingammaepsilon}, we can fix some $\delta>0$ such that for all $\varepsilon\in\lef]0,\delta\rig]$ we have the following inequalities which arise from the pseudovaluation properties of $\gamma_{\varepsilon,\operatorname{Id}_{k\lef[\underline{X}\rig]},b}$ \cite[proposition 2.3]{overconvergentwittvectors}:
\begin{equation*}
\begin{split}
\gamma_{\varepsilon,\operatorname{Id}_{k\lef[\underline{X}\rig]},b}\lef(v_{F}\lef(\sum_{a\in S}\eta_{a}\underline{X}^{a}\rig)\rig)&\geqslant\min_{a\in S}\lef\{\gamma_{\varepsilon,\operatorname{Id}_{k\lef[\underline{X}\rig]},b}\lef(\eta_{a}\rig)+\gamma_{\varepsilon,\operatorname{Id}_{k\lef[\underline{X}\rig]},b}\lef(v_{F}\lef(\underline{X}^{a}\rig)\rig)\rig\}\\
&\geqslant\min_{a\in S}\lef\{\operatorname{v}_{V}\lef(\eta_{a}\rig)+1-\mu-\varepsilon\sum_{i=1}^{n}b_{i}a_{i}\rig\}\\
&\geqslant1-\mu+\gamma_{\varepsilon,\operatorname{Id}_{k\lef[\underline{X}\rig]},b}\lef(\sum_{a\in S}\eta_{a}\underline{X}^{a}\rig)\\
\gamma_{\varepsilon,\operatorname{Id}_{k\lef[\underline{X}\rig]},b}\lef(v_{F}\lef(\sum_{a\in S}\eta_{a}\underline{X}^{a}\rig)\rig)&\geqslant1-\mu+\gamma_{\varepsilon,\operatorname{Id}_{k\lef[\underline{X}\rig]},b}\lef(x\rig)\text{.}
\end{split}
\end{equation*}

But the $p$-adic convergence of the series implies its $V$-adic convergence. So one can keep this lower bound by definition of $\gamma_{\varepsilon,\operatorname{Id}_{k\lef[\underline{X}\rig]},b}$ for $v_{F}\lef(x\rig)$.

The proof for $\zeta_{\varepsilon}$ goes exactly the same way, with $b=\lef(1,\ldots,1\rig)$ in order to apply proposition \ref{gammaandzetawitt} in the second inequality.
\end{proof}

The last lemma tells us that lifts of relatively perfect morphisms behave well with the Lazard morphism.

\begin{lemm}\label{compatibleetalefrobenii}
Let $\overline{\varphi}\colon\overline{R}\to\overline{S}$ be a relatively perfect morphism of smooth commutative $k$-algebras. Let $S$ be a smooth commutative $W\lef(k\rig)$-algebra lifting $\overline{S}$. Then, there is a lift of the étale morphism $\varphi\colon R^{\dagger}\to S^{\dagger}$ and a lift of the Frobenius endomorphism $G\colon S^{\dagger}\to S^{\dagger}$ such that the following square is commutative:
\begin{equation*}
\begin{tikzcd}
R^{\dagger}\arrow[r,"F"]\arrow[d,"\varphi"]&R^{\dagger}\arrow[d,"\varphi"]\\
S^{\dagger}\arrow[r,"G"]&S^{\dagger}\text{.}
\end{tikzcd}
\end{equation*}

In particular, the following diagram is also commutative:
\begin{equation*}
\begin{tikzcd}
R^{\dagger}\arrow[r,"t_{F}"]\arrow[d,"\varphi"]&W\lef(\overline{R}\rig)\arrow[d,"W\lef(\overline{\varphi}\rig)"]\\
S^{\dagger}\arrow[r,"t_{G}"]&W\lef(\overline{S}\rig)\text{.}
\end{tikzcd}
\end{equation*}
\end{lemm}

\begin{proof}
First lift $\overline{\varphi}$ to a morphism of $W\lef(k\rig)$-algebras $\varphi'\colon R^{\dagger}\to S^\dagger$. Modulo $p$, the first diagram we want to get is cocartesian by assumption. But we can lift $\overline{R}\otimes_{\operatorname{Frob}_{\overline{R}},\overline{\varphi}}\overline{S}$ to $\lef(R^{\dagger}\otimes_{F,\varphi'}S^{\dagger}\rig)^{\dagger}$, which is a weakly complete finitely generated $W\lef(k\rig)$-algebra in the sense of Monsky--Washnitzer; see for instance the proof of \cite[theorem 3.4]{formalcohomologyi}.

By \cite[3.3.2]{relevementdesalgebreslissesetdeleursmorphismes}, such lifts are unique; that is, there is an isomorphism of $W\lef(k\rig)$-algebras $S^{\dagger}\cong\lef(R^{\dagger}\otimes_{F,\varphi'}S^{\dagger}\rig)^{\dagger}$. And we have a commutative diagram:
\begin{equation*}
\begin{tikzcd}[column sep=huge]
R^{\dagger}\arrow[r,"F"]\arrow[d]&R^{\dagger}\arrow[d]\\
\lef(R^{\dagger}\otimes_{F,\varphi'}S^{\dagger}\rig)^{\dagger}\arrow[r,"\lef(F\otimes\operatorname{Id}_{S^{\dagger}}\rig)^{\dagger}"]&\lef(R^{\dagger}\otimes_{F,\varphi'}S^{\dagger}\rig)^{\dagger}\text{.}
\end{tikzcd}
\end{equation*}

The horizontal maps are given by the coprojection.
\end{proof}

\section{Structure of the overconvergent Witt vectors ring}

Throughout this section, we will consider a perfect commutative $\mathbb{F}_{p}$-algebra $k$ such that $W\lef(k\rig)$ is Noetherian, and a relatively perfect morphism of smooth commutative $k$-algebras $\psi\colon k\lef[\underline{X}\rig]\to\overline{R}$. As in the previous sections, we will lift smoothly $\overline{R}$ to an $W\lef(k\rig)$-algebra $R$. Except where otherwise stated, we consider a lift $F\colon R^{\dagger}\to R^{\dagger}$ of $\operatorname{Frob}_{\overline{R}}$ giving rise to:
\begin{equation*}
t_{F}\colon\Omega^{\mathrm{sep}}_{R^{\dagger}/W\lef(k\rig)}\to W^{\dagger}\Omega_{\overline{R}/k}\text{.}
\end{equation*}

By definition, the following commutative diagram is a cocartesian square of $k\lef[\underline{X}\rig]$-algebras:
\begin{equation}\label{cocartesiansquare}
\begin{tikzcd}[column sep=huge]
k\lef[\underline{X}\rig]\arrow[r,"\operatorname{Frob}_{k\lef[\underline{X}\rig]}"]\arrow[d,"\psi"]&k\lef[\underline{X}\rig]\arrow[d,"\psi"]\\
\overline{R}\arrow[r,"\operatorname{Frob}_{\overline{R}}"]&\overline{R}\text{.}
\end{tikzcd}
\end{equation}

According to \cite[0EBS]{stacksproject}, étale morphisms are relatively perfect. This implies that we can cover a smooth variety over $k$ with affine spectra $\operatorname{Spec}\lef(\overline{R}\rig)$ with $\overline{R}$ as above. In this section, our goal will be to give the structure, as an $R^{\dagger}$-algebra, of the ring $W^{\dagger}\lef(\overline{R}\rig)$ of overconvergent Witt vectors with coefficients in such rings.

The above square also implies that there is an isomorphism of $k\lef[\underline{X}\rig]$-algebras for every $l\in\mathbb{N}^{*}$:
\begin{equation}\label{etalefrobeniusdecomposition}
\overline{R}\cong\bigoplus_{\substack{a\colon\lef\llbracket1,n\rig\rrbracket\to\mathbb{N}\lef[\frac{1}{p}\rig]\\\forall i\in\lef\llbracket1,n\rig\rrbracket,\ 0\leqslant a_{i}<1\\\operatorname{v}_{p}\lef(a\rig)\geqslant-l}}\overline{R}\otimes_{k\lef[\underline{X}\rig]}\psi\lef(\underline{X}^{p^{l}a}\rig){\operatorname{Frob}_{k\lef[\underline{X}\rig]}}^{l}\lef(k\lef[\underline{X}\rig]\rig)\text{.}
\end{equation}

These decompositions translate as follows on the ring of overconvergent Witt vectors.

\begin{prop}\label{divisionbypsum}
With the above notations, consider an overconvergent series:
\begin{equation*}
x=\sum_{\substack{a\colon\lef\llbracket1,n\rig\rrbracket\to\mathbb{N}\lef[\frac{1}{p}\rig]\\\forall i\in\lef\llbracket1,n\rig\rrbracket,\ 0\leqslant a_{i}<1}}t_{F}\lef(r_{a}\rig)W\lef(\psi\rig)\lef(e\lef(1,a,\emptyset\rig)\rig)\in W^{\dagger}\lef(\overline{R}\rig)
\end{equation*}
where $r_{a}\in R^{\dagger}$ for any weight function $a$ in the series. Then, for all $l\in\mathbb{N}$, one has:
\begin{equation*}
x\in V^{l}\lef(W^{\dagger}\lef(\overline{R}\rig)\rig)\iff\forall a\colon\lef\llbracket1,n\rig\rrbracket\to\mathbb{N}\lef[\frac{1}{p}\rig],\ p^{\max\lef\{l+\operatorname{v}_{p}\lef(a\rig),0\rig\}}\mid r_{a}\text{.}
\end{equation*}
\end{prop}

\begin{proof}
The condition is clearly sufficient, because we have $pw=V\lef(F\lef(w\rig)\rig)$ for any Witt vector $w$ over a ring of characteristic $p$.

We show the implication by induction on $l\in\mathbb{N}$, the case $l=0$ being obvious. So let $l\in\mathbb{N}$ satisfying the proposition, and assume that:
\begin{equation*}
x=\sum_{\substack{a\colon\lef\llbracket1,n\rig\rrbracket\to\mathbb{N}\lef[\frac{1}{p}\rig]\\\forall i\in\lef\llbracket1,n\rig\rrbracket,\ 0\leqslant a_{i}<1}}t_{F}\lef(r_{a}\rig)W\lef(\psi\rig)\lef(e\lef(1,a,\emptyset\rig)\rig)\in V^{l+1}\lef(W^{\dagger}\lef(\overline{R}\rig)\rig)\text{.}
\end{equation*}

By induction hypothesis, one has $r_{a}=p^{\max\lef\{l-u\lef(a\rig),0\rig\}}r_{a}'$ for any $a$ in the series, where all $r_{a}'\in R^{\dagger}$. With the usual arguments on Witt vectors in characteristic $p$, this implies that:
\begin{equation*}
x_{l}=\sum_{\substack{a\colon\lef\llbracket1,n\rig\rrbracket\to\mathbb{N}\lef[\frac{1}{p}\rig]\\\forall i\in\lef\llbracket1,n\rig\rrbracket,\ 0\leqslant a_{i}<1\\\operatorname{v}_{p}\lef(a\rig)\geqslant-l}}\overline{r_{a}'}^{p^{l}}\psi\lef(\underline{X}^{p^{l}a}\rig)=0\text{.}
\end{equation*}

In particular, using cocartesian square \eqref{cocartesiansquare}, we deduce that $\overline{r_{a}'}^{p^{l}}=0$ for all $a$ in the sum by \eqref{etalefrobeniusdecomposition}. As $\overline{R}$ is reduced \cite[033B]{stacksproject}, this concludes our proof.
\end{proof}

The next proposition details the $R^{\dagger}$-algebra structure of the overconvergent Witt vectors ring in the case where $R=W\lef(k\rig)\lef[\underline{X}\rig]$. We will reduce to this case to prove the general one.

\begin{prop}\label{wittpolynomialstructure}
Assume that $R=W\lef(k\rig)\lef[\underline{X}\rig]$. Denote by $\mathcal{A}$ the set of all weight functions $a\colon\lef\llbracket1,n\rig\rrbracket\to\mathbb{N}\lef[\frac{1}{p}\rig]$ such that $0\leqslant a_{i}<1$ for all $i\in\lef\llbracket1,n\rig\rrbracket$.

Let $b=\lef(b_{i}\rig)_{i\in\lef\llbracket1,n\rig\rrbracket}\in\lef]0,+\infty\rig[^{n}$. Then, there is $\delta>0$ such that for any $\varepsilon\in\lef]0,\delta\rig]$ and all $w\in W^{\dagger}\lef(k\lef[\underline{X}\rig]\rig)$ there exists a unique function:
\begin{equation*}
r\colon\begin{array}{rl}\mathcal{A}\to&{W\lef(k\rig)\lef[\underline{X}\rig]}^{\dagger}\\a\mapsto&r_{a}\end{array}
\end{equation*}
such that:
\begin{gather*}
w=\sum_{a\in\mathcal{A}}t_{F}\lef(r_{a}\rig)e\lef(1,a,\emptyset\rig)\text{,}\\
\forall a\in\mathcal{A},\ \gamma_{\varepsilon,\operatorname{Id}_{k\lef[\underline{X}\rig]},b}\lef(t_{F}\lef(r_{a}\rig)\rig)+\gamma_{\varepsilon,\operatorname{Id}_{k\lef[\underline{X}\rig]},b}\lef(e\lef(1,a,\emptyset\rig)\rig)\geqslant\gamma_{\varepsilon,\operatorname {Id}_{k\lef[\underline{X}\rig]},b}\lef(w\rig)\text{.}
\end{gather*}
\end{prop}

\begin{proof}
The unicity is given by proposition \ref{divisionbypsum} and $p$-adic separatedness, so we only have to focus on the existence of $r$. We shall use the following map:
\begin{equation*}
\rho\colon\begin{array}{rl}k\lef[\underline{X}\rig]\to&{W\lef(k\rig)\lef[\underline{X}\rig]}^{\dagger}\\\sum_{a\in\mathbb{N}^{n}}c_{a}\underline{X}^{a}\mapsto&\sum_{a\in\mathbb{N}^{n}}\lef[c_{a}\rig]\underline{X}^{a}\end{array}\text{,}
\end{equation*}
where we have put $c_{a}\in k$ for all $a\in\mathbb{N}^{n}$.

Let $w\in W^{\dagger}\lef(k\lef[\underline{X}\rig]\rig)$. We derive from proposition \ref{mingammaepsilon} that there is $\delta>0$ such that for any $\varepsilon\in\lef]0,\delta\rig]$ we get for any $P\in k\lef[\underline{X}\rig]$:
\begin{equation}\label{vbandgammaepsilonb}
\gamma_{\varepsilon,\operatorname{Id}_{k\lef[\underline{X}\rig]},b}\lef(t_{F}\lef(\rho\lef(P\rig)\rig)\rig)\geqslant v_{b}\lef(P\rig)\text{.}
\end{equation}

Hence, for any such $\varepsilon\in\lef]0,\delta\rig]$ we have $\gamma_{\varepsilon,\operatorname{Id}_{k\lef[\underline{X}\rig]},b}\lef(t_{F}\lef(\rho\lef(w_{0}\rig)\rig)\rig)\geqslant\gamma_{\varepsilon,\operatorname{Id}_{k\lef[\underline{X}\rig]},b}\lef(w\rig)$. So we can find an element $w^{\lef(1\rig)}\coloneqq w-t_{F}\lef(\rho\lef(w_{0}\rig)\rig)\in V\lef(W^{\dagger}\lef(k\lef[\underline{X}\rig]\rig)\rig)$ such that $\gamma_{\varepsilon,\operatorname{Id}_{k\lef[\underline{X}\rig]},b}\lef(w^{\lef(1\rig)}\rig)\geqslant\gamma_{\varepsilon,\operatorname{Id}_{k\lef[\underline{X}\rig]},b}\lef(w\rig)$ holds.

Let $l\in\mathbb{N}^{*}$. Assume given $w^{\lef(l\rig)}\in V^{l}\lef(W^{\dagger}\lef(k\lef[\underline{X}\rig]\rig)\rig)$ such that for all $\varepsilon\in\lef]0,\delta\rig]$ we have:
\begin{equation*}
\gamma_{\varepsilon,\operatorname{Id}_{k\lef[\underline{X}\rig]},b}\lef(w^{\lef(l\rig)}\rig)\geqslant\gamma_{\varepsilon,\operatorname{Id}_{k\lef[\underline{X}\rig]},b}\lef(w\rig)\text{,}
\end{equation*}
and for all $a\in\mathcal{A}$ satisfying $\operatorname{v}_{p}\lef(a\rig)>-l$ there is $r^{\lef(l\rig)}_{a}\in{W\lef(k\rig)\lef[\underline{X}\rig]}^{\dagger}$ such that:
\begin{gather*}
\operatorname{v}_{p}\lef(r^{\lef(l\rig)}_{a}-r^{\lef(l+1\rig)}_{a}\rig)\geqslant l+\operatorname{v}_{p}\lef(a\rig)\text{,}\\
w-w^{\lef(l\rig)}=\sum_{\substack{a\in\mathcal{A}\\\operatorname{v}_{p}\lef(a\rig)>-l}}t_{F}\lef(r^{\lef(l\rig)}_{a}\rig)e\lef(1,a,\emptyset\rig)\text{,}\\
\gamma_{\varepsilon,\operatorname{Id}_{k\lef[\underline{X}\rig]},b}\lef(t_{F}\lef(r^{\lef(l\rig)}_{a}\rig)\rig)+\gamma_{\varepsilon,\operatorname{Id}_{k\lef[\underline{X}\rig]},b}\lef(e\lef(1,a,\emptyset\rig)\rig)\geqslant\gamma_{\varepsilon,\operatorname{Id}_{k\lef[\underline{X}\rig]},b}\lef(w\rig)\text{.}
\end{gather*}

We have just seen that such elements exist when $l=1$, except for the condition on the $p$-adic pseudovaluation. As $k$ is perfect, we can write uniquely:
\begin{equation*}
{w^{\lef(l\rig)}}_{l}=\sum_{i=0}^{l}{\operatorname{Frob}_{k\lef[\underline{X}\rig]}}^{i}\lef(P_{l,i}\lef(\underline{X}\rig)\rig)\text{,}
\end{equation*}
where $P_{l,i}\lef(\underline{X}\rig)\in k\lef[\underline{X}\rig]$ for all $i\in\lef\llbracket0,l\rig\rrbracket$, and no monomial in $P_{l,i}\lef(\underline{X}\rig)$ for $i\neq l$ is in the image of the Frobenius endomorphism. This last condition implies that for all $i\in\lef\llbracket0,l-1\rig\rrbracket$ we have a unique decomposition:
\begin{equation*}
P_{l,i}\lef(\underline{X}\rig)=\sum_{\substack{a\in\mathcal{A}\\\operatorname{v}_{p}\lef(a\rig)=i-l}}\underline{X}^{p^{l-i}a}{\operatorname{Frob}_{k\lef[\underline{X}\rig]}}^{l-i}\lef(P_{l,i,a}\lef(\underline{X}\rig)\rig)\text{,}
\end{equation*}
where all $P_{l,i,a}\lef(\underline{X}\rig)\in k\lef[\underline{X}\rig]$ for every $i\in\lef\llbracket0,l-1\rig\rrbracket$ and every weight function $a\in\mathcal{A}$ with $\operatorname{v}_{p}\lef(a\rig)=i-l$. To sum up, we have a unique writing of the following form:
\begin{equation*}
{w^{\lef(l\rig)}}_{l}={\operatorname{Frob}_{k\lef[\underline{X}\rig]}}^{l}\lef(P_{l,l}\lef(\underline{X}\rig)\rig)+\sum_{i=0}^{l-1}\sum_{\substack{a\in\mathcal{A}\\\operatorname{v}_{p}\lef(a\rig)=i-l}}\underline{X}^{p^{l}a}{\operatorname{Frob}_{k\lef[\underline{X}\rig]}}^{l}\lef(P_{l,i,a}\lef(\underline{X}\rig)\rig)\text{.}
\end{equation*}

Moreover, the polynomials therein satisfy by definition of $\gamma_{\varepsilon,\operatorname{Id}_{k\lef[\underline{X}\rig]},b}$ the following inequalities:
\begin{equation*}
l-\varepsilon\sum_{t=1}^{n}b_{t}a_{t}-\varepsilon v_{b}\lef(P_{l,i,a}\lef(\underline{X}\rig)\rig)\geqslant\gamma_{\varepsilon,\operatorname{Id}_{k\lef[\underline{X}\rig]},b}\lef(w^{\lef(l\rig)}\rig)\geqslant\gamma_{\varepsilon,\operatorname{Id}_{k\lef[\underline{X}\rig]},b}\lef(w\rig)\text{.}
\end{equation*}

We also have $l-\varepsilon v_{b}\lef(P_{l,l}\lef(\underline{X}\rig)\rig)\geqslant\gamma_{\varepsilon,\operatorname{Id}_{k\lef[\underline{X}\rig]},b}\lef(w\rig)$. Therefore, using \eqref{vxfyvxy} and the fact that multiplication by $p$ in Witt vectors equals $V\circ W\lef(\operatorname{Frob}_{k\lef[\underline{X}\rig]}\rig)$, we are able to find $w^{\lef(l+1\rig)}\in V^{l+1}\lef(W^{\dagger}\lef(k\lef[\underline{X}\rig]\rig)\rig)$ satisfying:
\begin{multline*}
w^{\lef(l\rig)}-w^{\lef(l+1\rig)}=t_{F}\lef(p^{l}\rho\lef(P_{l,l}\lef(\underline{X}\rig)\rig)\rig)\\+\sum_{i=0}^{l-1}\sum_{\substack{a\in\mathcal{A}\\\operatorname{v}_{p}\lef(a\rig)=i-l}}V^{l-i}\lef(\lef[\underline{X}^{p^{l-i}a}\rig]\rig)t_{F}\lef(p^{i}\rho\lef(P_{l,i,a}\lef(\underline{X}\rig)\rig)\rig)\text{.}
\end{multline*}

By definition, $V^{l-i}\lef(\lef[\underline{X}^{p^{l-i}a}\rig]\rig)=e\lef(1,a,\emptyset\rig)$ for all $a$ in the double sum above, so in particular $\gamma_{\varepsilon,\operatorname{Id}_{k\lef[\underline{X}\rig]},b}\lef(e\lef(1,a,\emptyset\rig)\rig)=l-i-\varepsilon\sum_{t=1}^{n}b_{t}a_{t}$. Hence, using \eqref{multiplicationbypgamma} and \eqref{vbandgammaepsilonb}, we get by induction on $l\in\mathbb{N}^{*}$ that the $w^{\lef(l\rig)}$ satisfying the above conditions always exist, since $\gamma_{\varepsilon,\operatorname{Id}_{k\lef[\underline{X}\rig]},b}$ is a pseudovaluation \cite[proposition 2.3]{overconvergentwittvectors}. This concludes the proof by $p$-adic overconvergence.
\end{proof}

To conclude this section, we spell out the $R^{\dagger}$-algebra structure of the ring of overconvergent Witt vectors in the case of a finite free étale algebra over a localisation of a polynomial ring. The author thanks here Andreas Langer, who explained to him that it was not necessary to localise the polynomial algebra in order to give a similar description of the overconvergent de Rham--Witt complex, locally on a smooth variety. As this part is quite technical, the reader may want to use this remark to simplify the proof.

However, in some cases, this work appears to be useful. We give a small application in the last section.

We let $P\in k\lef[\underline{X}\rig]$. We assume given a finite free étale morphism of commutative $k$-algebras, as in proposition \ref{characteriserelativelyperfect}:
\begin{equation*}
{k\lef[\underline{X}\rig]}_{\lef\langle P\rig\rangle}\to\overline{R}\text{.}
\end{equation*}

We will consider $\lef(\overline{r_{i}}\rig)_{i\in\lef\llbracket1,m\rig\rrbracket}$ a basis of $\overline{R}$ as a ${k\lef[\underline{X}\rig]}_{\lef\langle P\rig\rangle}$-module, where $m\in\mathbb{N}$ is the rank of $\overline{R}$. In what follows, we shall write:
\begin{align*}
k\lef[\underline{T}\rig]&\coloneqq k\lef[X_{1},\ldots,X_{n},Y,Z_{1},\ldots,Z_{m}\rig]\text{,}\\
W\lef(k\rig)\lef[\underline{T}\rig]&\coloneqq W\lef(k\rig)\lef[X_{1},\ldots,X_{n},Y,Z_{1},\ldots,Z_{m}\rig]\text{.}
\end{align*}

We denote by $\mathcal{A}$, $\mathcal{A}'$ and $\mathcal{A}_{\mathrm{ext}}$ the sets of weight functions $a$ with $n$, $n+1$ and $n+1+m$ variables respectively and with values in $\lef[0,1\rig[$.

We shall change our previous notation, and consider the following lift of the Frobenius endomorphism:
\begin{equation*}
F\colon{W\lef(k\rig)\lef[\underline{T}\rig]}^{\dagger}\to{W\lef(k\rig)\lef[\underline{T}\rig]}^{\dagger}\text{.}
\end{equation*}

\begin{prop}\label{overconvergentwittvectorsstructure}
Let $\varphi\colon k\lef[\underline{T}\rig]\to\overline{R}$ be the surjective morphism of $k\lef[\underline{X}\rig]$-algebras such that $\varphi\lef(Y\rig)=P^{-1}$ and $\varphi\lef(Z_{i}\rig)=\overline{r_{i}}$ for each $i\in\lef\llbracket1,m\rig\rrbracket$.

Then for all $\eta>0$, there exists $\lef(b_{i}\rig)_{i\in\lef\llbracket1,n+m\rig\rrbracket}\in{\lef]0,+\infty\rig[}^{n+1+m}$ and $\delta>0$ such that for every $u\in\mathbb{N}$ and every $w\in V^{u}\lef(W^{\dagger}\lef(\overline{R}\rig)\rig)$ there is a map:
\begin{equation*}
h\colon\begin{array}{rl}\mathcal{A}\to&{W\lef(k\rig)\lef[\underline{T}\rig]}^{\dagger}\\a\mapsto&h_{a}\end{array}\text{,}
\end{equation*}
such that:
\begin{gather*}
W\lef(\varphi\rig)\lef(\sum_{a\in\mathcal{A}}t_{F}\lef(h_{a}\rig)e\lef(1,a,\emptyset\rig)\rig)=w\text{,}\\
\forall\varepsilon\in\lef]0,\delta\rig],\ \forall a\in\mathcal{A},\ \gamma_{\varepsilon,\operatorname{Id}_{k\lef[\underline{T}\rig]},b}\lef(t_{F}\lef(h_{a}\rig)e\lef(1,a,\emptyset\rig)\rig)\geqslant\gamma_{\varepsilon,\varphi,b}\lef(w\rig)-\eta\text{,}\\
\forall a\in\mathcal{A},\ p^{\max\lef\{u-u\lef(a\rig),0\rig\}}\mid h_{a}\text{.}
\end{gather*}
\end{prop}

\begin{proof}
We identify $W^{\dagger}\Omega^{\mathrm{int},0}_{k\lef[\underline{T}\rig]/k}$ with ${W\lef(k\rig)\lef[\underline{T}\rig]}^{\dagger}$ using lemma \ref{isodegreezero}. By proposition \ref{wittpolynomialstructure}, for all $x\in W^{\dagger}\lef(k\lef[\underline{T}\rig]\rig)$ there is a unique function:
\begin{equation*}
h\lef(x\rig)\colon\begin{array}{rl}\mathcal{A}_{\mathrm{ext}}\to&{W\lef(k\rig)\lef[\underline{T}\rig]}^{\dagger}\\a\mapsto&{h\lef(x\rig)}_{a}\end{array}\text{,}
\end{equation*}
such that:
\begin{equation}\label{xandrax}
x=\sum_{a\in\mathcal{A}_{\mathrm{ext}}}{h\lef(x\rig)}_{a}e\lef(1,a,\emptyset\rig)\text{.}
\end{equation}

We may assume without loss of generality that $\eta<1$. In this proof, we let $\varepsilon>0$ be a constant satisfying the following conditions: we must have $\varepsilon\leqslant\delta$, where $\delta$ is the smallest of the ones given by propositions \ref{vfgamma}, \ref{wittpolynomialstructure} and \ref{decompositionfreewitt}, and we keep from the latter $\lef(b_{i}\rig)_{i\in\lef\llbracket1,n+1+m\rig\rrbracket}\in{\lef]0,+\infty\rig[}^{n+1+m}$ and the constant $E\geqslant0$; in particular, we have:
\begin{align}
\forall x\in {W\lef(k\rig)\lef[\underline{T}\rig]}^{\dagger},\ &\gamma_{\varepsilon,\operatorname{Id}_{k\lef[\underline{T}\rig]},b}\lef(t_{F}\lef(x\rig)\rig)\geqslant\gamma_{\varepsilon,\operatorname{Id}_{k\lef[\underline{T}\rig]},b}\lef(x\rig)\text{,}\label{tgnormegauss}\\
\forall x\in {W\lef(k\rig)\lef[\underline{T}\rig]}^{\dagger},\ &\gamma_{\varepsilon,\operatorname{Id}_{k\lef[\underline{T}\rig]},b}\lef(v_{F}\lef(x\rig)\rig)\geqslant\gamma_{\varepsilon,\operatorname{Id}_{k\lef[\underline{T}\rig]},b}\lef(x\rig)+\frac{\eta}{2}\text{,}\label{tgnormegaussbis}\\
\forall x\in W^{\dagger}\lef(k\lef[\underline{T}\rig]\rig),\ \forall a\in\mathcal{A}_{\mathrm{ext}},\ &\begin{multlined}\gamma_{\varepsilon,\operatorname{Id}_{k\lef[\underline{T}\rig]},b}\lef({h\lef(x\rig)}_{a}\rig)+\gamma_{\varepsilon,\operatorname{Id}_{k\lef[\underline{T}\rig]},b}\lef(e\lef(1,a,\emptyset\rig)\rig)\\\geqslant\gamma_{\varepsilon,\operatorname{Id}_{k\lef[\underline{T}\rig]},b}\lef(x\rig)\text{.}\end{multlined}\label{conditionsurlesk}
\end{align}

Let $\lef(r_{i}\rig)_{i\in\lef\llbracket1,m\rig\rrbracket}\in{R^{\dagger}}^{m}$ be elements lifting the basis $\lef(\overline{r_{i}}\rig)_{i\in\lef\llbracket1,m\rig\rrbracket}$. We shall also fix antecedents $\lef(\beta_{i}\rig)_{i\in\lef\llbracket1,m\rig\rrbracket}\in{{W\lef(k\rig)\lef[\underline{T}\rig]}^{\dagger}}^{m}$ by $W\lef(\varphi\rig)$ of $\lef(r_{i}\rig)_{i\in\lef\llbracket1,m\rig\rrbracket}$. Going back to $\delta$, we also assume that the following constant:
\begin{equation}\label{constanteetapreuve}
\zeta\coloneqq\min\lef\{\gamma_{\varepsilon,\operatorname{Id}_{k\lef[\underline{T}\rig]},b}\lef(\beta_{i}\rig)\mid i\in\lef\llbracket1,m\rig\rrbracket\rig\}\text{,}
\end{equation}
satisfies:
\begin{equation}\label{conditiontechniqueepsilon}
\varepsilon\lef(v_{b}\lef(P\rig)-b_{n+1}-E\rig)+\zeta\geqslant-\frac{\eta}{4}\text{.}
\end{equation}

Let $u\in\mathbb{N}$ and $w\in V^{u}\lef(W^{\dagger}\lef(\overline{R}\rig)\rig)$. We can find $\rho\lef(0\rig)\in V^{u}\lef(W^{\dagger}\lef(k\lef[\underline{T}\rig]\rig)\rig)$ such that:
\begin{align*}
W\lef(\varphi\rig)\lef(\rho\lef(0\rig)\rig)&=w\text{,}\\
\gamma_{\varepsilon,\operatorname{Id}_{k\lef[\underline{T}\rig]},b}\lef(\rho\lef(0\rig)\rig)&\geqslant\gamma_{\varepsilon,\varphi,b}\lef(w\rig)-\frac{\eta}{2}\text{.}
\end{align*}

Let $s\in\mathbb{N}$. Assume that for any $s'\in\lef\llbracket0,s\rig\rrbracket$ we have a map:
\begin{equation*}
{w\lef(s'\rig)}_{\bullet}\colon\begin{array}{rl}\mathcal{A}\to&{W\lef(k\rig)\lef[\underline{T}\rig]}^{\dagger}\\a\mapsto&{w\lef(s'\rig)}_{a}\end{array}\text{,}
\end{equation*}
and an element $\rho\lef(s'\rig)\in W^{\dagger}\lef(k\lef[\underline{T}\rig]\rig)$ such that for:
\begin{equation}\label{defwsprime}
w\lef(s'\rig)\coloneqq\sum_{a\in\mathcal{A}}t_{F}\lef({w\lef(s'\rig)}_{a}\rig)e\lef(1,a,\emptyset\rig)\text{,}
\end{equation}
we have:
\begin{align}
\forall s'\in\lef\llbracket0,s\rig\rrbracket,\ \forall a\in\mathcal{A},\ &\begin{multlined}\gamma_{\varepsilon,\operatorname{Id}_{k\lef[\underline{T}\rig]},b}\lef(t_{F}\lef({w\lef(s'\rig)}_{a}\rig)\rig)+\gamma_{\varepsilon,\operatorname{Id}_{k\lef[\underline{T}\rig]},b}\lef(e\lef(1,a,\emptyset\rig)\rig)\\\geqslant\gamma_{\varepsilon,\operatorname{Id}_{k\lef[\underline{T}\rig]},b}\lef(\rho\lef(0\rig)\rig)-\frac{\eta}{2}\text{,}\end{multlined}\label{localipreuzero}\\
\forall s'\in\lef\llbracket0,s\rig\rrbracket,\ &\gamma_{\varepsilon,\operatorname{Id}_{k\lef[\underline{T}\rig]},b}\lef(\rho\lef(s'\rig)\rig)\geqslant\gamma_{\varepsilon,\operatorname{Id}_{k\lef[\underline{T}\rig]},b}\lef(\rho\lef(0\rig)\rig)\text{,}\label{localipreuun}\\
\forall s'\in\lef\llbracket0,s\rig\rrbracket,\ &W\lef(\varphi\rig)\lef(w\lef(s'\rig)+\rho\lef(s'\rig)\rig)=w\text{,}\label{localipreudeux}\\
\forall s'\in\lef\llbracket1,s\rig\rrbracket,\ &w\lef(s'\rig)-w\lef(s'-1\rig)\in V^{s'-1}\lef(W^{\dagger}\lef(k\lef[\underline{T}\rig]\rig)\rig)\text{,}\label{localipreuquatre}\\
\forall s'\in\lef\llbracket1,s\rig\rrbracket,\ &w\lef(s'\rig)\in V^{u}\lef(W^{\dagger}\lef(k\lef[\underline{T}\rig]\rig)\rig)\text{,}\label{localipreucinq}\\
\forall s'\in\lef\llbracket0,s\rig\rrbracket,\ &\rho\lef(s'\rig)\in V^{\max\lef\{s',u\rig\}}\lef(W^{\dagger}\lef(k\lef[\underline{T}\rig]\rig)\rig)\text{.}\label{localipreusix}
\end{align}

Notice that the zero function and $\rho\lef(0\rig)$ satisfy all these conditions for $s=0$. So now let us assume that these conditions are met for some $s\in\mathbb{N}$, and let us construct $w\lef(s+1\rig)$ and $\rho\lef(s+1\rig)$.

Let $k\lef[\underline{X},Y\rig]\coloneqq k\lef[X_{1},\ldots,X_{n},Y\rig]$. For any weight function $a\in\mathcal{A}'$, we shall consider $f\lef(a\rig)\coloneqq V^{u\lef(a\rig)}\lef(\lef[\underline{X}^{a|_{\lef\llbracket1,n\rig\rrbracket}}P^{p^{u\lef(a\rig)}\lef(1-a_{n+1}\rig)}\rig]\rig)$. Since $\varphi\lef(Y\rig)=P^{-1}$, we find that:
\begin{align}
W\lef(\varphi\rig)\lef(e\lef(1,a,\emptyset\rig)\rig)&=W\lef(\varphi\rig)\lef(f\lef(a\rig)\lef[Y\rig]\rig)\text{,}\label{fetfprimepareil}\\
\gamma_{\varepsilon,\operatorname{Id}_{k\lef[\underline{T}\rig]},b}\lef(f\lef(a\rig)\rig)&\geqslant\gamma_{\varepsilon,\operatorname{Id}_{k\lef[\underline{T}\rig]},b}\lef(e\lef(1,a,\emptyset\rig)\rig)+\varepsilon v_{b}\lef(P\rig)\text{.}\label{fprimenormegauss}
\end{align}

Fix some:
\begin{equation}\label{definitiond}
C\in\lef]0,\frac{\eta}{4}\rig]\text{.}
\end{equation}

Proposition \ref{decompositionfreewitt} and \eqref{localipreusix} yield the existence for all $i\in\lef\llbracket1,n\rig\rrbracket$ of Witt vectors $\rho\lef(s,i\rig)\in V^{\max\lef\{s,u\rig\}}\lef(W^{\dagger}\lef(k\lef[\underline{X},Y\rig]\rig)\rig)$ such that:
\begin{align}
\gamma_{\varepsilon,\operatorname{Id}_{k\lef[\underline{X},Y\rig]},\lef(b_{j}\rig)_{j\in\lef\llbracket1,m+1\rig\rrbracket}}\lef(\rho\lef(s,i\rig)\rig)&\geqslant\gamma_{\varepsilon,\operatorname{Id}_{k\lef[\underline{T}\rig]},b}\lef(\rho\lef(s\rig)\rig)-\varepsilon E-C\text{,}\label{rhosinormegauss}\\
W\lef(\varphi\rig)\lef(\sum_{i=1}^{n}t_{F}\lef(\beta_{i}\rig)\rho\lef(s,i\rig)\rig)&=W\lef(\varphi\rig)\lef(\rho\lef(s\rig)\rig)\label{deftfrhosi}\text{.}
\end{align}

Since $\gamma_{\varepsilon,\operatorname{Id}_{k\lef[\underline{T}\rig]},b}$ is a pseudovaluation \cite[proposition 2.3]{overconvergentwittvectors}, we get for every $a\in\mathcal{A}'$ and every $i\in\lef\llbracket1,n\rig\rrbracket$:
\begin{equation}\label{csinormegausspre}
\begin{split}
\gamma_{\varepsilon,\operatorname{Id}_{k\lef[\underline{T}\rig]},b}\lef({h\lef(\rho\lef(s,i\rig)\rig)}_{a}f\lef(a\rig)\lef[Y\rig]\rig)&\overset{\hphantom{\eqref{fprimenormegauss}}}{\geqslant}\gamma_{\varepsilon,\operatorname{Id}_{k\lef[\underline{T}\rig]},b}\lef({h\lef(\rho\lef(s,i\rig)\rig)}_{a}\rig)+\gamma_{\varepsilon,\operatorname{Id}_{k\lef[\underline{T}\rig]},b}\lef(f\lef(a\rig)\rig)-\varepsilon b_{n+1}\\
&\overset{\eqref{conditionsurlesk}}{\overset{\eqref{fprimenormegauss}}{\geqslant}}\gamma_{\varepsilon,\operatorname{Id}_{k\lef[\underline{T}\rig]},b}\lef(\rho\lef(s,i\rig)\rig)+\varepsilon\lef(v_{b}\lef(P\rig)-b_{n+1}\rig)\\
\gamma_{\varepsilon,\operatorname{Id}_{k\lef[\underline{T}\rig]},b}\lef({h\lef(\rho\lef(s,i\rig)\rig)}_{a}f\lef(a\rig)\lef[Y\rig]\rig)&\overset{\eqref{localipreuun}}{\overset{\eqref{rhosinormegauss}}{\geqslant}}\gamma_{\varepsilon,\operatorname{Id}_{k\lef[\underline{T}\rig]},b}\lef(\rho\lef(0\rig)\rig)+\varepsilon\lef(v_{b}\lef(P\rig)-b_{n+1}-E\rig)-C\text{.}
\end{split}
\end{equation}

For any $i\in\lef\llbracket1,n\rig\rrbracket$, we will be considering the $V$-adically convergent series $c\lef(s,i\rig)\coloneqq\sum_{a\in\mathcal{A}'}{h\lef(\rho\lef(s,i\rig)\rig)}_{a}f\lef(a\rig)\lef[Y\rig]$. By construction, we have:
\begin{align}
\gamma_{\varepsilon,\operatorname{Id}_{k\lef[\underline{T}\rig]},b}\lef(c\lef(s,i\rig)\rig)&\overset{\eqref{conditiontechniqueepsilon}}{\overset{\eqref{definitiond}}{\overset{\eqref{csinormegausspre}}{\geqslant}}}\gamma_{\varepsilon,\operatorname{Id}_{k\lef[\underline{T}\rig]},b}\lef(\rho\lef(0\rig)\rig)-\frac{\eta}{2}-\zeta\text{,}\label{csinormegauss}\\
W\lef(\varphi\rig)\lef(c\lef(s,i\rig)\rig)&\overset{\eqref{xandrax}}{\overset{\eqref{fetfprimepareil}}{=}}W\lef(\varphi\rig)\lef(\rho\lef(s,i\rig)\rig)\text{.}\label{csigammasi}
\end{align}

Also, for every weight function $a\in\mathcal{A}_{\mathrm{ext}}$ which is not in $\mathcal{A}$ we must have ${h\lef(c\lef(s,i\rig)\rig)}_{a}=0$.

Moreover, for every $i\in\lef\llbracket1,n\rig\rrbracket$, we get from proposition \ref{divisionbypsum} and \eqref{localipreusix}:
\begin{equation}\label{csidansv}
c\lef(s,i\rig)\in V^{\max\lef\{s,u\rig\}}\lef(W^{\dagger}\lef(k\lef[\underline{X},Y\rig]\rig)\rig)\text{.}
\end{equation}

For any weight function $a\in\mathcal{A}$, let:
\begin{equation}\label{defwsplusun}
{w\lef(s+1\rig)}_{a}\coloneqq{w\lef(s\rig)}_{a}+\sum_{i=1}^{n}\beta_{i}{h\lef(c\lef(s,i\rig)\rig)}_{a}\text{.}
\end{equation}

Since $\gamma_{\varepsilon,\operatorname{Id}_{k\lef[\underline{T}\rig]},b}$ is a pseudovaluation, for any weight function $a\in\mathcal{A}$ and every $i\in\lef\llbracket1,n\rig\rrbracket$ we have:
\begin{equation*}
\begin{split}
\gamma_{\varepsilon,\operatorname{Id}_{k\lef[\underline{T}\rig]},b}\lef(\beta_{i}{h\lef(c\lef(s,i\rig)\rig)}_{a}e\lef(1,a,\emptyset\rig)\rig)&\overset{\hphantom{\eqref{csinormegauss}}}{\overset{\eqref{conditionsurlesk}}{\overset{\eqref{constanteetapreuve}}{\geqslant}}}\gamma_{\varepsilon,\operatorname{Id}_{k\lef[\underline{T}\rig]},b}\lef(c\lef(s,i\rig)\rig)+\zeta\\
&\overset{\eqref{csinormegauss}}{\geqslant}\gamma_{\varepsilon,\operatorname{Id}_{k\lef[\underline{T}\rig]},b}\lef(\rho\lef(0\rig)\rig)-\frac{\eta}{2}\text{.}
\end{split}
\end{equation*}

Thus, applying \eqref{tgnormegauss} and \eqref{localipreuzero} yields:
\begin{equation*}
\forall a\in\mathcal{A},\ \gamma_{\varepsilon,\operatorname{Id}_{k\lef[\underline{T}\rig]},b}\lef(t_{F}\lef({w\lef(s+1\rig)}_{a}\rig)e\lef(1,a,\emptyset\rig)\rig)\geqslant\gamma_{\varepsilon,\operatorname{Id}_{k\lef[\underline{T}\rig]},b}\lef(\rho\lef(0\rig)\rig)-\frac{\eta}{2}\text{.}
\end{equation*}

So we have shown condition \eqref{localipreuzero} at rank $s+1$. We can now define $V$-adically convergent series:
\begin{gather*}
\begin{multlined}w\lef(s+1\rig)\coloneqq\sum_{a\in\mathcal{A}}t_{F}\lef({w\lef(s+1\rig)}_{a}\rig)e\lef(1,a,\emptyset\rig)\\\overset{\eqref{defwsprime}}{\overset{\eqref{defwsplusun}}{=}}w\lef(s\rig)+\sum_{a\in\mathcal{A}}t_{F}\lef(\sum_{i=1}^{n}\beta_{i}{h\lef(c\lef(s,i\rig)\rig)}_{a}\rig)e\lef(1,a,\emptyset\rig)\text{,}\end{multlined}\\
\rho\lef(s+1\rig)\coloneqq\sum_{i=1}^{n}v_{F}\lef(\beta_{i}\rig)c\lef(s,i\rig)-\sum_{a\in\mathcal{A}}v_{F}\lef(\sum_{i=1}^{n}\beta_{i}{h\lef(c\lef(s,i\rig)\rig)}_{a}\rig)e\lef(1,a,\emptyset\rig)\text{.}
\end{gather*}

Thus, by \eqref{tgnormegaussbis}, \eqref{conditionsurlesk}, \eqref{constanteetapreuve} and \eqref{csinormegauss}, and since $\gamma_{\varepsilon,\operatorname{Id}_{k\lef[\underline{T}\rig]},b}$ is a pseudovaluation:
\begin{equation*}
\gamma_{\varepsilon,\operatorname{Id}_{k\lef[\underline{T}\rig]},b}\lef(\rho\lef(s+1\rig)\rig)\geqslant\gamma_{\varepsilon,\operatorname{Id}_{k\lef[\underline{T}\rig]},b}\lef(\rho\lef(0\rig)\rig)\text{.}
\end{equation*}

So we have shown condition \eqref{localipreuun} at rank $s+1$.

But we also have:
\begin{equation*}
\begin{split}
W\lef(\varphi\rig)&\lef(w\lef(s+1\rig)+\rho\lef(s+1\rig)\rig)\\
&\overset{\hphantom{\eqref{deftfrhosi}}}{=}W\lef(\varphi\rig)\lef(w\lef(s\rig)+\sum_{i=1}^{n}\lef(v_{F}\lef(\beta_{i}\rig)c\lef(s,i\rig)+\sum_{a\in\mathcal{A}}\beta_{i}{h\lef(c\lef(s,i\rig)\rig)}_{a}e\lef(1,a,\emptyset\rig)\rig)\rig)\\
&\overset{\hphantom{\eqref{deftfrhosi}}}{\overset{\eqref{xandrax}}{=}}W\lef(\varphi\rig)\lef(w\lef(s\rig)+\sum_{i=1}^{n}\lef(v_{F}\lef(\beta_{i}\rig)c\lef(s,i\rig)+\beta_{i}c\lef(s,i\rig)\rig)\rig)\\
&\overset{\eqref{csigammasi}}{=}W\lef(\varphi\rig)\lef(w\lef(s\rig)+\sum_{i=1}^{n}t_{F}\lef(\beta_{i}\rig)\rho\lef(s,i\rig)\rig)\\
&\overset{\eqref{deftfrhosi}}{=}W\lef(\varphi\rig)\lef(w\lef(s\rig)+\rho\lef(s\rig)\rig)\\
&\overset{\eqref{localipreudeux}}{=}w\text{.}
\end{split}
\end{equation*}

In other terms, we have condition \eqref{localipreudeux} at rank $s+1$. By \eqref{vfisinvw}, \eqref{localipreusix}, \eqref{csidansv} and proposition \ref{divisionbypsum}, we can prove that:
\begin{align*}
w\lef(s+1\rig)-w\lef(s\rig)&\in V^{\max\lef\{s,u\rig\}}\lef(W^{\dagger}\lef(k\lef[\underline{X},Y\rig]\rig)\rig)\text{,}\\
\rho\lef(s+1\rig)&\in V^{\max\lef\{s+1,u\rig\}}\lef(W^{\dagger}\lef(k\lef[\underline{X},Y\rig]\rig)\rig)\text{.}
\end{align*}

This implies conditions \eqref{localipreuquatre}, \eqref{localipreucinq} and \eqref{localipreusix} at rank $s+1$. So we have demonstrated my induction on $s\in\mathbb{N}$ all the properties needed.

By $V$-adic convergence, we shall define $\hat{w}\coloneqq\lim_{s\rightarrow+\infty}w\lef(s\rig)$, and by construction of our sequence, we have $\gamma_{\varepsilon,\operatorname{Id}_{k\lef[\underline{T}\rig]},b}\lef(\hat{w}\rig)\geqslant\gamma_{\varepsilon,\operatorname{Id}_{k\lef[\underline{T}\rig]},b}\lef(\rho\lef(0\rig)\rig)-\frac{\eta}{2}$, $W\lef(\varphi\rig)\lef(\hat{w}\rig)=w$ and $\hat{w}\in V^{u}\lef(W^{\dagger}\lef(k\lef[\underline{T}\rig]\rig)\rig)$, which gives us the third formula of the statement by proposition \ref{divisionbypsum}. Moreover, we can get a writing of the form:
\begin{equation*}
\hat{w}=\sum_{a\in\mathcal{A}}t_{F}\lef(h_{a}\rig)e\lef(1,a,\emptyset\rig)\text{,}
\end{equation*}
where all $h_{a}\in{W\lef(k\rig)\lef[\underline{T}\rig]}^{\dagger}$, and by construction:
\begin{equation*}
\gamma_{\varepsilon,\operatorname{Id}_{k\lef[\underline{T}\rig]},b}\lef(t_{F}\lef(h_{a}\rig)\rig)+\gamma_{\varepsilon,\operatorname{Id}_{k\lef[\underline{T}\rig]},b}\lef(e\lef(1,a,\emptyset\rig)\rig)\geqslant\gamma_{\varepsilon,\operatorname{Id}_{k\lef[\underline{T}\rig]},b}\lef(\rho\lef(0\rig)\rig)-\frac{\eta}{2}\text{.}
\end{equation*}
\end{proof}

\section{The general structure theorem}

In this section, we shall prove the main theorem of the article about the local structure of the overconvergent de Rham--Witt complex. Unless otherwise stated, $k$ will denote a perfect commutative $\mathbb{F}_{p}$-algebra such that $W\lef(k\rig)$ is a Noetherian ring, and for a fixed $n\in\mathbb{N}^{*}$ we consider $k\lef[\underline{X}\rig]=k\lef[X_{1},\ldots,X_{n}\rig]$.

We have seen that in this context, $k$ has a rebar cage given in example \ref{perfectrebarcage}. We have previously defined subsets $G\lef(t,m\rig)$ and $H\lef(t,m\rig)$ of $W\Omega^{t}_{k\lef[\underline{X}\rig]/k}$ for $t\in\mathbb{N}$ and $m\in\mathbb{N}$. We shall consider $G\lef(t\rig)\coloneqq\bigsqcup_{m\in\mathbb{N}}G\lef(t,m\rig)$ and $H\lef(t\rig)\coloneqq\bigsqcup_{m\in\mathbb{N}}H\lef(t,m\rig)$.

The rebar cage in this context yields that $H\lef(t,m\rig)=\emptyset$ when $m\neq0$, and lightens the definition of $G\lef(t,m\rig)$. Thus, one can show that:
\begin{align*}
G\lef(t\rig)&=\lef\{e\lef(1,\frac{a+p^{u}\chi_{J}}{p^{u}},I\cup J\rig)\mid\begin{array}{c}u\in\mathbb{N},\ \lef(a,I\rig)\in\mathcal{P},\\J\subset\lef\llbracket1,n\rig\rrbracket\smallsetminus\operatorname{Supp}\lef(a\rig),\\\operatorname{v}_{p}\lef(a\rig)=0,\ \forall i\in\lef\llbracket1,n\rig\rrbracket,\ a_{i}<p^{u},\\I_{0}\neq\emptyset,\ \#I+\#J=t\end{array}\rig\}\text{,}\\
H\lef(t\rig)&=\lef\{e\lef(1,\chi_{I},I\rig)\mid\begin{array}{c}\lef(\chi_{I},I\rig)\in\mathcal{P},\\\#I=t\end{array}\rig\}\text{.}
\end{align*}

We also have the convention:
\begin{equation*}
G\lef(-1\rig)\coloneqq\emptyset\text{.}
\end{equation*}

As in \eqref{enotation}, for every $e\in G\lef(t\rig)\sqcup H\lef(t\rig)$ we shall denote by $\lef(a\lef(e\rig),I\lef(e\rig)\rig)\in\mathcal{P}$ the weight function and partition satisfying:
\begin{equation*}
e=e\lef(1,a\lef(e\rig),I\lef(e\rig)\rig)\text{.}
\end{equation*}

As a first step, we show the theorem in the case of a polynomial algebra. Recall that we have introduced $\zeta_{\varepsilon}$ in \eqref{zetaepsilon} and $t_{F}$ in definition \ref{lazardmorphism}.

\begin{prop}\label{structurepolynomialodrw}
Let $t\in\mathbb{N}$. Identify canonically $\Omega^{\mathrm{sep},t}_{{W\lef(k\rig)\lef[\underline{X}\rig]}^{\dagger}/W\lef(k\rig)}$ as a sub-$W\lef(k\rig)$-module of $W^{\dagger}\Omega^{t}_{k\lef[\underline{X}\rig]/k}$ using proposition \ref{lazardbijectivity}. For every $x\in W^{\dagger}\Omega^{t}_{k\lef[\underline{X}\rig]/k}$ there exists a unique map:
\begin{equation*}
s\colon\begin{array}{rl}H\lef(t\rig)\sqcup G\lef(t\rig)\sqcup G\lef(t-1\rig)\to&{W\lef(k\rig)\lef[\underline{X}\rig]}^{\dagger}\\e\mapsto&s\lef(e\rig)\end{array}
\end{equation*}
such that:
\begin{gather*}
x=\sum_{e\in H\lef(t\rig)}t_{F}\lef(s\lef(e\rig)e\rig)+\sum_{e\in G\lef(t\rig)}t_{F}\lef(s\lef(e\rig)\rig)e+d\lef(\sum_{e\in G\lef(t-1\rig)}t_{F}\lef(s\lef(e\rig)\rig)e\rig)\text{,}\\
\forall l\in\mathbb{N},\ p^{l}\mid x\iff\lef(\forall e\in H\lef(t\rig)\sqcup G\lef(t\rig)\sqcup G\lef(t-1\rig),\ p^{l}\mid s\lef(e\rig)\rig)\text{,}\\
\begin{multlined}\forall u\in\mathbb{N},\ x\in\operatorname{Fil}^{u}W^{\dagger}\Omega^{t}_{k\lef[\underline{X}\rig]/k}\\\iff\lef(\forall e\in H\lef(t\rig)\sqcup G\lef(t\rig)\sqcup G\lef(t-1\rig),\ p^{\max\lef\{u-u\lef(a\lef(e\rig)\rig),0\rig\}}\mid s\lef(e\rig)\rig)\text{.}\end{multlined}
\end{gather*}

Moreover, there exists $\delta>0$ independently on the choice of $x$ satisfying:
\begin{equation*}
\forall\varepsilon\in\lef]0,\delta\rig],\ \forall e\in H\lef(t\rig)\sqcup G\lef(t\rig)\sqcup G\lef(t-1\rig),\ \zeta_{\varepsilon}\lef(t_{F}\lef(s\lef(e\rig)\rig)\rig)\geqslant\zeta_{\varepsilon}\lef(x\rig)-\zeta_{\varepsilon}\lef(e\rig)\text{.}
\end{equation*}
\end{prop}

\begin{proof}
Denote by $\mathcal{F}$ the set of weight functions $a\colon\lef\llbracket1,n\rig\rrbracket\to\mathbb{N}\lef[\frac{1}{p}\rig]$. Consider the following map:
\begin{equation*}
\sigma\colon\begin{array}{rl}k\lef[\underline{X}\rig]\to&{W\lef(k\rig)\lef[\underline{X}\rig]}^{\dagger}\\\sum_{\substack{a\in\mathcal{F}\\u\lef(a\rig)=0}}c_{a}\underline{X}^{a}\mapsto&\sum_{\substack{a\in\mathcal{F}\\u\lef(a\rig)=0}}\lef[c_{a}\rig]\underline{X}^{a}\end{array}\text{,}
\end{equation*}
where all $c_{a}\in k$. Notice that for every $P\in k\lef[\underline{X}\rig]$ and every $\varepsilon>0$ we have $\zeta_{\varepsilon}\lef(\sigma\lef(P\rig)\rig)\geqslant-\varepsilon\deg\lef(P\rig)$, where we identify ${W\lef(k\rig)\lef[\underline{X}\rig]}^{\dagger}$ as a subset of $W^{\dagger}\Omega_{k\lef[\underline{X}\rig]/k}$ using proposition \ref{lazardbijectivity}.

We derive from proposition \ref{vfgamma} that there exists $\delta>0$ such that:
\begin{equation*}
\forall\varepsilon\in\lef]0,\delta\rig],\ \forall P\in k\lef[\underline{X}\rig],\ \zeta_{\varepsilon}\lef(t_{F}\lef(\sigma\lef(P\rig)\rig)\rig)\geqslant-\varepsilon\deg\lef(P\rig)\text{.}
\end{equation*}

Let $x\in W^{\dagger}\Omega^{t}_{k\lef[\underline{X}\rig]/k}$. We get a map $P\colon\begin{array}{rl}H\lef(t\rig)\to&k\lef[\underline{X}\rig]\\e\mapsto&P_{e}\end{array}$ from propositions \ref{truncatedderhamwittlocalstructure} and \ref{lazardinjectivity} such that:
\begin{gather*}
y\coloneqq x-\sum_{e\in H\lef(t\rig)}t_{F}\lef(\sigma\lef(P_{e}\rig)e\rig)\in pW^{\dagger}\Omega^{\mathrm{int},t}_{k\lef[\underline{X}\rig]/k}+W^{\dagger}\Omega^{\mathrm{frac},t}_{k\lef[\underline{X}\rig]/k}\text{,}\\
\forall\varepsilon\in\lef]0,\delta\rig],\ \forall e\in H\lef(t\rig),\ -\varepsilon\lef(\deg\lef(P_{e}\rig)+\lef\lvert a\lef(e\rig)\rig\rvert\rig)\geqslant\zeta_{\varepsilon}\lef(x\rig)\text{.}
\end{gather*}

Of course, $P$ is the zero map when $x\in\operatorname{Fil}^{u}W^{\dagger}\Omega^{t}_{k\lef[\underline{X}\rig]/k}$ for some $u\in\mathbb{N}^{*}$.

After applying proposition \ref{truncatedderhamwittlocalstructure} for a second time, we can extend $P$ to a map $P\colon\begin{array}{rl}H\lef(t\rig)\sqcup G\lef(t\rig)\sqcup G\lef(t-1\rig)\to&k\lef[\underline{X}\rig]\\e\mapsto&P_{e}\end{array}$ such that:
\begin{gather*}
y\equiv\sum_{e\in G\lef(t\rig)}\lef[P_{e}\rig]e+d\lef(\sum_{e\in G\lef(t-1\rig)}\lef[P_{e}\rig]e\rig)\pmod{p}\text{,}\\
\forall\varepsilon\in\lef]0,\delta\rig],\ \forall e\in G\lef(t\rig)\sqcup G\lef(t-1\rig),\ \lef(t+1\rig)u\lef(a\lef(e\rig)\rig)-\varepsilon\lef(\deg\lef(P_{e}\rig)+\lef\lvert a\lef(e\rig)\rig\rvert\rig)\geqslant\zeta_{\varepsilon}\lef(y\rig)\text{,}\\
\begin{multlined}\forall u\in\mathbb{N},\ x\in\operatorname{Fil}^{u}W^{\dagger}\Omega^{t}_{k\lef[\underline{X}\rig]/k}\\\iff\lef(\forall e\in H\lef(t\rig)\sqcup G\lef(t\rig)\sqcup G\lef(t-1\rig),\ u\lef(a\lef(e\rig)\rig)<u\implies P_{e}=0\rig)\text{.}\end{multlined}
\end{gather*}

The second line can be reworded as follows:
\begin{equation*}
\forall\varepsilon\in\lef]0,\delta\rig],\ \forall e\in G\lef(t\rig)\sqcup G\lef(t-1\rig),\ \zeta_{\varepsilon}\lef(t_{F}\lef(\sigma\lef(P_{e}\rig)\rig)\rig)\geqslant\zeta_{\varepsilon}\lef(x\rig)-\zeta_{\varepsilon}\lef(e\rig)\text{.}
\end{equation*}

We deduce from \eqref{vxfyvxy} and \eqref{vfisinvw} that:
\begin{equation*}
\forall e\in G\lef(t\rig)\sqcup G\lef(t-1\rig),\ t_{F}\lef(\sigma\lef(P_{e}\rig)\rig)e\equiv\lef[P_{e}\rig]e\pmod{p}\text{.}
\end{equation*}

Since $k$ is perfect, there is no $p$-torsion in the de Rham--Witt complex. Now, as $\zeta_{\varepsilon}$ is a pseudovaluation \cite[theorem 3.17]{pseudovaluationsonthederhamwittcomplex}, we can apply \eqref{multiplicationbypzeta} and thus get by induction on $l\in\mathbb{N}$ a process, described above, which gives us a sequence of maps $\lef(s_{l}\rig)_{l\in\mathbb{N}}\in\lef({{W\lef(k\rig)\lef[\underline{X}\rig]}^{\dagger}}^{H\lef(t\rig)\sqcup G\lef(t\rig)\sqcup G\lef(t-1\rig)}\rig)^{\mathbb{N}}$ such that for any $l\in\mathbb{N}$ and any $e\in H\lef(t\rig)\sqcup G\lef(t\rig)\sqcup G\lef(t-1\rig)$ we find:
\begin{gather*}
\zeta_{\varepsilon}\lef(t_{F}\lef(s_{l}\lef(e\rig)\rig)\rig)\geqslant\zeta_{\varepsilon}\lef(x\rig)-\zeta_{\varepsilon}\lef(e\rig)-2nl\text{,}\\
\begin{multlined}x\equiv\\\sum_{l'=0}^{l}p^{l'}\lef(\sum_{e\in H\lef(t\rig)}t_{F}\lef(s_{l'}\lef(e\rig)e\rig)+\sum_{e\in G\lef(t\rig)}t_{F}\lef(s_{l'}\lef(e\rig)\rig)e+d\lef(\sum_{e\in G\lef(t-1\rig)}t_{F}\lef(s_{l'}\lef(e\rig)\rig)e\rig)\rig)\\\pmod{p^{l+1}}\text{.}\end{multlined}
\end{gather*}

We thus get the existence of the map $s$ by $p$-adic overconvergence, and use theorem \ref{structuretheoremconvergent} for the unicity.
\end{proof}

\begin{rema}
In the statement of proposition \ref{structurepolynomialodrw}, we see that when $F=\widetilde{\operatorname{Frob}}$, then after applying propositions \ref{zpcombination} and \ref{lazardbijectivity}, we end up with the direct sum decompositions \eqref{decomposeoverconvergent} and \eqref{decomposefracoverconvergent}.
\end{rema}

We now begin the proof of the main theorem. We will need to introduce many notations which are gathered below. First, we fix:
\begin{equation*}
P\in k\lef[\underline{X}\rig]\text{.}
\end{equation*}

We shall now assume given a finite free étale morphism of commutative $k$-algebras, as in proposition \ref{characteriserelativelyperfect}:
\begin{equation*}
{k\lef[\underline{X}\rig]}_{\lef\langle P\rig\rangle}\to\overline{R}\text{.}
\end{equation*}

We will consider $\lef(r_{i}\rig)_{i\in\lef\llbracket1,m\rig\rrbracket}$ a basis of $\overline{R}$ as a ${k\lef[\underline{X}\rig]}_{\lef\langle P\rig\rangle}$-module, where $m\in\mathbb{N}$ is the rank of $\overline{R}$. In what follows, we shall write:
\begin{align*}
k\lef[\underline{T}\rig]&\coloneqq k\lef[X_{1},\ldots,X_{n},Y,Z_{1},\ldots,Z_{m}\rig]\text{,}\\
k\lef[\underline{P}\rig]&\coloneqq k\lef[X_{1},\ldots,X_{n},Y^{p},{Z_{1}}^{p},\ldots,{Z_{m}}^{p}\rig]\text{,}\\
W\lef(k\rig)\lef[\underline{P}\rig]&\coloneqq W\lef(k\rig)\lef[X_{1},\ldots,X_{n},Y^{p},{Z_{1}}^{p},\ldots,{Z_{m}}^{p}\rig]\text{.}
\end{align*}

We put:
\begin{equation*}
\varphi\colon k\lef[\underline{T}\rig]\to\overline{R}\text{,}
\end{equation*}
the surjective morphism of $k\lef[\underline{X}\rig]$-algebras such that $\varphi\lef(Y\rig)=P^{-1}$ and $\varphi\lef(Z_{i}\rig)=r_{i}$ for all $i\in\lef\llbracket1,m\rig\rrbracket$. By proposition \ref{characteriserelativelyperfect}, the morphism $\varphi|_{k\lef[\underline{P}\rig]}$ is also surjective. We will also denote this restriction as $\varphi\colon k\lef[\underline{P}\rig]\to\overline{R}$, which will lead to no confusion in the context.

For any $i\in\lef\llbracket0,m\rig\rrbracket$ we choose once and for all elements:
\begin{equation*}
\alpha_{i}\in\Omega^{\mathrm{sep},1}_{{W\lef(k\rig)\lef[\underline{P}\rig]}^{\dagger}/W\lef(k\rig)}\text{,}
\end{equation*}
and we want them to satisfy the following properties:
\begin{align*}
W\lef(\varphi\rig)\lef(\alpha_{0}\rig)&=W\lef(\varphi\rig)\lef(F\lef(d\lef(Y\rig)\rig)\rig)\text{,}\\
\forall i\in\lef\llbracket1,m\rig\rrbracket,\ W\lef(\varphi\rig)\lef(\alpha_{i}\rig)&=W\lef(\varphi\rig)\lef(F\lef(d\lef(Z_{i}\rig)\rig)\rig)\text{.}
\end{align*}

Now, we fix:
\begin{equation*}
\delta>0\text{.}
\end{equation*}

It shall be small enough to satisfy proposition \ref{vfgamma} with $\mu=\frac{1}{2}$. That is, we have:
\begin{equation}\label{zetavftfrob}
\forall\varepsilon\in\lef]0,\delta\rig],\ \forall x\in{W\lef(k\rig)\lef[\underline{P}\rig]}^{\dagger},\ \zeta_{\varepsilon}\lef(v_{F}\lef(x\rig)\rig)\geqslant\zeta_{\varepsilon}\lef(t_{\widetilde{\operatorname{Frob}}}\lef(x\rig)\rig)+\frac{1}{2}\text{.}
\end{equation}

Moreover, we also want $\delta$ to be small enough to suit proposition \ref{structurepolynomialodrw}, and so that:
\begin{equation*}
\forall\varepsilon\in\lef]0,\delta\rig],\ \forall i\in\lef\llbracket0,m\rig\rrbracket,\ \zeta_{\varepsilon}\lef(t_{\widetilde{\operatorname{Frob}}}\lef(\alpha_{i}\rig)\rig)\geqslant\frac{1}{2}-2\lef(n+1+m\rig)\text{.}
\end{equation*}

For $\varepsilon\in\lef]0,\delta\rig]$, as $\zeta_{\varepsilon}$ is a pseudovaluation \cite[theorem 3.17]{pseudovaluationsonthederhamwittcomplex}, we get after applying \eqref{dfpfd}, \eqref{multiplicationbypzeta} and \eqref{dzetaepsilon} on the complex $W\Omega_{k\lef[\underline{P}\rig]/k}$ that:
\begin{align*}
\zeta_{\varepsilon}\lef(t_{\widetilde{\operatorname{Frob}}}\lef(p\alpha_{0}\rig)\rig)=\frac{1}{2}-\varepsilon&\geqslant\zeta_{\varepsilon}\lef(d\lef(\lef[Y^{p}\rig]\rig)\rig)=-\varepsilon\text{,}\\
\forall i\in\lef\llbracket1,m\rig\rrbracket,\ \zeta_{\varepsilon}\lef(t_{\widetilde{\operatorname{Frob}}}\lef(p\alpha_{i}\rig)\rig)=\frac{1}{2}-\varepsilon&\geqslant\zeta_{\varepsilon}\lef(d\lef(\lef[{Z_{i}}^{p}\rig]\rig)\rig)=-\varepsilon\text{,}\\
W\lef(\varphi\rig)\lef(p\alpha_{0}\rig)&=W\lef(\varphi\rig)\lef(d\lef(Y^{p}\rig)\rig)\text{,}\\
\forall i\in\lef\llbracket1,m\rig\rrbracket,\ W\lef(\varphi\rig)\lef(p\alpha_{i}\rig)&=W\lef(\varphi\rig)\lef(d\lef({Z_{i}}^{p}\rig)\rig)\text{.}
\end{align*}

By convention, except when otherwise indicated, all weight functions shall be defined on $\lef\llbracket1,n\rig\rrbracket$. That is, for every weight function $a$ and any partition $I$ of $a$, we shall have $e\lef(1,a,I\rig)\in W\Omega^{\dagger}_{k\lef[\underline{X}\rig]/k}$. In particular, we also have $H\lef(t\rig)\subset W^{\dagger}\Omega^{t}_{k\lef[\underline{X}\rig]/k}$ and $G\lef(t\rig)\subset W^{\dagger}\Omega^{t}_{k\lef[\underline{X}\rig]/k}$ for $t\in\mathbb{N}$.

\begin{deff}
We shall say that a weight function $a\colon\lef\llbracket1,n+1+m\rig\rrbracket\to\mathbb{N}\lef[\frac{1}{p}\rig]$ is \textbf{extended}. The supports and partitions on these weight functions are defined as usual.

These weight functions are related to the complex $W\Omega_{k\lef[\underline{P}\rig]/k}$. We will denote by $\mathcal{P}^{\underline{P}}$ the set of the couples $\lef(a,I\rig)$, where $a$ is an extended weight function, and $I$ is a partition of $a$.
\end{deff}

For any weight function $a$, extended or not, we will write:
\begin{equation*}
\underline{P}^{a}\coloneqq\prod_{j=1}^{n}{X_{j}}^{a_{j}}\times Y^{pa_{n+1}}\times\prod_{j=1}^{m}{Z_{j}}^{pa_{n+1+j}}\text{.}
\end{equation*}

So that, as usual, we can put:
\begin{align*}
g\lef(a\rig)&\coloneqq F^{u\lef(a\rig)+\operatorname{v}_{p}\lef(a\rig)}\lef(d\lef(V^{u\lef(a\rig)}\lef(\lef[\underline{P}^{p^{-\operatorname{v}_{p}\lef(a\rig)}a}\rig]\rig)\rig)\rig)\text{,}\\
u\lef(a\rig)&\coloneqq\min\lef\{-\operatorname{v}_{p}\lef(a\rig),0\rig\}\text{.}
\end{align*}

For every $t\in\mathbb{N}$, we shall also put:
\begin{align*}
G^{\underline{P}}\lef(t\rig)&\coloneqq\lef\{e\lef(1,\frac{a+p^{u}\chi_{J}}{p^{u}},I\cup J\rig)\mid\begin{array}{c}u\in\mathbb{N},\ \lef(a,I\rig)\in\mathcal{P}^{\underline{P}},\\J\subset\lef\llbracket1,n+1+m\rig\rrbracket\smallsetminus\operatorname{Supp}\lef(a\rig),\\\operatorname{v}_{p}\lef(a\rig)=0,\\\forall i\in\lef\llbracket1,n+1+m\rig\rrbracket,\ a_{i}<p^{u},\\I_{0}\neq\emptyset,\ \#I+\#J=t\end{array}\rig\}\text{,}\\
H^{\underline{P}}\lef(t\rig)&\coloneqq\lef\{e\lef(1,\chi_{I},I\rig)\mid\begin{array}{c}\lef(\chi_{I},I\rig)\in\mathcal{P}^{\underline{P}},\\\#I=t\end{array}\rig\}\text{.}
\end{align*}

We keep the convention:
\begin{equation*}
G^{\underline{P}}\lef(-1\rig)\coloneqq\emptyset\text{.}
\end{equation*}

For every $t\in\mathbb{N}$ we define the following ${W\lef(k\rig)\lef[\underline{P}\rig]}^{\dagger}$-modules:
\begin{align*}
W^{\dagger}\Omega^{\underline{X}-\mathrm{int},t}_{k\lef[\underline{P}\rig]/k}&\coloneqq\lef\{\sum_{e\in H\lef(t\rig)}t_{F}\lef(s_{e}e\rig)\mid\begin{array}{l}\exists C\in\mathbb{R},\ \exists\varepsilon>0,\ \forall e\in H\lef(t\rig),\\\lef(s_{e}\in W^{\dagger}\Omega^{\mathrm{int},0}_{k\lef[\underline{P}\rig]/k}\rig)\wedge\lef(\zeta_{\varepsilon}\lef(s_{e}\rig)+\zeta_{\varepsilon}\lef(e\rig)\geqslant C\rig)\end{array}\rig\}\text{,}\\
W^{\dagger}\Omega^{\underline{X}-\mathrm{frp},t}_{k\lef[\underline{P}\rig]/k}&\coloneqq\lef\{\sum_{e\in G\lef(t\rig)}t_{F}\lef(s_{e}\rig)e\mid\begin{array}{l}\exists C\in\mathbb{R},\ \exists\varepsilon>0,\ \forall e\in G\lef(t\rig),\\\lef(s_{e}\in W^{\dagger}\Omega^{\mathrm{int},0}_{k\lef[\underline{P}\rig]/k}\rig)\wedge\lef(\zeta_{\varepsilon}\lef(s_{e}\rig)+\zeta_{\varepsilon}\lef(e\rig)\geqslant C\rig)\end{array}\rig\}\text{.}
\end{align*}

And we let:
\begin{equation*}
W^{\dagger}\Omega^{\underline{X},t}_{k\lef[\underline{P}\rig]/k}\coloneqq W^{\dagger}\Omega^{\underline{X}-\mathrm{int},t}_{k\lef[\underline{P}\rig]/k}+W^{\dagger}\Omega^{\underline{X}-\mathrm{frp},t}_{k\lef[\underline{P}\rig]/k}+d\lef(W^{\dagger}\Omega^{\underline{X}-\mathrm{int},t-1}_{k\lef[\underline{P}\rig]/k}\rig)\text{.}
\end{equation*}

Our goal will be to prove that this complex surjects onto $W^{\dagger}\Omega^{t}_{\overline{R}/k}$. We begin with the integral part.

\begin{prop}\label{intyieldsxint}
Let $t\in\mathbb{N}$. For every $x\in W^{\dagger}\Omega^{\mathrm{int},t}_{k\lef[\underline{P}\rig]/k}$, there is $z\in W^{\dagger}\Omega^{\underline{X}-\mathrm{int},t}_{k\lef[\underline{P}\rig]/k}$ such that:
\begin{gather*}
W\lef(\varphi\rig)\lef(t_{F}\lef(x\rig)\rig)=W\lef(\varphi\rig)\lef(z\rig)\text{,}\\
\forall\varepsilon\in\lef]0,\delta\rig],\ \zeta_{\varepsilon}\lef(z\rig)\geqslant\zeta_{\varepsilon}\lef(t_{F}\lef(x\rig)\rig)\text{.}
\end{gather*}
\end{prop}

\begin{proof}
Using proposition \ref{structurepolynomialodrw}, we know that $x$ is a linear combination of products of the form $s\prod_{i=1}^{t}d\lef(x_{i}\rig)$, with all $x_{i}\in\lef\{X_{1},\ldots,X_{n},Y^{p},{Z_{1}}^{p},\ldots,{Z_{m}}^{p}\rig\}$ and $s\in W^{\dagger}\Omega^{\mathrm{int},0}_{k\lef[\underline{P}\rig]/k}$ with:
\begin{equation*}
\zeta_{\varepsilon}\lef(s\rig)\geqslant\zeta_{\varepsilon}\lef(x\rig)-\varepsilon t\text{.}
\end{equation*}

Assume there is $i\in\lef\llbracket1,t\rig\rrbracket$ such that $x_{i}\in\lef\{Y^{p},{Z_{1}}^{p},\ldots,{Z_{m}}^{p}\rig\}$. Then, we can replace $d\lef(x_{i}\rig)$ by the adequate $p\alpha_{j}$ where $j\in\lef\llbracket0,m\rig\rrbracket$, as defined in the beginning of this section. The new product $y$ will then satisfy $\zeta_{\varepsilon}\lef(y\rig)\geqslant\zeta_{\varepsilon}\lef(x\rig)$ since $\zeta_{\varepsilon}$ is a pseudovaluation \cite[theorem 3.17]{pseudovaluationsonthederhamwittcomplex}.

Hence, we found an element of the required form modulo $p$. We can thus repeat the process, so that by overconvergence we find an element $z'\in W^{\dagger}\Omega^{\mathrm{int},t}_{k\lef[\underline{P}\rig]/k}$ such that:
\begin{gather*}
t_{F}\lef(z'\rig)\in W^{\dagger}\Omega^{\underline{X}-\mathrm{int},t}_{k\lef[\underline{P}\rig]/k}\text{,}\\
W\lef(\varphi\rig)\lef(t_{F}\lef(x\rig)\rig)=W\lef(\varphi\rig)\lef(t_{F}\lef(z'\rig)\rig)\text{,}\\
\zeta_{\varepsilon}\lef(t_{F}\lef(z'\rig)\rig)\geqslant\zeta_{\varepsilon}\lef(x\rig)\text{.}
\end{gather*}

Applying proposition \ref{structurepolynomialodrw} to $x$ and \eqref{zetavftfrob}, we find with the pseudovaluation property that $\zeta_{\varepsilon}\lef(v_{F}\lef(x\rig)\rig)>\zeta_{\varepsilon}\lef(x\rig)$, so that $\zeta_{\varepsilon}\lef(x\rig)=\zeta_{\varepsilon}\lef(t_{F}\lef(x\rig)\rig)$ and we are done.
\end{proof}

The proof of the theorem works by demonstrating it modulo $p$ without losing too much overconvergence. The following technical lemma will allow us to do so.

\begin{lemm}\label{technicalgptrewrite}
There exists $\delta'\in\lef]0,\delta\rig]$ such that for all $e\in G^{\underline{P}}\lef(t\rig)$ we can find $e',x\in W^{\dagger}\Omega_{k\lef[\underline{P}\rig]/k}^{t}$ satisfying:
\begin{align*}
W\lef(\varphi\rig)\lef(e'+px\rig)&=W\lef(\varphi\rig)\lef(e\rig)\text{,}\\
\forall\varepsilon\in\lef]0,\delta'\rig],\ \zeta_{\varepsilon}\lef(e'\rig)&\geqslant\zeta_{\varepsilon}\lef(e\rig)-\frac{1}{2}\text{,}\\
\forall\varepsilon\in\lef]0,\delta\rig],\ \zeta_{\varepsilon}\lef(px\rig)&\geqslant\zeta_{\varepsilon}\lef(e\rig)\text{.}
\end{align*}

Moreover, $e'$ can be written as an overconvergent series of products of the form $wz$, with $w\in{W\lef(k\rig)\lef[\underline{P}\rig]}^{\dagger}$ and $z\in W^{\dagger}\Omega_{k\lef[\underline{X}\rig]/k}^{t}$, and this series is congruent modulo $p$ to an element in $W^{\dagger}\Omega_{k\lef[\underline{X}\rig]/k}^{\underline{X}-\mathrm{frp},t}$.
\end{lemm}

\begin{proof}
Let us write for simplicity $a\coloneqq a\lef(e\rig)$ and $I\coloneqq I\lef(e\rig)$. We must have $\#I=t$. Then, we have by definition:
\begin{multline*}
e=V^{u\lef(a\rig)}\lef(\lef[\underline{P}^{p^{u\lef(a\rig)}a|_{I_{0}}}\rig]\rig)\\\times\prod_{l=1}^{t}F^{u\lef(a|_{I_{l}}\rig)+\operatorname{v}_{p}\lef(a|_{I_{l}}\rig)}\lef(d\lef(V^{u\lef(a|_{I_{l}}\rig)}\lef(\lef[\underline{P}^{p^{-\operatorname{v}_{p}\lef(a|_{I_{l}}\rig)}a|_{I_{l}}}\rig]\rig)\rig)\rig)\text{.}
\end{multline*}

Applying \eqref{vxfyvxy} and \eqref{fdvisd} we thus get:
\begin{equation}\label{formulafore}
e=V^{u\lef(a\rig)}\lef(\lef[\underline{P}^{p^{u\lef(a\rig)}a|_{I_{0}}}\rig]\times\prod_{l=1}^{t}F^{u\lef(a\rig)+\operatorname{v}_{p}\lef(a|_{I_{l}}\rig)}\lef(d\lef(\lef[\underline{P}^{p^{-\operatorname{v}_{p}\lef(a|_{I_{l}}\rig)}a|_{I_{l}}}\rig]\rig)\rig)\rig)\text{.}
\end{equation}

Now, for every $l\in\lef\llbracket1,t\rig\rrbracket$ one can use \eqref{fdttdt} and \eqref{dproducts} to rewrite the factor $F^{u\lef(a\rig)+\operatorname{v}_{p}\lef(a|_{I_{l}}\rig)}\lef(d\lef(\lef[\underline{P}^{p^{-\operatorname{v}_{p}\lef(a|_{I_{l}}\rig)}a|_{I_{l}}}\rig]\rig)\rig)$ as a linear combination of elements of the form $\lef[\underline{P}^{\kappa}\rig]d\lef(\lef[\underline{P}^{\beta}\rig]\rig)$, where $\kappa$ and $\beta$ are weight functions taking values in $\mathbb{N}$, $\lef\lvert\beta\rig\rvert=1$ and $\lef\lvert\kappa+\beta\rig\rvert=p^{u\lef(a\rig)}\lef\lvert a|_{I_{l}}\rig\rvert$.

When $d\lef(\underline{P}^{\beta}\rig)=d\lef(X_{i}\rig)$ for some $i\in\lef\llbracket1,n\rig\rrbracket$, leave it as it is. Otherwise, replace it the adequate $t_{\widetilde{\operatorname{Frob}}}\lef(p\alpha_{i'}\rig)$ that we have introduced earlier in this section, where $i'\in\lef\llbracket0,m\rig\rrbracket$. Then, multiply all these combinations together with $\lef[\underline{P}^{p^{u\lef(a\rig)}a|_{I_{0}}}\rig]$, and apply $V^{u\lef(a\rig)}$ as in \eqref{formulafore}. This gives us an element of the form $V^{u\lef(a\rig)}\lef(y\rig)$ where $y\in W^{\dagger}\Omega_{k\lef[\underline{P}\rig]/k}$. By construction:
\begin{equation*}
W\lef(\varphi\rig)\lef(V^{u\lef(a\rig)}\lef(y\rig)\rig)=W\lef(\varphi\rig)\lef(e\rig)\text{.}
\end{equation*}

Fix $\varepsilon\in\lef]0,\delta\rig]$. Recall that $\zeta_{\varepsilon}$, which we introduced in \eqref{zetaepsilon}, is a pseudovaluation \cite[theorem 3.17]{pseudovaluationsonthederhamwittcomplex}, so we have:
\begin{equation*}
\zeta_{\varepsilon}\lef(y\rig)\geqslant-\varepsilon p^{u\lef(a\rig)}\lef\lvert a\rig\rvert\text{.}
\end{equation*}

Using proposition \ref{vactionone} and the definition of $\zeta_{\varepsilon}$, we deduce that:
\begin{equation*}
\zeta_{\varepsilon}\lef(V^{u\lef(a\rig)}\lef(y\rig)\rig)\geqslant\lef(t+1\rig)u\lef(a\rig)-\varepsilon\lef\lvert a\rig\rvert=\zeta_{\varepsilon}\lef(e\rig)\text{.}
\end{equation*}

Notice that if, in the above process, we have done a replacement with some $t_{\widetilde{\operatorname{Frob}}}\lef(p\alpha_{i'}\rig)$, then $V^{u\lef(a\rig)}\lef(y\rig)$ is divisible by $p$. The linear combination of these $V^{u\lef(a\rig)}\lef(y\rig)$ gives us the $x$ of the statement.

So let us now focus on the remaining terms of the linear combination; that is, the other $V^{u\lef(a\rig)}\lef(y\rig)$ we got without replacement. In that setting, we can use \cite[1.16]{derhamwittcohomologyforaproperandsmoothmorphism} to write:
\begin{equation*}
V^{u\lef(a\rig)}\lef(y\rig)=V^{u\lef(a\rig)}\lef(\lef[\underline{P}^{\gamma}\rig]\rig)\times\prod_{i=1}^{t}d\lef(V^{u\lef(a\rig)}\lef(\lef[\omega_{i}\rig]\rig)\rig)\text{,}
\end{equation*}
where $\gamma$ is a weight function with values in $\mathbb{N}$ such that $\lef\lvert\gamma\rig\rvert=p^{u\lef(a\rig)}\lef\lvert a\rig\rvert-t$, and where each $\omega_{i}=X_{i'}$ for some $i'\in\lef\llbracket1,n\rig\rrbracket$.

We know by proposition \ref{overconvergentwittvectorsstructure} that there exists $\delta'>0$ and $b\in{\lef]0,+\infty\rig[}^{n+1+m}$ independently on any of the above choices, as well as a map:
\begin{equation*}
s\colon\begin{array}{rl}H\lef(0\rig)\sqcup G\lef(0\rig)\to&{W\lef(k\rig)\lef[\underline{P}\rig]}^{\dagger}\\f\mapsto&s\lef(f\rig)\end{array}
\end{equation*}
such that:
\begin{align*}
W\lef(\varphi\rig)\lef(\sum_{f\in H\lef(0\rig)}s\lef(f\rig)f+\sum_{f\in G\lef(0\rig)}s\lef(f\rig)f\rig)&=W\lef(\varphi\rig)\lef(V^{u\lef(a\rig)}\lef(\lef[\underline{P}^{\gamma}\rig]\rig)\rig)\text{,}\\
\forall\varepsilon\in\lef]0,\delta'\rig],\ \forall f\in H\lef(0\rig)\sqcup G\lef(0\rig),\ \gamma_{\varepsilon,\operatorname{Id}_{k\lef[\underline{P}\rig]},b}\lef(s\lef(f\rig)f\rig)&\geqslant\gamma_{\varepsilon,\operatorname{Id}_{k\lef[\underline{P}\rig]},b}\lef(V^{u\lef(a\rig)}\lef(\lef[\underline{P}^{\gamma}\rig]\rig)\rig)-\frac{1}{4}\text{,}\\
\forall f\in H\lef(0\rig)\sqcup G\lef(0\rig),\ p^{\max\lef\{u\lef(a\rig)-u\lef(a\lef(f\rig)\rig),0\rig\}}&\mid s\lef(f\rig)\text{.}
\end{align*}

By \eqref{equivalencesofgamma}, we know that we can find $m,M\in\lef]0,+\infty\rig[$ such that $\displaystyle\frac{M}{m}\geqslant1$ and that for every $\varepsilon\in\lef]0,m\delta'\rig]$ and every $f\in H\lef(0\rig)\sqcup G\lef(0\rig)$ we have:
\begin{equation*}
\gamma_{\varepsilon,\operatorname{Id}_{k\lef[\underline{P}\rig]},\lef(1,\ldots,1\rig)}\lef(s\lef(f\rig)f\rig)\geqslant\gamma_{\frac{M}{m}\varepsilon,\operatorname{Id}_{k\lef[\underline{P}\rig]},\lef(1,\ldots,1\rig)}\lef(V^{u\lef(a\rig)}\lef(\lef[\underline{P}^{\gamma}\rig]\rig)\rig)-\frac{1}{4}\text{.}
\end{equation*}

So by virtue of proposition \ref{gammaandzetawitt}, we get that:
\begin{equation*}
\zeta_{\varepsilon}\lef(s\lef(f\rig)f\rig)\geqslant\gamma_{\frac{M}{m}\varepsilon,\operatorname{Id}_{k\lef[\underline{P}\rig]},\lef(1,\ldots,1\rig)}\lef(V^{u\lef(a\rig)}\lef(\lef[\underline{P}^{\gamma}\rig]\rig)\rig)-\frac{1}{4}\geqslant u\lef(a\rig)-\frac{M}{m}\varepsilon\lef(\lef\lvert a\rig\rvert-p^{-u\lef(a\rig)}t\rig)-\frac{1}{4}\text{.}
\end{equation*}

Notice that, since $e\in G^{\underline{P}}\lef(t\rig)$, we have $\lef\lvert a\rig\rvert\leqslant n+1+m$. So after reducing the constant $\delta'$ if needed, we may assume that: 
\begin{equation*}
-\frac{M}{m}\varepsilon\lef(\lef\lvert a\rig\rvert-p^{-u\lef(a\rig)}t\rig)\geqslant-\varepsilon\lef(\lef\lvert a\rig\rvert-p^{-u\lef(a\rig)}t\rig)-\frac{1}{4}\text{.}
\end{equation*}

Using the pseudovaluation property we are now in a position to see that for every $f\in H\lef(0\rig)\sqcup G\lef(0\rig)$ we have:
\begin{equation*}
\zeta_{\varepsilon}\lef(s\lef(f\rig)f\times\prod_{i=1}^{t}d\lef(V^{u\lef(a\rig)}\lef(\lef[\omega_{i}\rig]\rig)\rig)\rig)\geqslant\lef(t+1\rig)u\lef(a\rig)-\varepsilon\lef\lvert a\rig\rvert-\frac{1}{2}=\zeta_{\varepsilon}\lef(e\rig)-\frac{1}{2}\text{.}
\end{equation*}

We can conclude the proof of this lemma by letting $e'$ be the linear combination described above of the elements:
\begin{equation*}
\lef(\sum_{f\in H\lef(0\rig)}s\lef(f\rig)f+\sum_{f\in G\lef(0\rig)}s\lef(f\rig)f\rig)\times\prod_{i=1}^{t}d\lef(V^{u\lef(a\rig)}\lef(\lef[\omega_{i}\rig]\rig)\rig)\text{.}
\end{equation*}

We conclude using \cite[proposition 2.8]{pseudovaluationsonthederhamwittcomplex} as $k$ is perfect.
\end{proof}

\begin{thrm}\label{maintheorem}
Let $t\in\mathbb{N}$ and $x\in W^{\dagger}\Omega^{t}_{\overline{R}/k}$. Then, there exists a unique map:
\begin{equation*}
s\colon\begin{array}{rl}H\lef(t\rig)\sqcup G\lef(t\rig)\sqcup G\lef(t-1\rig)\to&R^{\dagger}\\e\mapsto&s_{e}\end{array}
\end{equation*}
such that:
\begin{equation*}
x=\sum_{e\in H\lef(t\rig)}t_{F}\lef(s_{e}e\rig)+\sum_{e\in G\lef(t\rig)}t_{F}\lef(s_{e}\rig)e+d\lef(\sum_{e\in G\lef(t-1\rig)}t_{F}\lef(s_{e}\rig)e\rig)\text{.}
\end{equation*}

In other terms, the map $W^{\dagger}\Omega^{\underline{X},t}_{k\lef[\underline{P}\rig]/k}\to W^{\dagger}\Omega^{t}_{\overline{R}/k}$ induced by $\varphi$ is surjective.
\end{thrm}

\begin{proof}
Any element in $W^{\dagger}\Omega_{\overline{R}/k}^{t}$ is the image of some $x\in W^{\dagger}\Omega_{k\lef[\underline{P}\rig]/k}^{t}$. Let $\varepsilon>0$ be such that $\zeta_{\varepsilon}\lef(x\rig)\neq-\infty$.

By proposition \ref{structurepolynomialodrw}, and after reducing $\varepsilon$ if needed, we know that there exists a map:
\begin{equation*}
s\colon\begin{array}{rl}H^{\underline{P}}\lef(t\rig)\sqcup G^{\underline{P}}\lef(t\rig)\sqcup G^{\underline{P}}\lef(t-1\rig)\to&{W\lef(k\rig)\lef[\underline{P}\rig]}^{\dagger}\\e\mapsto&s\lef(e\rig)\end{array}
\end{equation*}
such that:
\begin{gather*}
x=\sum_{e\in H^{\underline{P}}\lef(t\rig)}t_{F}\lef(s\lef(e\rig)e\rig)+\sum_{e\in G^{\underline{P}}\lef(t\rig)}t_{F}\lef(s\lef(e\rig)\rig)e+d\lef(\sum_{e\in G^{\underline{P}}\lef(t-1\rig)}t_{F}\lef(s\lef(e\rig)\rig)e\rig)\text{,}\\
\forall e\in H^{\underline{P}}\lef(t\rig)\sqcup G^{\underline{P}}\lef(t\rig)\sqcup G^{\underline{P}}\lef(t-1\rig),\ \zeta_{\varepsilon}\lef(t_{F}\lef(s_{e}\rig)\rig)\geqslant\zeta_{\varepsilon}\lef(x\rig)-\zeta_{\varepsilon}\lef(e\rig)\text{.}
\end{gather*}

After reducing $\varepsilon$ once more if needed, we can apply proposition \ref{intyieldsxint} to find $w\in W^{\dagger}\Omega^{\underline{X}-\mathrm{int},t}_{k\lef[\underline{P}\rig]/k}$ such that:
\begin{align*}
W\lef(\varphi\rig)\lef(w\rig)&=W\lef(\varphi\rig)\lef(\sum_{e\in H^{\underline{P}}\lef(t\rig)}t_{F}\lef(s\lef(e\rig)e\rig)\rig)\text{,}\\
\zeta_{\varepsilon}\lef(w\rig)&\geqslant\zeta_{\varepsilon}\lef(x\rig)\text{.}
\end{align*}

Now fix $e\in G^{\underline{P}}\lef(t\rig)$. If we also have $e\in G\lef(t\rig)$ then in the above series the product $t_{F}\lef(s\lef(e\rig)\rig)e\in W^{\dagger}\Omega^{\underline{X}-\mathrm{frp},t}_{k\lef[\underline{P}\rig]/k}$. So let us focus on the case where $e\in G^{\underline{P}}\lef(t\rig)\smallsetminus G\lef(t\rig)$.

We have $p\mid v_{F}\lef(s\lef(e\rig)\rig)e$ because of \eqref{vfisinvw}. So after reducing $\varepsilon$ one last time if needed, we can apply lemma \ref{technicalgptrewrite} to $e$ and \eqref{zetavftfrob} to $s\lef(e\rig)$. Combined with proposition \ref{structurepolynomialodrw}, they tell us that there exists $x'\in W^{\dagger}\Omega_{k\lef[\underline{P}\rig]/k}^{t}$ and a map:
\begin{equation*}
s'\colon\begin{array}{rl}H^{\underline{P}}\lef(t\rig)\sqcup G^{\underline{P}}\lef(t\rig)\sqcup G^{\underline{P}}\lef(t-1\rig)\to&{W\lef(k\rig)\lef[\underline{P}\rig]}^{\dagger}\\f\mapsto&s'\lef(f\rig)\end{array}
\end{equation*}
such that:
\begin{multline*}
W\lef(\varphi\rig)\lef(t_{F}\lef(s\lef(e\rig)\rig)e\rig)\\=W\lef(\varphi\rig)\lef(\sum_{f\in H^{\underline{P}}\lef(t\rig)}s'\lef(f\rig)f+\sum_{f\in G^{\underline{P}}\lef(t\rig)}s'\lef(f\rig)f+d\lef(\sum_{f\in G^{\underline{P}}\lef(t-1\rig)}s'\lef(f\rig)f\rig)+px'\rig)\text{.}
\end{multline*}

Moreover:
\begin{align*}
\zeta_{\varepsilon}\lef(px'\rig)&\geqslant\zeta_{\varepsilon}\lef(x\rig)\text{,}\\
\forall f\in H^{\underline{P}}\lef(t\rig)\sqcup G^{\underline{P}}\lef(t\rig)\sqcup G^{\underline{P}}\lef(t-1\rig),\ \zeta_{\varepsilon}\lef(s'\lef(f\rig)\rig)+\zeta_{\varepsilon}\lef(f\rig)&\geqslant\zeta_{\varepsilon}\lef(x\rig)-\frac{1}{2}\text{,}\\
\forall f\in H^{\underline{P}}\lef(t\rig)\sqcup G^{\underline{P}}\lef(t\rig)\sqcup G^{\underline{P}}\lef(t-1\rig),\ u\lef(f|_{\lef\llbracket n+1,n+1+m\rig\rrbracket}\rig)>0&\implies s'\lef(f\rig)=0\text{,}\\
\forall f\in H^{\underline{P}}\lef(t\rig)\sqcup G^{\underline{P}}\lef(t\rig)\sqcup G^{\underline{P}}\lef(t-1\rig),\ p^{\max\lef\{u\lef(a\lef(e\rig)\rig)-u\lef(a\lef(f\rig)\rig),0\rig\}}&\mid s'\lef(f\rig)f\text{.}
\end{align*}

Here, we have used the fact that $\zeta_{\varepsilon}$ is a pseudovaluation \cite[theorem 3.17]{pseudovaluationsonthederhamwittcomplex}.

When an element $f$ in any of the above series is not in $H\lef(t\rig)\sqcup G\lef(t\rig)\sqcup G\lef(t-1\rig)$, this thus implies by our construction that $f$ is divisible by $d\lef(z\rig)$, where either $z=Y^{p}$ or $z={Z_{i}}^{p}$ for some $i\in\lef\llbracket1,m\rig\rrbracket$. We can thus replace this $d\lef(z\rig)$ by the adequate $t_{\widetilde{\operatorname{Frob}}}\lef(p\alpha_{i'}\rig)$ that we have introduced earlier in this section, for some $i'\in\lef\llbracket0,m\rig\rrbracket$, to obtain an element $pf'\in W^{\dagger}\Omega_{k\lef[\underline{P}\rig]/k}^{t}$. We have $\zeta_{\varepsilon}\lef(s'\lef(f\rig)pf'\rig)\geqslant\zeta_{\varepsilon}\lef(e\rig)$ because of the replacement. This means that we can in fact assume that $s'$ takes values in $H\lef(t\rig)\sqcup G\lef(t\rig)\sqcup G\lef(t-1\rig)$.

Furthermore, observe that:
\begin{multline*}
\zeta_{\varepsilon}\lef(\sum_{f\in H\lef(t\rig)}v_{F}\lef(s'\lef(f\rig)f\rig)+\sum_{f\in G\lef(t\rig)}v_{F}\lef(s'\lef(f\rig)\rig)f+d\lef(\sum_{f\in G\lef(t-1\rig)}v_{F}\lef(s'\lef(f\rig)\rig)f\rig)\rig)\\\geqslant\zeta_{\varepsilon}\lef(x\rig)\text{.}
\end{multline*}

This is again a consequence of \eqref{zetavftfrob} and of the pseudovaluation property. For the series indexed on $f\in H\lef(t\rig)$, use \eqref{productformulavf} to write $v_{F}\lef(s'\lef(f\rig)f\rig)$ as a linear combination of products between $t_{F}\lef(s'\lef(f\rig)\rig)$, $v_{F}\lef(s'\lef(f\rig)\rig)$, $d\lef(t_{F}\lef(X_{i}\rig)\rig)$ or $d\lef(v_{F}\lef(X_{i}\rig)\rig)$ where $i\in\lef\llbracket1,n\rig\rrbracket$, and conclude with \eqref{dzetaepsilon}.

Moreover, notice that these series are all divisible by $p$. So this means that there is $x''\in W^{\dagger}\Omega_{k\lef[\underline{P}\rig]/k}^{t}$ with $\zeta_{\varepsilon}\lef(px''\rig)\geqslant\zeta_{\varepsilon}\lef(x\rig)$ satisfying:
\begin{multline*}
W\lef(\varphi\rig)\lef(t_{F}\lef(s\lef(e\rig)\rig)e-px''\rig)=\\W\lef(\varphi\rig)\lef(\sum_{f\in H\lef(t\rig)}t_{F}\lef(s'\lef(f\rig)f\rig)+\sum_{f\in G\lef(t\rig)}t_{F}\lef(s'\lef(f\rig)\rig)f+d\lef(\sum_{f\in G\lef(t-1\rig)}t_{F}\lef(s'\lef(f\rig)\rig)f\rig)\rig)\text{.}
\end{multline*}

Of course, the same process works by taking $e\in G^{\underline{P}}\lef(t-1\rig)$ instead. So we get the theorem modulo $p$ by overconvergence in the canonical topology. Thus we can repeat this process, and conclude by overconvergence again, the unicity being given by theorem \ref{structuretheoremconvergent}.
\end{proof}

Notice that this theorem gives, as usual, a decomposition in three sub-$W\lef(k\rig)$-modules.

Before giving a first application of this theorem, recall that the association $\operatorname{Spec}\lef(A\rig)\to W^{\dagger}\Omega_{A/k}^{\bullet}$ on any affine scheme extends to a complex of Zariski sheaves $W^{\dagger}\Omega_{X/k}^{\bullet}$ on any variety $X$ smooth over a perfect field $k$ of characteristic $p$ \cite[theorem 1.7]{overconvergentderhamwittcohomology}.

This enables us to give a much shorter proof of \cite[theorem 1.8]{overconvergentderhamwittcohomology}.

\begin{coro}
Assume that $k$ is a perfect field of characteristic $p$. Let $X$ be a smooth $k$-variety. Then $W^{\dagger}\Omega_{X/k}^{\bullet}$ is a complex of étale sheaves on $X$.
\end{coro}

\begin{proof}
The question is Zariski local on $X$. So by a result of Kedlaya \cite[theorem 2]{moreetalecoversofaffinespacesinpositivecharacteristic}, we can assume that $X=\operatorname{Spec}\lef(\overline{R}\rig)$ is an affine scheme, finite étale over $\operatorname{Spec}\lef(k\lef[\underline{X}\rig]\rig)$. But $W\Omega_{X/k}$ is a complex of étale sheaves on $X$; see, for instance, \cite[proposition 1.11]{derhamwittcohomologyforaproperandsmoothmorphism}. So it is enough to prove that for any étale morphism $\varphi\colon Y\to\operatorname{Spec}\lef(\overline{R}\rig)$ and any $w\in W\Omega_{\overline{R}/k}^{\bullet}$, we have $w\in W^{\dagger}\Omega_{\overline{R}/k}^{\bullet}$ if and only if $\varphi^{*}\lef(w\rig)\in W^{\dagger}\Omega_{Y/k}^{\bullet}$. The implication is given by functoriality, so let us focus on the converse.

Let $\operatorname{Spec}\lef(\overline{S}\rig)$ be an open affine of $Y$. Combining \cite[03GU, 05DR and 02ML]{stacksproject}, we know that we can localise the étale morphism $\operatorname{Spec}\lef(\overline{S}\rig)\to\operatorname{Spec}\lef(k\lef[\underline{X}\rig]\rig)$ at some $P\in k\lef[\underline{X}\rig]$ to get a finite étale morphism. This morphism is also finitely presented and flat, so in particular $\overline{S}$ is a locally free $k\lef[\underline{X}\rig]$-module, so we can assume that $\overline{S}_{\lef\langle P\rig\rangle}$ is a free ${k\lef[\underline{X}\rig]}_{\lef\langle P\rig\rangle}$-module.

Apply theorem \ref{structuretheoremconvergent} to get a map:
\begin{equation*}
s\colon\begin{array}{rl}H\lef(t\rig)\sqcup G\lef(t\rig)\sqcup G\lef(t-1\rig)\to&\widehat{R}\\e\mapsto&s_{e}\end{array}
\end{equation*}
such that:
\begin{equation*}
x=\sum_{e\in H\lef(t\rig)}t_{\widehat{F}}\lef(s_{e}e\rig)+\sum_{e\in G\lef(t\rig)}t_{\widehat{F}}\lef(s_{e}\rig)e+d\lef(\sum_{e\in G\lef(t-1\rig)}t_{\widehat{F}}\lef(s_{e}\rig)e\rig)\text{.}
\end{equation*}

Now lift $\overline{S}_{\lef\langle P\rig\rangle}$ to a commutative smooth $W\lef(k\rig)$-algebra $S$. By lemma \ref{compatibleetalefrobenii}, we can lift the Frobenius morphism on $\overline{S}_{\lef\langle P\rig\rangle}$ to $G\colon S^{\dagger}\to S^{\dagger}$ and the étale map $\overline{R}\to\overline{S}_{\lef\langle P\rig\rangle}$ to a morphism $\phi\colon R^{\dagger}\to S^{\dagger}$ such that $G\circ\phi=\phi\circ F$.

Let $e\in H\lef(t\rig)\sqcup G\lef(t\rig)\sqcup G\lef(t-1\rig)$. Theorem \ref{maintheorem} implies that $\widehat{\phi}\lef(s_{e}\rig)\in S^{\dagger}$. By the same theorem, we are done if we show that this implies that $s_{e}\in R^{\dagger}$. But this is an easy consequence of \cite[proposition 2.16]{overconvergentwittvectors} and the theorem again.
\end{proof}

In a subsequent paper, we shall see how to use the structure theorem to give an interpretation of overconvergent $F$-isocrystals for the overconvergent de Rham--Witt complex.

\end{document}